\documentclass[12pt, twoside]{scrartcl}
\usepackage[automark, headsepline]{scrlayer-scrpage}

\usepackage{mathtools}  

\usepackage{amsmath}
\usepackage{amsthm}
\usepackage{amsxtra}
\usepackage{amssymb}

\usepackage{stmaryrd}

\usepackage{microtype}

\usepackage[utf8]{inputenc}

\usepackage[numbers]{natbib}

\usepackage[dvipsnames]{xcolor}

\usepackage{pdfpages}

\usepackage{url}

\usepackage[psdextra]{hyperref}
\hypersetup{
  linkcolor = NavyBlue!80!Black,
  citecolor = Plum,
  urlcolor = Cerulean!70!Black,
  colorlinks = true,
  pdfauthor={Eric Hofmann}
}

\newcommand{\wwedge}[1]{\sideset{}{^{#1}}\bigwedge}
\renewcommand{\Re}{\operatorname{Re}}

\newcommand*{\R}{\mathbb{R}}
\newcommand*{\Q}{\mathbb{Q}}
\newcommand*{\Z}{\mathbb{Z}}

\newcommand*{\C}{\mathbb{C}}
\newcommand*{\F}{\mathbb{F}}
\newcommand*{\W}{\mathbb{W}}

\newcommand*{\OF}{\mathcal{O}_{\F}}
\newcommand*{\DF}{\mathcal{D}_{\F}}

\newcommand*{\Ug}{\mathrm{U}}
\newcommand*{\Og}{\mathrm{O}}

\newcommand*{\SL}{\mathrm{SL}}
\newcommand*{\SU}{\mathrm{SU}}

\newcommand*{\frakg}{\mathfrak{g}}
\newcommand*{\fraku}{\mathfrak{u}}
\newcommand*{\frakp}{\mathfrak{p}}
\newcommand*{\frakk}{\mathfrak{k}}

\newcommand*{\h}{\mathbb{H}}
\newcommand*{\Hp}{\h}
\newcommand*{\Dom}{\mathbb{D}}

\newcommand*{\Hll}{\mathcal{H}_{\ell, \ell'}} 
\newcommand*{\myz}{\mathfrak{z}}

\DeclarePairedDelimiter{\hlfp}{(}{)}
\DeclarePairedDelimiter{\nvert}{\lvert}{\rvert}
\newcommand*{\hlf}[2]{\hlfp*{ #1, #2}} 
\newcommand*{\hlfempty}{\hlf{\cdot}{\cdot}}
\newcommand*{\blf}[2]{\hlf{#1}{#2}_{\R}}
\newcommand*{\blfempty}{\hlfempty_{\R}}
\newcommand*{\abs}[1]{\nvert*{#1}}

\DeclarePairedDelimiter{\sangl}{\langle}{\rangle}
\newcommand*{\sform}[2]{\sangl*{#1, #2}}

\newcommand*{\Qf}[1]{Q\left(#1 \right)}
\newcommand*{\dual}{\sharp}
\newcommand*{\ebase}{\mathfrak{e}}

\newcommand*{\Schw}{\mathcal{S}}
\newcommand*{\Schwartz}{\mathcal{S}}

\newcommand*{\Da}{\mathcal{D}}

\newcommand*{\psikm}{\psi}

\newcommand*{\MfL}[2]{\mathrm{M}_{#1,#2}}
\newcommand*{\MfLw}[2]{\mathrm{M}_{#1,#2}^!}
\newcommand*{\CfL}[2]{\mathrm{S}_{#1,#2}}

\newcommand*{\Mfw}[1]{\mathrm{M}_{#1}^!}

\newcommand*{\Hmfp}{\mathrm{H}^{+}}
\newcommand*{\HmfL}[2]{\mathrm{H}_{#1, #2}}
\newcommand*{\HmfLp}[2]{\mathrm{H}_{#1, #2}^+}

\DeclareMathOperator*{\CT}{CT}

\newcommand*{\psipq}{\psi_{p,q}}

\newcommand*{\BB}{\mathbf{B}}
\newcommand*{\CC}{\mathbf{C}}
\newcommand*{\AAp}[1]{\mathbf{A}_{#1}}
\newcommand*{\AAm}[1]{\mathrm{A}_{#1}}
\newcommand*{\AAb}[1]{\mathbb{A}_{#1}}

\newcommand*{\besspoly}{h}
\newcommand*{\Vint}[2]{\mathcal{V}_{#1}\left( #2\right)}

\numberwithin{equation}{section}

\theoremstyle{plain}
\newtheorem{theorem}{Theorem}[section]
\newtheorem{corollary}[theorem]{Corollary}
\newtheorem{lemma}[theorem]{Lemma}
\newtheorem{proposition}[theorem]{Proposition}
\newtheorem*{theorem*}{Theorem}
\newtheorem*{proposition*}{Proposition}
\newtheorem*{corollary*}{Corollary}

\theoremstyle{definition}
\newtheorem{definition}[theorem]{Definition}
\newtheorem{remark}[theorem]{Remark}
\newtheorem{example}[theorem]{Example}

\theoremstyle{remark}
\newtheorem{notation}[theorem]{Notation}
\newtheorem*{rmk*}{Remark}

\KOMAoptions{abstract=on, bibliography=totoc}

\title{The Fourier-Jacobi expansion of the singular theta lift}
\author{Eric Hofmann}
\begin{document}

\maketitle
\begin{abstract}
 We give an explicit evaluation of the Fourier-Jacobi expansions of the singular theta lift of Borcherds type constructed in \cite{FH21} for the dual reductive pair $\Ug(1,1)\times\Ug(p,q)$, where $p,q\geq 1$. The input functions are allowed to be arbitrary  harmonic weak Maass form of weight $k = 2 -p -q$. For this purpose, we adapt a method introduced by Kudla in \cite{K16}.
 Further, as an application, for the case $\Ug(p,1)$, we recover a new infinite product expansion for a Borcherds form, analogous to the case $\Og(p,2)$  treated in \cite{K16}.      
\end{abstract}
\setcounter{tocdepth}{\subsectiontocdepth}
\tableofcontents
\section{Introduction}
Recently, in \cite{FH21}, Jens Funke and the author employed singular theta lifts for the  dual reductive pair $\Ug(1,1) \times \Ug(p,q)$ to construct Green forms for codimension $q$ cycles on Shimura varieties associated to unitary groups. In particular, they constructed a singular theta lift of Borcherds type for weak harmonic Maass forms $f \in \Hmfp_{k}$ of weight $k = 2-q-p$ as input functions.
The intent of the present paper is to evaluate this lifting explicitly and to determine its Fourier-Jacobi expansion, for any signature $(p,q)$, $p,q\geq 1$, and any input function $f\in \Hmfp_{k}$. Note that, in general, there is no complete Fourier expansion, as the symmetric domain for $\Ug(p,q)$ is not a tube domain (unless $p=q$).

For this purpose, we might consider adapting the method introduced by Borcherds in \cite{Bo98}, which he used to calculate the Fourier expansion of his singular theta lift for indefinite orthogonal groups $\Og(p,q)$. This method, however, relies on Poisson summation and the successive reduction to lattices of lower dimension and reduced signature. For example, for $\Og(p,2)$, the case giving rise to the celebrated Borcherds products\footnote{Note that our signature convention differs from that used by Borcherds in \cite{Bo98}.}, the reduction goes from signature $(p,2)$ to the Lorentzian case $(p-1,1)$ and finally to definite lattices $(p-2,0)$.
These repeated reduction steps make Borcherds' recipe appear unfeasible here, as we strive for a general explicit evaluation for any signature $(p,q)$ with $p,q\geq 1$. 

Instead, we choose a  more representation theoretic approach. More precisely, we adapt a method introduced by Kudla in \cite{K16}, which starts with a decomposition of the theta integral along $\SL_2(\Z)$-orbits. 
For details of this procedure see Section \ref{sec:eval_Kstyle} below. 

Initially, we carry out the evaluation only at the base point of the symmetric domain, see Theorem \ref{thm:Phipq_atz0} for the results. (In fact, the actual evaluation of the unfolding integrals, which is rather technical, is carried out in a separate Section \ref{sec:unfolding}.)   

Then, we apply the operation of the Heisenberg group to obtain an explicit form of the Fourier-Jacobi expansion, see Section \ref{sec:FJ_exp}. In this context, it is quite natural to use to group operation as coordinates, viewing every point of the symmetric domain as the image of the base point under the operation of a suitable group element.

Finally, as an application, we examine the case of signature $(p,1)$ in more detail, see Section \ref{sec:p_one}. We recover an infinite product expansion for a Borcherds form, somewhat similar to that in \cite{K16}.  

For our approach it is advantageous to use the mixed model of the Weil representation. Hence,
in Section \ref{sec:psi_to_mm}, as a preparation for the main part, we determine the intertwiners from the more common Schrödinger model and apply them to the Schwartz form $\psi$ from \cite{FH21} and to the Fourier coefficients of the input function for the lift. In a sense, our approach here is somewhat similar to the way Murase and Sugano \citep[see][]{MuSu2} use the mixed model in their evaluation of the Fourier-Jacobi expansion of the unitary Kudla lift for the reductive pair $\Ug(1,1)\times \Ug(2,1)$, the original reference for which is   \citep[][]{K79}. 

A number of problems remain open for future work. For example, 
it should be interesting to study the behavior of the lift on the boundary of $\Dom$ in terms of the Fourier-Jacobi expansion. This seems particularly feasible in signature $(1,q)$ as the compactification theory is essentially the same (via complementary
 spaces) as that for $(p,1)$, which is well-known \citep[see][]{Hof17}, or \citep[cf.][]{How13} for more details .
 
 Also, we should point out that for $p,q > 1$ we determine the Fourier-Jacobi expansion only up to a constant (see Remark \ref{rmk:no_rkzero_term}). To explicitly calculate this final contribution is thus also left open for the time being.

 Most of the results in this paper\footnote{with the notable exception of the product expansions in Section \ref{sec:p_one}.} are contained in the author's Habilitationsschrift \cite{H21}, mainly in Chapter 4  and Appendix A.2.1.

 \subsection{Overview of results}

The main goal of the paper is to calculate an explicit form of the Fourier Jacobi expansion for the  singular theta lift $\Phi(z,f, \psi)$ for a harmonic Maass form $f \in Hmfp_k$ with weight $k=2-p-q$.
We work in the mixed model of the Weil representation, for the setup of which see Section \ref{sec:psi_to_mm}.

Let $V$ be a complex hermitian space of dimension $m$ and signature $(p,q)$ with $p,q\geq 1$ and  $p + q = m$. Let $G = \Ug(V)$ be the unitary group of $V$. We can view the associated hermitian domain as $\Dom$ as the Grassmannian of negative definite $q$-planes in $V$. 

The Schwartz form $\psi$ from \cite{FH21} is contained in the space of Schwartz form over $V$ valued in the closed differential $(q-1, q-1)$-forms, 
\[
\psi \in \left[ \Schwartz(V) \otimes \mathcal{A}^{q-1,q-1} \right]. 
\]
Let $L$ be an even lattice in $V$ and assume for the purposes of this introduction that $L$ is unimodular. We take $G_L$ to be a finite index subgroup of the stabilizer of $L$ in $G$. 

To formulate the intertwiners for the mixed model we the use the following coordinates: 
Let $\ell$, $\ell'$ be isotropic lattice vectors with $\hlf{\ell}{\ell'} =1$. Denote $W = V\cap\ell^\perp\cap\ell'^\perp$. For a vector $x \in V$ write $x$ in the form $x = \alpha\ell + x_0 + \beta\ell'$, with $x_0 \in W$. 
Now, in passing to the mixed model, one has to carry out a partial Fourier transform in the variable $\alpha$. We denote the new variable by $\beta'$ and set $\eta = [\beta, \beta']$. For a lattice vector $x$, we can consider $\eta$ as a rational $2\times 2$-matrix.

Now, the main idea from \cite{K16} for the  evaluation of the regularized integral is that due to invariance under the operation of $\SL_2(\Z)$  the regularized integral can be decomposed by systems of representatives of $\SL_2(\Z)$-orbits of  $\eta$ (viewed as a rational $2\times 2$-matrix), which are further differentiated according to their rank. Thus, we set 
\[
\begin{gathered}
\Phi(z, f, \psi) = \Phi_0(z,f,\psi) +  2\cdot \sum_{i=1}^2 \Phi_i (z, f, \psi),\quad\text{with} \\
\Phi_i(z, f, \psi) \vcentcolon = 
\int_{ \SL_2(\Z)\backslash \mathbb{H}}^{reg}
\sum_{\substack{ \eta =  [\beta, \beta'] / \sim\\  \operatorname{rank}(\eta) = i}}
\sum_{\gamma\in\SL_2(\Z)_\eta\backslash \SL_2(\Z)} \left\langle 
f(\gamma\tau), \overline {\theta_{\eta}(\gamma\tau, z)} \right\rangle_L 
v^{-2} du\, dv.
\end{gathered}
\]
(The factor $2$ for the terms with non-vanishing $\eta$ appears here, because $-1_2$ operates trivially, hence the contributions of $\gamma$ and $-\gamma$ are the same.)
Moreover, due to the rapid decay of the integrand, the integrals can be evaluated for each term separately, with fixed $\eta$, and summed up later $(i=0,1,2)$:
\[
\Phi_i(z, f, \psi) =\vcentcolon \sum_{\substack{ \eta =  [\beta, \beta'] / \sim\\  \operatorname{rank}(\eta) = i}} \sum_{\gamma \in \SL_2(\Z)_\eta\backslash \SL_2(\Z)} \sum_{n \gg -\infty}
\left({a}^+(n) \phi_i(n,\eta)^+ + {a}^{-}(n)\phi_i(n,\eta)^{-}\right).
\]
Here, ${a}^+(n)$ and ${a}^-(n)$ denote the Fourier coefficients of the holomorphic part and non-holomorphic part of $f$, respectively\footnote{Strictly speaking, we need to consider the Fourier coefficients in the mixed model, in the main text they will be denoted by $\hat{a}^\pm(n)$.  Since in the introduction we assume $L$ to be unimodular, the distinction is unnecessary here.}. The terms  $\phi_i(n,\eta)^{\pm}$ are given by
\begin{align*}
\phi_i(n,\eta)^+ & = \int_{\SL_2(\Z)_\eta\backslash \Hp}^{reg} e^{2\pi i n u} \theta_\eta(\tau, z) \, d\mu(\tau), \\ 
\phi_i(n,\eta)^-& = \int_{\SL_2(\Z)_\eta \backslash \Hp}^{reg} \Gamma(1-k, 4\pi\abs{n})e^{2\pi i n u} \theta_{\eta}(\tau, z)\,  d\mu(\tau) 
\qquad (n\neq 0), \\
\phi_i(0, \eta)^- & = \int_{\SL_2(\Z)_\eta\backslash \Hp}^{reg} \theta_{\eta}(\tau, z) v^{1-k}\, d\mu(\tau).
\end{align*}
For $i=0$ the domain of integration  is given by a usual fundamental domain $\SL_2(\Z)\backslash \Hp$. For $i = 1$, it is given by $\SL_2(\Z)_\infty \backslash \Hp$. Finally for $i=2$ we integrate over the whole upper half plane. 

We remark that we do not determine the \lq rank $0$\rq\ (i.e.~ $i=0$) term in general, which is given by the convolution integral of an indefinite theta series and contributes only an additive constant to the lift. However, this term falls out in signature $(1,q)$ (an example we treat prominently, see Example \ref{ex:sig_q1} and Corollary \ref{cor:FJ_oneq}). In signature $(p,1)$, where the theta series is definite, it can be worked out using the methods of Borcherds \cite{Bo98}, see below.

To facilitate the calculation somewhat we first evaluate all terms at the base point $z_0 \in \Dom$, the results can be found in Theorem \ref{thm:Phipq_atz0}. Then, we apply the group action of $G$ to obtain the Fourier-Jacobi expansion. Consider the Levi decomposition $G = NAM$, with $M\simeq\SU(p-1,q-1)$,  $A\simeq \mathrm{GL}([\ell])$  and $N$ a Heisenberg group. We use elements of these subgroups as coordinates and give the Fourier-Jacobi expansion of the lift in these terms, see Theorem \ref{thm:FJexp_pq}.

Since the notation is quite involved, rather than reproducing both main Theorems in full generality here, we state their results in two  example cases,  signature $(p,1)$ and signature $(1,q)$.  In the latter case, for simplification we assume that the input function $f$ is a weakly holomorphic modular form, i.e.\ $f\in \Mfw{k}$. Also recall our assumption that $L$ is unimodular.

\paragraph{Signature $(p,1)$} 
 If the signature of $V$ is $(p,1)$, the Schwartz form $\psi$ is essentially the Gaussian, indeed $\psi = 2i\varphi_0$. In the mixed model and using the coordinates introduced above, it takes the form  
	\[
	\widehat\psi_{p,1}(\sqrt{2}(\eta, x_0), \tau) = 
	2i \exp\left(-\frac{2\pi}{v}\left[  \abs{\beta' + \bar\tau\beta }^2 + 2v 
	\Im\left(\beta' \bar\beta\right) \right]  \right) e^{2\pi i \tau 
		\hlf{x_0}{x_0}} \otimes 1. 
	\]
Thus, the resulting contributions to the lift of $f \in \mathrm{H}^+_{1-p}$ take a fairly simple form. For example, for fixed $\eta$ and $n$ one has
	\[
	\begin{aligned}
	\phi_2(n,\eta)^+(z_0)  
	& = \frac{\sqrt{2}}{2}\abs{\beta} {\BB_\eta}^{-\frac12}
	\exp\left(
	- \frac{2\pi}{\abs{\beta}^2} \abs{\AAp{\eta}}
	\left( \tfrac{1}{2} \BB_\eta \right)^{\frac12} \right)
	e^{-2\pi  i \CC_\eta\left( \hlf{x_0}{x_0} + n\right)},
	\end{aligned}
	\]
where $\AAp{\eta}, \BB_\eta$ and $\CC$ are defined as
\[
  \begin{gathered}
    \AAp{\eta}(n , x_0) =  
n  + 2\abs{\beta}^2 + \hlf{x_0}{x_0}, \quad 
 \CC_\eta =  \frac{\Re(\beta'\bar\beta)}{\abs{\beta}^2} \quad\text{and}\quad
\BB_\eta = 2\Im\left(\beta'\bar\beta\right)^2. 
\end{gathered}
\]
The other contributions $\phi_i^\pm(n,\eta)^\pm$, $i=1,2$ for this case can be found in Example \ref{ex:sig_p1}.

The Fourier-Jacobi expansion in this signature  takes the form (see Corollary \ref{cor:FJ_exp_pone}):
\[
\frac{1}{2i}\Phi(z,f, \psi) = c_0'(t,w) + \sum_{\kappa \in \Q^\times} c_\kappa'(t,w) e^{2\pi i \kappa \Re\tau_\ell},
\]
where the constant term $c_0(t,w)$ is given by
\[
\begin{aligned}
c_0'(t,w)  =  & \; 4\pi I_0 +  t^2 \sum_{\beta' = (a,b)}
  \biggl[
a^-(0)  \frac{1}{\left( 2\pi t^2 \abs{\beta'}\right)^{p+1}} \Gamma(p+1) \;+ \\
& \;  \sum_{n\neq 0} \Bigl( a^+(n)  \frac{1}{2 \pi t^2\abs{\beta'}^2}
+      a^-(n) (p-1)! \sum_{r = 0}^{p-1} \frac{\pi^r}{r!} (4\abs{n})^{\frac{r}{2} - \frac{1}{4}} t^{r - \frac12} \abs{\beta'}^{r-\frac12}
\\ & \qquad \cdot h_r\left( \frac{1}{4\pi t \abs{\beta'}\abs{n}^{\frac12}}  \right) e^{-4\pi t\abs{\beta'}\abs{n}^\frac12}\Bigr) \biggr]  
\cdot e\Bigl(  - \Re\left(\beta' \hlf{x_0}{w}\right) \Bigr).
\end{aligned}
\]
The constant $I_0$ occurring here is a rational number,  which can be evaluated using the methods of \cite{Bo98}, see \cite{K16} for details.
The coefficients $c_\kappa(t,w)$ $(\kappa>0)$ take the following form (note that since these coefficients come from the integrals $\Phi_i(z, f, \psi)$ with $i\neq 0$, a factor $2$ occurs here, see above)
\[
c_\kappa'(t,w)  = 2\sum_{a,b}\sum_m A_\kappa(n, [\beta, \beta'])(t,w),
\]
wherein 
\[
\begin{multlined}
A_\kappa(n, [\beta, \beta'])(t,w) = \left( a^+( n) \frac{\sqrt{2}}{2} 
\frac{ t\abs{\beta}}{\BB_\eta^{\frac12}} + a^-(n) A^-_n(\eta; t, w) \right) \\
\cdot
\exp\left( - 2\frac{\pi}{\abs{\beta}^2}
\abs{\AAp{t\eta}(x_0 - \beta w ) }\left(\tfrac12\BB_\eta\right)^{\frac12}
-  2\pi i\bigl[  \CC_{\eta}\left(\hlf{x_0}{x_0} + n\right) + 2\alpha\Im\hlf{x_0}{w} \bigr]\right),
\end{multlined}
\]
with a non-holomorphic term $A^-_n(\eta; t,w)$, given by
\[
\begin{multlined}
\sqrt{2}
(p-1)!\sum_{r=0}^{p-1}
\frac{(4\abs{n}\pi)^r}{r!} t^{2r + 1 + 2} \abs{\beta} \BB_{ \eta }^{\frac{r-1}2 } 
\left(
\tfrac12 \AAp{t\eta}^2(x_0 - \beta w) -  2 t^2 \abs{\beta}^2 \left(2\abs{n} - n\right) \right)^{-\frac{r}{2}} \\
\cdot
h_{\max\{0,r-1\}}\left(
\frac{\abs{\beta}^2}{2\pi} 
\left(  \left( \tfrac12\AAp{t\eta}(x_0 - \beta w) + 2t^2\abs{\beta}^2( 2\abs{n} - n) 
\right)   \BB_\eta
\right)^{-\frac12}   
\right).          
\end{multlined} 
\]
We remark that in this signature, the Fourier-Jacobi expansion can be used to obtain a form of Borcherds product expansion, distinct from that in \cite{Hof11, Hof14}, see Section \ref{sec:p_one} below. 

\paragraph{Signature $(1,q)$} As a second example, we consider the case of signature $(1,q)$. For simplicity, let $f$ be a weakly holomorphic modular form, so there is no non-holomorphic part. Denote by $a(n)$ be the Fourier coefficients of $f$ in the mixed model.

Then,
\[
\Phi(z, f, \psi_{1,q}) =2\cdot \frac{2i(-1)^{q-1}}{2^{2(q-1)}} 
\Bigl( \sum_{i=1}^2 \sum_{\substack{ \eta =  [\beta, \beta'] / \sim\\  \operatorname{rank}(\eta) = i}}
\sum_{n\gg -\infty}{a}(n)  \phi_i(n,\eta)^+   \Bigr)\otimes \Omega_{q-1}\bigl(\underline{1}; \underline{1}\bigr).
\]
See Subsection \ref{subsec:schwartz_forms} concerning the notation of the differential form. 
The non-singular terms are given as follows. For two positive integers $h, j$ we define indices $\nu$, $\nu'$ by setting $\nu = h + j - q - \frac12$ and   $\nu' =  \abs{\nu} - \frac12$. Then, we have
\[
\begin{multlined}
  \phi_{2}(n,\eta)^{+} (z_0)   = 
  (-\pi)^{q-1}\frac{ \sqrt{2}}{2} \sum_{M=0}^{2(q-1)}  R_{q-1,q-1}(\eta; M) \sum_{j=0}^{\lfloor \frac{M}{2} \rfloor} \frac{1}{\pi^j} \abs{\beta}^{-2(M-j)} \\
   \cdot \sum_{h = 0}^{M-2j} \frac{i^h}{2^{h +3j}}\AAm{\eta}^h ( - \CC_\eta)^{M-2j-h}  \abs{\beta}^{2h - 2j +1} \binom{M-2j}{h} \frac{M!}{j! (M-2j)!} \\   \cdot
  \exp\left(- \frac{2\pi}{\abs{\beta}^2} \left( \AAb{\eta}\BB_\eta\right)^{\frac12} \right)
e\left( - \CC_\eta\left( \hlf{x_0}{x_0} + n \right) \right) 
\cdot
   \AAb{\eta}^{ -\frac{\nu}{2} - \frac14} \BB_\eta^{\frac{\nu}{2} - \frac14} 
h_{\nu'} \left( \frac{\abs{\beta}^2}{2\pi \left( \AAb{\eta}\BB_\eta\right)^{\frac12}}  \right),  
\end{multlined}
\]
with
\[
R_{q-1,q-1}(\eta; M) = \sum_{\substack{0\leq \mu_1,\mu_2 \leq q-1 \\ \mu_1 + \mu_2 = M}}
\abs{\beta}^M \abs{\beta'}^{q-1} \Re\left(\beta'^{-\mu_1} \bar\beta'^{-\mu_2} \right) \binom{q-1}{\mu_1} \binom{q-1}{\mu_2}.
\]
For the rank 1 terms one has
\[
  \begin{gathered}
\phi_1(n,\eta)^+(z_0) = (-\pi)^{q-1} \abs{\beta'}^{q-2} \left(2 \abs{n} \right)^{\frac{q}{2}} 
K_q\left( 2\sqrt{2}\pi \abs{\beta'}\abs{n}^\frac12 \right) \quad(n\neq 0), \\
\phi_1(0,\eta)^+(z_0) = 
 (-\pi)^{q-1} \abs{\beta'}^{q-2} \left(2 \abs{\Qf{x_{0,-}}} \right)^{\frac{q}{2}} 
 K_q\left( 2\sqrt{2}\pi \abs{\beta'}\abs{\Qf{x_{0,-}} }^\frac12 \right).
\end{gathered}
\]
The Fourier-Jacobi expansion of $\Phi(z, f, \psi_{1,q})$ is given by (see Corollary \ref{cor:FJ_oneq} below)  
  \[
      \frac{2^{2(q-1)}}{2i(-1)^{q-1}}  \Phi(z, f, \psi_{1,q}) =
      \left( c_0'(t,w) +
        \sum_{\kappa \in \Q^\times} c_\kappa'(t,w) e^{2\pi i \kappa r} \right)
      \otimes \Omega_{q-1;q-1}(\underline{1}, \underline{1}), \\
    \]  
    with  
    \[\begin{aligned} 
        c_0'(t,w) &= 2\cdot\sum_{ab}\sum_{n}
        B^{\underline{1}, \underline{1}}
          \left(n, \left( \begin{smallmatrix}a \\ b 
              \end{smallmatrix}
              \right) \right)(t,w)
         \quad\text{and}\\
          c_\kappa'(t,w) & =
          2\cdot\sum_{\substack{a,\alpha \\ a\alpha = \kappa }} \sum_b \sum_m
          A^{\underline{1}, \underline{1}}_\kappa\left(n, \left(\begin{smallmatrix} a & b \\ & \alpha\end{smallmatrix}\right) \right)(t,w) \quad(\kappa\neq 0). 
    \end{aligned}
  \]
 Herein, the  coefficients $B^{\underline{1}, \underline{1}}(n, \left(\begin{smallmatrix} a \\ b 
 \end{smallmatrix}\right))(t,w,\mu)$ are given by
 \begin{equation*}
 (- \pi)^{q-1}  
 a^+(n) t^{q}\abs{\beta'}^{q-2}\left(2 \abs{n}\right)^{\frac{q}{2}} K_{q}\left(2 \sqrt{2} \pi \abs{\beta'} \abs{n}^{\frac12} \right)   
 e\left( - \Re\left(\beta' \hlf{x_0}{w} \right)\right) 
\end{equation*} 
for $n\neq 0$ and by 
\begin{equation*}
(- \pi)^{q-1} 2  t^{q} \abs{\beta'}^{q-2} e\left( - \Re\left(\beta' \hlf{x_0}{w} \right)\right) 
 a^+(0) \abs{\Qf{x_{0,-}}} ^{\frac{q}{2}}
K_{q}\left(2^{\frac32}   \pi \abs{\beta'} \abs{\Qf{x_{0,-}}}^{\frac12} \right)
\end{equation*} 
for $n=0$. 
The coefficients $A^{\underline{1}, \underline{1}}_\kappa$ are given as follows (with $\eta = \left(\begin{smallmatrix} a & b \\ & \alpha\end{smallmatrix}\right)$):
 \begin{multline*} 
 A^{\underline{1}, \underline{1}}_{\kappa}\left(n,\eta \right)  = (-\pi)^{q-1}  t^{2(q-1)} 
\frac{\sqrt{2}}{2}\,
\sum_{M=0}^{2(q-1)}  R_{q-1, q-1}(\eta,M)
\\  \;\cdot\, \sum_{j=0}^{\left\lfloor \frac{M}{2} \right\rfloor} \frac{1}{\pi^j} \sum_{h = 0}^{M -2j}
\frac{i^h}{2^{h + 3j}}  
\AAm{t\eta}^h(x_0 - \beta w) \left( - \CC_\eta\right)^{M - 2j -h}  t^{2h-2j}\abs{\beta}^{- 2h - 2 j}
\binom{M -2 j}{h} \frac{M!}{j! (M-2j)!}  \notag
\\
\cdot  a^+(n) A^+_n( \eta; t,w)  \cdot e\Bigl( - \CC_\eta\left(\hlf{x_0}{x_0}  + n\right) + 2\alpha\Im\hlf{x_0}{w} \Bigr),
\end{multline*}
where $A^+_n$ is given by  
\begin{multline*} 
A^+_n( \eta; t,w) =  t^{2\nu-1}\left({\AAb{t\eta}(x_0 - \beta w)}\right)^{-\frac{\nu}{2} - \frac14}
 {\BB_\eta}^{\frac{\nu}{2}-\frac14}
 \exp\left(-\frac{2\pi}{\abs{\beta}^2} \left( \AAb{t\eta}(x_0 - \beta w) \BB_\eta\right)^{\frac12}\right) \\
\cdot h_{\nu'} \left( \frac{\abs{\beta}^2 }{ 2\pi  \left( \AAb{t\eta}(x_0 - \beta w) \BB_\eta\right)^{\frac12} }\right), 
\end{multline*}
with indices $\nu = h+j+1-q$ and $\nu' = \abs{\nu} - \frac12  = \operatorname{max}(q-1-h-j, h+j-q)$.

\bigskip
\paragraph{Acknowledgments}  The initial stages of this work were carried out during the author's stay\footnote{Supported by a research fellowship (Forschungsstipendium) of the DFG.} at the Department of Mathematical Sciences at Durham University during the academic year 2017/18 as a spin-off from the work on \cite{FH21}, and I would like to thank the Department for its hospitality during my stay.
Further, I want to thank Jens Funke, my coauthor in \cite{FH21}, for his initial contribution, as both the idea of adapting Kudla's method from \cite{K16} and that of using the group action as coordinates were due to him. I would also like to thank him for a number of helpful discussions and encouragement during later stages of this work. 
I would further like to thank Winfried Kohnen, Steven Kudla and Rainer Weissauer for refereeing my Habilitationsschrift, which already contained most of the results in this paper. 

\section{Setup and notation} \label{sec:setup}

\subsection{A hermitian space} 
Let $V$ be a complex vector space of dimension $m$, equipped with a non-degenerate indefinite hermitian form $\hlfempty$ of signature $(p,q)$ with $p,q\geq 1$ and $p+q=m$. We will assume that $\hlfempty$ is $\C$-linear in its second argument and conjugate-linear in the first argument.

Pick standard orthogonal basis elements\footnote{In this paper, following Kudla and Millson \cite{KM90} we use 'early' Greek letters for indices ranging from $1$ to $p$ and 'late' Greek letters for indices ranging from $p+1$ to $p+q$.} $v_1,  \dotsc, v_m$ with $\hlf{v_\alpha}{v_\alpha} = 1$ for $\alpha =1, \dotsc, p$ and $\hlf{v_\mu}{v_\mu} = -1$ for $\mu = p+1, \dotsc, m$. Let $z_\alpha$ and $z_\mu$ denote the corresponding coordinate functions. Hence for $x\in v$, one has
\[
x = \sum_\alpha  z_\alpha v_\alpha + \sum_\mu z_\mu v_\mu \quad \text{and} \quad 
\hlf{x}{x} = \sum_{\alpha} \abs{z_\alpha}^2 + \sum_{\mu} \abs{z_\mu}^2.
\]
Further, we may consider $V$ as a real quadratic space with the bilinear form $\blfempty \vcentcolon= \Re\hlfempty$. Denote this quadratic space by $V_\R$.

Denote by $G \simeq \Ug(p,q)$ the unitary group of $V$ and let $\Dom = G/K$, with $K \simeq \Ug(p)\times\Ug(q)$, be the symmetric domain for the operation of $G$ on $V$. We realize $\Dom$  as the Grassmannian of negative definite $q$-planes in $V$,
\[
\Dom \simeq \left\{  z \subset V\,;\, \operatorname{dim}(z) = q, \hlfempty\mid_z < 0\right\},
\]
and fix $z_0 = \operatorname{span}_\C \left\{z_\mu ;\,  \mu = p+1, \dotsc, m \right\}$ as the base-point.

Given $z \in \Dom$, for $x\in V$ denote by $x_z$ and $x_{z^\perp}$ the orthogonal projections onto $z$ and its orthogonal complement $z^\perp$, respectively. Then, the associated standard majorant $\hlf{x}{x}_z$ is given by
\[
\hlf{x}{x}_z = \hlf{x_{z^\perp}}{x_{z^\perp}} - \hlf{x_z}{x_z}.
\]

\paragraph{Hermitian lattices} 
Let $\F = \Q\left( \sqrt{\DF}\right)$ be an imaginary quadratic number field with discriminant $\DF$ for which we fix an embedding of $\F$ into $\C$.
We denote by $\OF$ the ring of integers and let $\delta_\F$ be the square root of the discriminant, with the principal branch of the complex square root function. 
Further, denote by $\DF^{-1}$ the inverse different ideal, given by $\delta_\F^{-1}\OF$. 
Finally, we denote by $\kappa_\F$ a generator of $\OF$ as a $\Z$-module,  chosen such that $\Im \kappa_\F = \frac12\delta_\F$.

Let $L\subset V$ be an even hermitian lattice, i.e.\ a projective module over the ring of integers $\OF$ in $\F$, on which the restriction of $\hlfempty$ is $\OF$-valued. Let $L$ be of full rank, i.e.~ $V = L \otimes_{\OF}\C$. 
The dual lattice $L^\dual$ is given by
 \[
	L^\dual = \left\{ x\in V \,;\, \hlf{x}{\lambda} \in \DF^{-1}, \forall \lambda \in L \right\}
	= \left\{ x\in V\,;\, \operatorname{trace}_{\F/\Q}\hlf{x}{\lambda} \in \Z, \forall \lambda \in L  \right\}.  
 \]
 The quotient $L^\dual/L$ is called the discriminant group of $L$.
Note that since $L$ is even, $L\subset L^\dual$, hence $L$ is integral. Also, evenness implies that $\hlf{\lambda}{\lambda}\in \Z$. 
Further, given an even hermitian lattice $L$ we denote by $L^-$ the same $\OF$-module as $L$ but with the hermitian form $-\hlfempty$.

Finally, let $G_L$ be a unitary modular group, i.e.\ a finite index subgroup of the  stabilizer $\operatorname{Stab}_G(L^\dual/L)$ of the discriminant group of $L$ in $G$. 

\subsection{Spaces of vector valued modular forms}
We denote the standard basis elements of  the group algebra $\C[L^\dual/L]$ by $\ebase_h$ $(h\in L^\dual/L)$ and introduce a hermitian pairing 
\[
\sform{}{}_L\,: \C[L^\dual/L] \times \C[L^\dual/L] \longrightarrow \C
\]
by setting $\sform{\ebase_h}{\ebase_g}_L = \delta_{h,g}$ for $h, g \in L^\dual/L$. Similarly for $L^{-}$.

Recall the finite Weil representation for the operation of $\SL_2(\R)$ on $\C[L^\dual/L]$, which we denote by $\rho_L$. It is most easily described through the operation of the generators $T= \left( \begin{smallmatrix} 1 & 1 \\ 0 & 1 \end{smallmatrix} \right)$ and  $S= \left( \begin{smallmatrix} 0 & -1 \\ 1 & \phantom{-}0 \end{smallmatrix} \right)$:
\[
\rho_L(T) \ebase_h = e\left(\hlf{h}{h} \right) \ebase_h,\quad
\rho_L(S)  \ebase_h = \frac{i^{p-q}}{\sqrt{\abs{L^\dual/L}}} \sum_{g \in L^ \dual/L} e\left( -2\blf{g}{h}\right) \ebase_g.
\]
We denote by $\rho^\vee_L \simeq \bar\rho_{L}$ the dual representation. Note that $\rho^\vee_L \simeq \rho_{L^{-}}$.

For $k\in\Z$ and $\gamma\in\SL_2(\R)$ define the usual $k$-slash-operation on functions $\C[L^\dual/L] \rightarrow \C$ as
\[
f\mid_{k,L} \gamma = \left( c\tau  + d\right)^{-1} \rho_L(\gamma)^{-1} f\left(\gamma\tau\right)\qquad \text{if}\quad
  \gamma = \begin{pmatrix}a & b \\ c & d \end{pmatrix}.
\] 
The slash-operation for the dual representation $\rho_L^\vee \simeq \rho_{L^{-}}$ is defined similarly.

We recall some definitions for spaces of vector-valued modular forms transforming under $\rho_L$. 

\begin{definition}
 For $k\in\Z$, let $\CfL{k}{L}$, $\MfL{k}{L}$ and $\MfLw{k}{L}$ be the spaces of holomorphic functions $f\,:\, \Hp\rightarrow\C[L^\dual/L]$ which satisfy 
 \begin{enumerate}
  \item  $f \mid_{k,L} \gamma = f$ for all $\gamma\in\SL_2(\Z)$.
  \item If $f$ in $\MfLw{k}{L}$, then $f$ is meromorphic at the cusp $\infty$, while if $f \in \MfL{k}{L}$, it is holomorphic at the cusp. Finally, if $f\in \CfL{k}{L}$ it vanishes at $\infty$.
  Clearly $\CfL{k}{L} \subset \MfL{k}{L} \subset \MfLw{k}{L}$. The elements of these spaces are called cusps forms, (holomorphic) modular forms and weakly holomorphic modular forms, respectively. 
 \end{enumerate}
 \end{definition}

 Next, we introduce harmonic weak Maass forms following \cite{BrF04}. 
 \begin{definition}
  For $k\in\Z$, let $\HmfL{k}{L}$ be the space of twice continuously differentiable functions $f\,:\, \C\rightarrow \C[L^\dual/L]$ which satisfy
  \begin{enumerate}
   \item $f\mid_{k,L} \gamma = f$ for all $\gamma\in\SL_2(\R)$.
   \item There exists a constant $C>0$ such that $f(\tau) = O\left(e^{C v}\right)$ as $v\rightarrow\infty$.
   \item $f$ is harmonic, i.e. $\Delta_k f  = 0$.
  \end{enumerate}
 \end{definition}
 
 The space $\HmfLp{k}{L}$ is defined as the inverse image of the cusp forms $\CfL{2-k}{L^{-}}$ under the $\xi$-operator 
\[
 \xi_k : \HmfL{k}{L}\rightarrow \MfLw{2-k}{L^{-}},\qquad  f(\tau)\mapsto 2iv^k \frac{\overline{\partial f(\tau)}}{\partial \bar\tau}.
\]
 
 Recall that any harmonic weak Mass form has a decomposition $f(\tau) = f^+(\tau) +  f^-(\tau)$ into its holomorphic and its non-holomorphic part, respectively. 
The Fourier expansions of the two parts take the forms
\[
\begin{gathered}
f^+(\tau) = \sum_{h\in L^\dual/L} \sum_{n\in\Q}  a^+(h,n) e(n\tau) \ebase_h, \\
f^-(\tau) = \sum_{h\in L^\dual/L} \Bigl( a^-(h,0) v^{1-k} + \sum_{\substack{n \in \Q\\ n\neq 0}} a^-(h,n)   \Gamma(1-k, 4\pi nv) e(nu) \Bigr) \ebase_h.
\end{gathered}
\]
Finally, the principal part of $f$, denoted $P(f)$, is the Fourier polynomial 
\[
P(f)(\tau) = P(f^ +)(\tau) = \sum_{h \in L^ \dual/L}\sum_{\substack{n \in \Q\\ n<0}} a^+(h,n) e(n\tau) \ebase_h.
\]
It follows from these definitions that for $f \in \HmfLp{k}{L}$,
\[
  f(\tau) - P(f)(\tau) = O\left( e^{-Cv}\right) \qquad\text{as}\; v\rightarrow \infty 
\]
with some constant $C>0$. Further, there are exact sequences, see \citep[][Corollary 3.8]{BrF04} 
\begin{equation} \label{eq:exSeq}
  \begin{gathered}
    0 \longrightarrow \MfLw{k}{L} \longrightarrow \HmfL{k}{L} \stackrel{\xi_k}{\longrightarrow}  \MfLw{2-k}{L^-} \longrightarrow 0,  \\
0 \longrightarrow \MfLw{k}{L} \longrightarrow \HmfLp{k}{L} \stackrel{\xi_k}{\longrightarrow}  \CfL{2-k}{L^-} \longrightarrow 0.
\end{gathered}
\end{equation}
In particular, weakly holomorphic modular forms can be considered as harmonic Maass forms with trivial holomorphic part.

\subsection{Schwartz forms}\label{subsec:schwartz_forms}

\paragraph{The unitary Lie algebra} We recall some notation from \citep[][Sec.~2.2]{FH21}, see loc.~cit. for details.
Let $\frakg_0 = \fraku(V)$ be the Lie algebra of $G$. 
Denote by $\frakg = \frakg_0\otimes\C$ the complexification of the $\frakg_0$, which we view as a right vector space over $\C$. 

The Cartan decomposition $\frakg_0 = \frakk_0\oplus\frakp_0$ with $\frakk  = \operatorname{Lie}(K)= \fraku(p) \times \fraku(q)$ gives  
rise to a similar decomposition $\frakg = \frakk \oplus \frakp$, with complexified factors $\frakk = \frakk_0 \otimes \C$ and  $\frakp = \frakp_0 \otimes\C$. We denote the dual of $\frakp$ by $\frakp^*$.

Further, we have the Harish-Chandra decomposition pf $\frakg$, 
\[
\frakg= \frakk\oplus \frakp^+ \oplus \frakp^{-}. 
\]
We denote by $\left\{ \xi'_{\alpha\mu} \right\}$ and $\left\{ \xi''_{\alpha\mu} \right\}$ the dual basis elements of $\frakp^+$ and $\frakp^-$.

\paragraph{The special Schwartz form  \texorpdfstring{$\psi$}{psi}}

Denote by $\Schwartz(V)$ the Schwartz space of $V$. 
We briefly recall the definition of the special Schwartz form $\psi$ constructed in \cite{FH21} using the Schrödinger model of the Weil representation.
In the following, to emphasize the dependence on the signature of $V$, we will often denote Schwartz forms $\phi$  in $[\Schwartz(V) \otimes \mathcal{A}^\bullet(\Dom)]^G$ 
by $\phi_{p,q}$.

First recall that evaluation at the base point $z_0 \in \Dom$ yields an isomorphism
\[
   \bigl[\Schwartz(V) \otimes \mathcal{A}^\bullet(\Dom)\bigr]^G \simeq
\bigl[  \Schw(V) \otimes \wwedge{\bullet}(\frakp^*)\bigr]^K. 
\]
We denote by $\varphi_0^{p,q}$ the standard Gaussian,
\[
\varphi_0^{p,q} =  \varphi_0(x,z)  = e^{ -\pi\hlf{x}{x}_z } \in  \bigl[\Schwartz(V) \otimes \mathcal{A}^0(\Dom)\bigr]^G. 
\]
Evaluation at the base point  $z_0$ gives $\varphi_0(x) = \varphi_0(x,z_0) = e^{-\pi \sum_{i=1}^m \abs{z_0}^2 }\in \Schwartz(V)^K$.
\bigskip

Now, for the definition of the Schwartz form $\psi = \psipq$ 
 from \citep[][Section 3.3]{FH21}:
For indices $\gamma, \delta \in \{1,\dotsc, p\}$ we set 
\[
  \Da_\gamma \vcentcolon = \left( \bar{z}_\gamma -
  \frac{1}{\pi}\frac{\partial}{\partial z_\gamma}\right)\quad\text{and}\quad 
\overline{\Da}_\delta \vcentcolon =  \left( z_\delta - \frac{1}{\pi}\tfrac{\partial}{\partial \bar{z}_\delta}  \right),
\]
wherein
$\tfrac{\partial}{\partial z_{\gamma}} \vcentcolon = \tfrac12 \left(
  \frac{\partial}{\partial x_{\gamma}} -i \tfrac{\partial}{\partial
    y_{\gamma}} \right)$ and $\tfrac{\partial}{\partial \bar{z}_{\gamma}} \vcentcolon = \tfrac12 \left(
  \frac{\partial}{\partial x_{\gamma}}  + i \tfrac{\partial}{\partial   y_{\gamma}} \right)$, similar for $\delta$.
Further, using a multi-index notation with $\underline{\gamma} = 
\{\gamma_1, \dotsc, \gamma_{q-1}\}, \underline{\delta} = \{\delta_1, \dotsc, \delta_{q-1}\} \in \{1,\dotsc, p\}^{q-1}$, we set 
\[
\Da_{\underline{\gamma}} \vcentcolon= \prod_{j={1}}^{q-1} \Da_{\gamma_j}\quad\text{and}\quad
\overline{\Da}_{\underline{\delta}} \vcentcolon= \prod_{j=1}^{q-1} \overline{\Da}_{\delta_j}.
\]
Then, the form $\psipq$ is defined as 
 \[
\psipq \vcentcolon = \frac{2i (-1)^{q-1}}{2^{2(q-1)}}
\sum_{\substack{\underline\gamma = \{\gamma_1, \dotsc, \gamma_{q-1}\}  \\ \underline\delta = \{ \delta_1, \dotsc, \delta_{q-1}\}}}
\Da_{\underline\gamma}\overline{\Da}_{\underline\delta} \;\varphi_0^{p,q} \otimes \Omega_{q-1}(\underline{\gamma}; \underline{\delta}) \;\in
\bigl[  \Schw(V) \otimes \wwedge{q-1,q-1}(\frakp^*)\bigr]^G,
\]
with
\begin{align*}
 {\Omega_{q-1}}(\underline{\gamma};\underline{\delta}) 
 =  (-1)^{q(q-1)/2}  \sum_{j=1}^q\xi'_{\gamma_1 p+1} \wedge  \xi''_{{\delta_1}
      p+1} \wedge \dotsb \wedge \widehat{\hphantom{j} \xi'_{\cdot p+j} \wedge  \xi''_{\cdot p+j}} \dotsb  \wedge \xi'_{\gamma_{q-1} p+q} \wedge \xi''_{{\delta_{q-1}} p+q}.
\end{align*}
The notation here indicates that in the $j$-th term of the sum, the wedge product of  $\xi'$ and $\xi''$ with second index $p+j$ is omitted.

Finally, recall from \citep[][Prop.~3.2]{FH21} that $\psipq$ is invariant under the operation of $K$, i.e.
\[
\psipq \in \bigl[  \Schw(V) \otimes \wwedge{q-1,q-1}(\frakp^*)\bigr]^K,
\]
and has weight $r= p + q -2$ under the operation of $\SL_2(\R) \simeq \SU(1,1)$.

\begin{remark}
In the Schrödinger model, the Schwartz form $\psipq$ can be expressed using polynomials \citep[see][(3.4)]{FH21}: 
\[
\psipq(x,z_0) = \frac{2i(-1)^{q-1}}{2^{2(q-1)}} \sum_{\underline{\gamma}, \underline{\delta}} P_{\underline{\gamma}, \underline{ \delta}}^{2q-2}(x)\varphi_0^{p,q}\otimes   {\Omega_{q-1}}(\underline{\gamma};\underline{\delta}).
\]
The polynomials  $P_{\underline{\gamma}, \underline{ \delta}}^{2q-2}(x)$ here are of degree $2q-2$. They depend solely on the positive definite components of $x$ and contain only monomials of even degree. 
\end{remark}

\section{Passage to the mixed model} \label{sec:psi_to_mm}

In the main part of this paper most of our calculations take place using the mixed model of the Weil representation. Hence, in this section we carry out the transition from the more common Schrödinger model to the mixed model and construct the intertwiners for the operation of the groups $G = \Ug(p,q)$ and $G' = \SU(1,1) \simeq \SL_2(\R)$.

\paragraph{The mixed model}\label{par:mixedmodel}
The passage to the mixed model of the Weil representation can be realized through a partial Fourier transform. We use hyperbolic coordinates 
by setting
\[
  \ell \vcentcolon = \frac{1}{\sqrt{2}} \left( v_1 + v_m\right), \qquad
  \ell' \vcentcolon = \frac{1}{\sqrt{2}} \left( v_1 - v_m \right), 
\]
and write $x$ in the form $\alpha\ell + x_0 + \beta\ell'  = (\alpha, x_0, \beta)$, with $x_0 \in W \vcentcolon = V\cap \ell^\perp \cap \ell'^\perp$.
We denote the real and imaginary parts of the coordinates by writing $\alpha = \alpha_1 + i \alpha_2$  and $\beta = \beta_1 + i\beta_2$.

Now, passing to the mixed model amounts to calculating the partial Fourier transform with respect to the hyperbolic coordinate $\alpha$ attached to $\ell$ . Since $\alpha$ is a complex variable, one has to calculate the partial Fourier transform in the two real variables $\alpha_1$ and $\alpha_2$. The new coordinate is denoted by $\beta' = \beta_1' + i \beta_2'$. Hence, for a Schwartz form $\phi_{p,q}$, we set
\[
  \widehat\phi_{p,q}(\beta', x_0, \beta) \vcentcolon=
  \int_{\R^2} \phi_ {p,q}\left(\alpha, x_0, \beta\right) 
  e^{2\pi i \left( \alpha_1 \beta_1'  + \alpha_2\beta_2'\right)} \, d\alpha_1 d\alpha_2.
\]
Note that the integral converges since the integrand is a Schwartz function. For $\psipq$, we get the following:

 \begin{proposition} \label{prop:FT_i}
   For a multi-index $\underline\gamma$ denote by $n_\gamma$ the multiplicity of  1 in $\underline\gamma$ and denote by $\tilde\gamma$ the multi-index obtained from $\underline\gamma$ by removing all occurrences of $1$.

  The partial Fourier transform of $\psipq$ with respect to $\alpha$ is given by
  \[
    \begin{multlined}
     \widehat\psipq(\beta', x_0, \beta) = \\
     \frac{2 i (-1)^{q-1}}{2^{2(q-1)}}
     \!\!\!\sum_{\substack{\underline\gamma = \{\gamma_1, \dotsc, \gamma_{q-1}\} \\
       \underline\delta = \{\delta_1, \dotsc, \delta_{q-1}\}}} 
 \!\!  \left( \frac{-i\sqrt{\pi}}{\sqrt{2}}\right)^{n_\gamma + n_\delta} \!
   \left( \beta' - i\beta\right)^{n_\delta}
   \left(\bar\beta' - i\bar\beta\right)^{n_\gamma} 
   P_{\tilde\gamma, \tilde\delta}(x_{0,+})\hat\varphi_0^{p,q} \otimes \Omega_{q-1}(\underline\gamma; \underline\delta).
   \end{multlined}
 \]
 Here, $P_{\tilde\gamma, \tilde\delta}(x_{0, +})$ denotes the polynomial in the positive components of $x_0$, i.e.\ $z_2, \dotsc z_p$, obtained by applying $\Da_{\tilde\gamma}\bar \Da_{\tilde\delta}$ to $\varphi_0$. 
 Its  degree is $2q-2-n_\gamma - n_\delta$. Further, $\hat\varphi_0^{p,q}$ is given by 
 \[
\hat\varphi_0^{p,q} = \exp\left( - \pi\left(\abs{ \beta' -i\beta}^2
+ 2\Im\left(\beta'\bar\beta\right) 
+ 2 i \hlf{x_0}{x_0}_{z_0}  
\right) \right).
 \]
 In particular, for $(p,q) \in \left\{ (p,1), (q,1)\right\}$. we have
   \begin{gather}
\widehat\psi_{p,1} = 2i\, \hat\varphi_0^{p,1} \otimes 1 \quad \text{and} \\
     \widehat\psi_{1,q} = \frac{2i \pi^{q-1}}{2^{3(q-1)}}  \left( \beta' - i\beta\right)^{q-1}    \left(\bar\beta' - i\bar\beta\right)^{q-1} \hat\varphi_0^{1,q}\otimes \Omega_{q-1}\left( \underline{1}; \underline{1} \right).
\end{gather}
\end{proposition}
\begin{proof}
  The Fourier transform of $\varphi_0^{p,q}$ is easily obtained from Lemma \ref{lemma:FT_Bo}. The rest is immediate from the next Lemma \ref{lemma:FT_phik}.
\end{proof}

\begin{lemma}\label{lemma:FT_phik}
  Denote by $\varphi_{k_1, k_2}$ the Schwartz function given by $ \Da_1^{k_1} \bar{\Da}_1^{k_2} \varphi_0$. Then, the  partial Fourier transform 
  of $\varphi_{k_1, k_2}$ with respect to $\alpha$ is given by
  \[
    \hat\varphi_{k_1, k_2}(\beta', \beta, x_0)    =
   \left(\frac{-i\sqrt{\pi}}{\sqrt{2}}\right)^{k_1 + k_2} \left( \beta' - i\beta\right)^{k_2} \left( \bar\beta' - i\bar\beta\right)^{k_1} \hat\varphi_0^{p,q}.
  \]
\end{lemma}
\begin{proof}
  First assume $k_1 = k_2$ and denote by $L_k(t) = \frac{e^t}{k!} \left( \frac{d}{dt} \right)^k \left(e^{-t} t^k\right)$  the $k$-Laguerre polynomial, see p.~\pageref{par:sp_poly}. Then, by \eqref{eq:KM_Laguerre}  
  $\varphi_{k,k}(x)$ takes the form
\[
  \left(\frac{-1}{\pi}\right)^k 2^k k! \, L_{k} \bigl( 2\pi \abs{z_1}^2 \bigr) \varphi_0(x)
 =  \left(\frac{-1}{\pi}\right)^k 2^k k! \, L_{k} \bigl( 2\pi \abs{\alpha + \beta}^2 \bigr) \varphi_0(x).
\]
The  statement in this case follows from Lemma \ref{lemma:ft_laguerre} and the conclusions immediately following Lemma \ref{lemma:FT_Bo}. Indeed,
    \[
    \begin{aligned}
      \hat\varphi_{k,k}(\beta', \beta, x_0) & = \left(\frac{-\pi}{2}\right)^k
      \left( \left( \beta_1' - i \beta_1\right) ^2 +
            \left( \beta_2' - i \beta_2\right) ^2  
          \right)^k \hat\varphi_0^{p,q}\\
          &  = \left(\frac{-\pi}{2}\right)^k
          \left( \beta' - i \beta\right)^k
          \left( \bar\beta' - i \bar\beta\right) ^k  
          \hat\varphi_0^{p,q}. 
          \end{aligned}
  \]
  For the general case, we can assume $k_1 > k_2$. Further, it suffices to consider  $\varphi_{k+1, k}$, as the rest follows through symmetry and by induction. Set $x_i = \alpha_i + \beta_i$, $i=1,2$.

Denote by $H_k(t) = (-1)^k e^{t^2}\left(\frac{d}{dt} \right)^k e^{-t^2}$ the $k$-th Hermite polynomial. By \eqref{eq:KM_Laguerre} and \eqref{eq:Laguerre_binom_Hermite} from Appendix \ref{sec:functions}, one has
\[
  \begin{gathered}
    \left(\frac{1}{\pi} \right)^k 2^k k! L_k\left(2\pi\abs{z_1}^2 \right)\varphi_0 =   \Da_1^k \bar{\Da}_1^k \varphi_0 =  \left( \Da_1\bar{\Da}_1 \right)^k \varphi_0 \\
      = (2\pi)^{-k} \sum_{l=0}^k \binom{k}{l} H_{2(k-l)}\left(\sqrt{2\pi} x_1 \right)H_{2l}\left(\sqrt{2\pi} x_2\right) \varphi_0.
    \end{gathered}
  \]
  Hence, we get
\[
    \begin{multlined}
     (2\pi)^{k+\frac12}\Da_1^{k+1} \bar{\Da}_1^k \varphi_0 = (2\pi)^{k+\frac12}\Da_1 \left( \Da_1^{k} \bar{\Da}_1^k \varphi_0  \right)  =     \\
   	      \sum_{l = 0}^{k} \binom{k}{l}
	\left[ H_{2(k - l) +1}\bigl( \sqrt{2\pi}x_1\bigr) H_{2l}\bigl( \sqrt{2\pi} x_2 \bigr) - i H_{2(k-l)}\bigl(\sqrt{2\pi} x_1\bigr) H_{2l +1}\bigl(\sqrt{2\pi} x_2 \bigr)\right]\varphi_0.
      \end{multlined}
    \]
    Thus, with Lemma \ref{lemma:ft_hermite}, and arguing as before, the Fourier transform of $\varphi_{k+1, k}$ with respect to $\alpha$ is given by
    \[
      \begin{gathered}
        \frac{(-i \sqrt{\pi})^{2k+1}}{2^{\frac{2k+1}{2}}}\!\!
      \sum_{l=0}^{k}  \binom{k}{l} \left[ (\beta_1' - i \beta_1)^{2(k-l)+1}(\beta_2' - i\beta_2)^{2 l}  - i (\beta_1' - i \beta_1)^{2(k-l)}(\beta_2' - i\beta_2)^{2 l +1} \right]\hat\varphi_0^{p,q} \\
      = i \sqrt{\pi}  (-1)^{k+1}\left(\frac{\pi}{2} \right)^k
\left( (\beta_1' - i \beta_1) -i( \beta_2' - i\beta_2) \right) 
      \left( \left( \beta_1' - i \beta_1\right) ^2 +
            \left( \beta_2' - i \beta_2\right) ^2  
          \right)^k \hat\varphi_0^{p,q} \\
          = i \sqrt{\pi}  (-1)^{k+1}\left( \beta' - i\beta \right)^k \left( \bar\beta' - i \bar\beta\right)^{k+1} \hat\varphi_0^{p,q}.
      \end{gathered}
    \]
   This proves the Lemma.
  \end{proof}

\subsection[Intertwining]{Intertwining for \texorpdfstring{$G'$}{G'}  and  \texorpdfstring{$G$}{G}}\label{sec:intertwine_GG}
  
Up to here, through Proposition \ref{prop:FT_i} one has $\psipq$ in the mixed model only at the base-point $z_0$ of $\Dom$, and with $\tau$ fixed at the base point $i$ of the complex upper half-plane $\Hp$. Moving away from the respective base points is accomplished by applying the intertwining operators for $G' \simeq \SU(1,1) \simeq \SL_2(\R)$  and for $G$.

\subsubsection{Intertwining for \texorpdfstring{$\SL_2(\R)$}{G'}}\label{subsec:itw_sl}
Now, we determine the intertwining operators for the operation of $\SL_2(\R) \simeq \SU(1,1)$. To facilitate notation, set $G' = \SU(1,1)$. Following \cite{K16}, we define 
\[
\eta \vcentcolon = [\beta, \beta']  =
\begin{psmallmatrix}
  \beta_1 & \beta'_1 \\ \beta_2 & \beta'_2 
\end{psmallmatrix} \in 
\mathrm{Mat}(2\times 2, \R).
\]

\begin{lemma}\label{lemma:itw_SL} 
  Let $\phi$ be a Schwartz function, and $r$ its weight under the  operation of $K' = \Ug(1)$. The intertwining operators for the action of $G'$ are given by
  \begin{enumerate}
  \item \[
      \mathcal{F}\left(\omega
         \begin{psmallmatrix}
           \sqrt{v} & \\  & \sqrt{v}^{-1}
         \end{psmallmatrix}
      \varphi(\cdot)\right) (\beta', \beta)  =
      v^{-\frac{r}{2} + \frac{p-q}{2}}\frac{1}{v} \hat\varphi\left(\tfrac{1}{\sqrt{v}}\beta', \sqrt{v}\beta\right), 
    \]
  \item 
    \[
      \mathcal{F}\left(\omega
        \begin{psmallmatrix}
          1 & u \\  & 1
        \end{psmallmatrix}
        \varphi(\cdot) \right) (\beta', \beta) =
      \hat\varphi\left( \beta' + u\beta, \beta \right).
    \]
  \end{enumerate}
  Thus, $g_\tau = \begin{psmallmatrix} \sqrt{v} &  \frac{u}{\sqrt{v}}  \\ 0 & \frac{1}{\sqrt{v}} \end{psmallmatrix}$ operates as follows:
  \[
      \mathcal{F}\left( \omega(g_\tau')\varphi(\cdot) \right)(\beta', \beta) =   
  v^{-\frac{r}{2} + \frac{p+q}{2}-1} 
  \hat\varphi\left(\tfrac{1}{\sqrt{v}}\left(\beta' + u\beta \right), 
  \sqrt{v}\beta\right).\]
\end{lemma}
\begin{proof} Direct calculation.\end{proof}
Using the Lemma, one quickly obtains the (partial) Fourier transform of the Gaussian $\varphi_0(x,\tau)  = \varphi_0^{p,q}(x,\tau)$. It takes the form
\begin{equation}\label{eq:exp_pq_mixed} 
  \begin{multlined}
    \widehat\varphi_{0}^{p,q} ((\eta, x_0),\tau)  \\
  =      \exp\left( - \frac{\pi}{v}\left(\abs{ \beta'}^2 + \abs{\bar\tau\beta}^2 + 2u\Re\left(\beta'\bar\beta\right)\right) + 2\pi \bar\tau \hlf{x_{0,-}}{x_{0,-}} + 2\pi\tau\hlf{x_{0,+}}{x_{0,+}} \right) \\
     = \exp\left( - \frac{\pi}{v}\left(\abs{ \beta' + \bar\tau\beta}^2
      + 2v\Im\left(\beta'\bar\beta\right) \right) + 2\pi \bar\tau \hlf{x_{0,-}}{x_{0,-}} + 2\pi\tau\hlf{x_{0,+}}{x_{0,+}} \right).
  \end{multlined}
\end{equation}

\paragraph{Application to \texorpdfstring{$\psipq$}{the Schwartz form}} 
With the notation from Proposition \ref{prop:FT_i}, let $P_{\tilde\gamma, \tilde\delta; \ell}$ denote the homogeneous component of weight $\ell$ of the polynomial $P_{\tilde\gamma, \tilde\delta}$:
\begin{equation}\label{eq:defpqauxpolys}
  P_{\tilde\gamma, \tilde\delta}(x_{0,+}) =
   \sum_{\ell = 0}^{2q -2 -n_\gamma - n_\delta} 
  P_{\tilde\gamma, \tilde\delta; \ell}(x_{0,+}).
\end{equation}
Using Lemma \ref{lemma:itw_SL}, which gives the intertwining for the operation of 
$g_\tau = \left( \begin{smallmatrix}\sqrt{v} & u \sqrt{v}^{-1} \\ 0 & \sqrt{v}^{-1} \end{smallmatrix} \right)$, 
we get $\psipq(x,\tau)$ in the mixed model, given by $\widehat\psipq(x,\tau) = \mathcal{F}(\omega(g_\tau)\psipq(x))(\beta', \beta)$.
\begin{proposition} \label{prop:psipq_mixed}
In the mixed model the Schwartz form $\psipq(x,\tau)$ takes the following form: 
  \begin{equation}
  \begin{aligned}\label{eq:psipq_mixed} 
 \widehat\psipq ((\eta, x_0),\tau) = \frac{2 i (-1)^{q-1}}{2^{2(q-1)}}
     \!\!\!\sum_{\substack{\underline\gamma = \{\gamma_1, \dotsc, \gamma_{q-1}\} \\
       \underline\delta = \{\delta_1, \dotsc, \delta_{q-1}\}}} 
  & \bigl(-i\pi^{\frac12}\bigr)^{n_\gamma + n_\delta}
   \left( 2v \right)^{-\frac{n_\gamma + n_\delta}{2}}
   \left( \beta' + \bar\tau \beta\right)^{n_\delta}
   \left(\bar\beta' +\bar\tau \beta\right)^{n_\gamma}\\
& \cdot  \sum_{\ell = 0}^{2q -2 -n_\gamma - n_\delta}
   v^{\frac\ell{2}} P_{\tilde\gamma, \tilde\delta; \ell}(x_{0,+})
   \cdot\widehat{\varphi_0}^{p,q}
   \otimes \Omega_{q-1}(\underline\gamma; \underline\delta),
 \end{aligned}
  \end{equation}
 with the polynomials $P_{\tilde\gamma, \tilde\delta;\ell}$ from \eqref{eq:defpqauxpolys}.
\end{proposition}
 \begin{proof}
	The Proposition follows directly from Lemma \ref{lemma:itw_SL} with \eqref{eq:exp_pq_mixed}. Recall only that $\psipq$ has weight $r = p+q-2$. 
 \end{proof}
In the special case where the signature of $V$ is $(1,q)$, we have the following Corollary:
\begin{corollary}\label{corol:psi1q_mixed}
      In the mixed model, $\psi_{1,q}(x, \tau)$ is given by
      \begin{align*}
        \widehat{\psi_{1,q}}((\eta, x_0), \tau) &      
        =  \frac{2i \pi^{q-1}}{2^{3(q-1)}}
        v^{-q+1}    
        \hlf{ \vphantom{\dot{\beta}} [\beta, \beta'] g_\tau'}{[1,i]}^k
        \hlf{ \overline{[\beta, \beta']} g_\tau'}{[1,i]}^k 
        \widehat{\varphi_0}^{1,q} 
        \otimes \Omega_{q-1}(\underline{1}; \underline{1}) \\    
       & =         \frac{2i \pi^{q-1}}{2^{3(q-1)}}  v^{-q+1}    
         \left(\beta' + \bar\tau \beta \right)^{q-1}   
         \left(\bar\beta' + \bar\tau \bar\beta \right)^{q-1}
        \cdot\widehat{\varphi_0}^{1,q}  \otimes \Omega_{q-1}(\underline{1}; \underline{1}).
            \end{align*}
  \end{corollary}
 Note that for the special case of signature $(p,1)$, the Schwartz form $\psi_{p,1}(x, \tau)$ is given by the Gaußian
 $2i \varphi^{p,1}_0(x,\tau)$ and hence, by \eqref{eq:exp_pq_mixed} the partial Fourier transform at the base point takes the form
 \[
 \begin{aligned}
 \widehat{ \psi_{p,1}}((\eta, x_0), \tau) &  = 2i
\exp\left( - \frac{\pi}{v}\left( \abs{\beta'}^2 + \abs{\bar\tau\beta}^2 + 2u\Re\left(\beta'\bar\beta\right) \right) + 2\pi\tau  \hlf{x_0}{x_0} \right) \\
& = 2i \exp\left( - \frac{\pi}{v}\left( \abs{\beta' +  \bar\tau\beta}^2 + 2v\Re\left(\beta'\bar\beta\right) \right) + 2\pi\tau  \hlf{x_0}{x_0} \right). 
\end{aligned}
 \]

\subsubsection{Intertwining for \texorpdfstring{$G$}{G}}\label{subsec:itw_G}

Let us now consider the operation of the $G$ and, in particular, of its parabolic subgroups.  

\paragraph{Intertwining for the operation of the parabolic subgroup \texorpdfstring{$P_\ell\subset G$}{P(l) in G}} \label{paragraph:intertwine_NAM}
First, we determine the intertwining operators. We use the   
Levi decomposition $G = NAM$, where the Levi factor is given by the direct product of the groups $M\simeq \SU(p-1,q-1)$ and $A \simeq\mathrm{GL}([\ell])$. Using the basis $\ell, v_2, \dotsc, v_{m-1}, \ell'$, the elements of $A$ and $M$ are written as matrices in the form
\[
a(t) = \begin{pmatrix}
  t & &  \\
  & 1_{m-1} & \\
  & & t^{-1} 
\end{pmatrix} \quad \bigl(t\in\R_{>0}\bigr),
\qquad
\mu = \begin{pmatrix} 
1 & & \\ & \mu' & \\ & & 1 
\end{pmatrix}\quad \bigl(\mu' \in \SU(W)\bigr),
\]
while the elements of the Heisenberg group are given by matrices of the form
\[
\begin{aligned}
n(0,r) & = \begin{pmatrix}
  1 & 0 &  i r \\
   & 1_{m-1} &  \\
  & & 1
\end{pmatrix} \quad (r\in\R),\\
n(w,0) & = \begin{pmatrix}
  1 & -\bar{w}^t & -\frac12 \hlf{w}{w}  \\
   & 1_{m-1} &   w\\
  & & 1
\end{pmatrix} \quad (w\in W), 
\end{aligned}
\]
and satisfy the group law $n(w_2,0)\circ n(w_1,0) = n\left( w_1 + w_2, -\Im\hlf{w_2}{w_1}\right)$.   
\begin{lemma}\label{lemma:op_Pell}
 Let $\varphi$ be a Schwartz form. The intertwining operators for the operation of the subgroups $N$, $A$ and $M$ are given as follows: 
  \begin{enumerate}
  \item $\mathcal{F}\bigl(\widehat{n(0,r) \varphi} \bigr)$:
  \[
      \hat\varphi\left(([\beta, \beta'], x_0), \tau, z_0\right) 
      e\left(  r\Im(\beta'\bar\beta) \right).
    \]
  \item $\mathcal{F}\bigl(\widehat{ n(w,0) \varphi}\bigr)$:
    \[
      \hat\varphi\left(([\beta, \beta'], x_0 - \beta w), \tau, z_0\right) 
      e\left(\tfrac12 \Re(\beta'\bar\beta) \hlf{w}{w} - 
        \Re\left(\beta'\hlf{x_0}{w}\right)  \right).  
    \]
  \item $\mathcal{F}\bigl( \widehat{a(t) \varphi} \bigr)$:
    \[
      t^2 \hat\varphi\left((t[\beta, \beta'], x_0), \tau, z_0\right).	
    \]
  \item $\mathcal{F}\bigl(\widehat{\mu \varphi}\bigr)$:
    \[
      \hat\varphi\left([\beta, \beta'], \mu^{-1}x_0), \tau, z_o\right). 
    \]
  \end{enumerate}
(Note that if either $p$ or $q$ is $1$,  $M$ is compact.)
\end{lemma}
\begin{proof}
 Since
\[
g \varphi = \varphi( x, \tau, g z_0) = \varphi( g^{-1} x, \tau, z_0), 
\]
the operation of $N$ and the elements of the Levi-factor take the form
\[
\begin{aligned}
n(0,r)\varphi & =  \varphi(x, \tau ,n(0,r)z_0) = \varphi( (\alpha - ir\beta, x_0, \beta), \tau, 
z_0), \\
n(w,0) \varphi & = 
\varphi((\alpha, x_0, \beta), \tau, n(w,0)z_0 ) =
\varphi(n(-w,0) (\alpha, x_0, \beta), \tau, z_0 ) \\
& = \varphi((\alpha  +   \hlf{w}{x_0} - \beta\tfrac12\hlf{w}{w}, x_0 - w\beta, \beta), 
\tau, z_0), \\
a(t)\varphi & =  
\varphi\left((\alpha, x_0, \beta), \tau, a(t) z_0\right) =
\varphi\left((t^{-1}\alpha, x_0,  t\beta), \tau, z_0\right), \\
\mu\varphi & = \varphi((\alpha, x_0, \beta), \tau, \mu z_0) = \varphi ((\alpha, \mu^{-1}x_0, \beta),\tau, z_0). 
\end{aligned}
\]
The claim follows easily by calculating of the partial Fourier transform in $\alpha$.
\end{proof}

\paragraph{The Schwartz form \texorpdfstring{$\psipq$}{} in the mixed model} From Proposition \ref{prop:psipq_mixed}, though the intertwining operators from Lemma \ref{lemma:op_Pell}, we obtain the following Proposition:
 \begin{proposition} \label{prop:psipq_NAM_mixed} 
  Let $g = m(w,0)  \circ m(0,r) \circ a(t)$. Then, $\widehat{\psipq}(gz_0)$ is given by
  \begin{gather*}
    \widehat\psipq  ((\eta, x_0),\tau)(g z_0)  =
    \frac{2 i (-1)^{q-1}}{2^{2(q-1)}}
     \!\!\!\sum_{\substack{\underline\gamma = \{\gamma_1, \dotsc, \gamma_{q-1}\} \\
       \underline\delta = \{\delta_1, \dotsc, \delta_{q-1}\}}} 
   \bigl(-i\pi^{\frac12}\bigr)^{n_\gamma + n_\delta}
   \left( 2v \right)^{-\frac{n_\gamma + n_\delta}{2}}
 t^{2 + \frac{n_\gamma + n_\delta}{2}}
   \\
  \,\cdot\; 
    \left( \beta' + \bar\tau \beta\right)^{n_\delta}
   \left(\bar\beta' +\bar\tau \beta\right)^{n_\gamma}
  \sum_{\ell = 0}^{2q -2 -n_\gamma - n_\delta}
  v^{\frac\ell{2}} P_{\tilde\gamma, \tilde\delta; \ell}(x_{0,+} - \beta w_{+})\cdot \widehat\varphi_{0}^{p,q}((t\eta, x_{0,+} - \beta w_+), \tau) \\
  \,\cdot\; 
e\left( r\Im(\beta' \bar\beta) + \frac12\Re(\beta'\bar\beta) \hlf{w}{w} - \Re\left(\beta'\hlf{x_0}{w}\right)\right) 
\otimes \Omega_{q-1}(\underline\gamma; \underline\delta).
  \end{gather*}
We note that the intertwining operator for the action of $\mu \in M \simeq \SU(p-1, q-1)$ is trivial. 
\end{proposition}

\subsection[The Fourier expansion of a harmonic Maass form]{Fourier expansion of a weak harmonic Maass form in the mixed model}\label{subsec:Hmf_mixed}
Let $f\in \HmfLp{k}{L^-}$ be a weak harmonic Maass with weight $k = 2 - (p+q)$ under the discrete Weil representation of $\rho_{L^-}$. Let $a^+(h,n)$ and $a^-(h,n)$  denote the Fourier coefficients of its holomorphic part and non-holomorphic part, respectively. The Fourier expansion of $f$ in the mixed model can be described as follows, see \citep[][p.\ 23]{K16}:
\[
f(\tau) = \sum_{n\gg -\infty} \hat{\mathbf{a}}^{+}(n)q^n + \hat{\mathbf{a}}^- (0) v^{1-k} + \sum_{\substack{ n\gg-\infty \\ n\neq 0}} \hat{\mathbf{a}}^-(n)\Gamma\left(1-k, 4\pi\abs{n} v\right) e^{2\pi i n u},
\]
where we have introduced vector-valued Fourier coefficients 
$\hat{\mathbf{a}}^\pm(n)$ by setting.
\[
  \hat{\mathbf{a}}^\pm(n) \vcentcolon = \sum_{\substack{\lambda = \lambda_\ell \ell + \lambda_W + \lambda_{\ell'}\ell \\ \lambda \in L^\dual/L}} a^\pm(\lambda,n) \hat\ebase_\lambda
  =   \sum_{\substack{\lambda = \lambda_\ell \ell + \lambda_W + \lambda_{\ell'}\ell \\ \lambda \in L^\dual/L}} a^\pm(\lambda,n) e\left( - \lambda_2\cdot \beta'\right) \ebase_\lambda. 
\]

\section[The singular theta lift, Kudla's approach]{Evaluating the singular theta lift based on Kudla's approach}\label{sec:eval_Kstyle}
Denote by $\Theta(\tau, z; \psipq)$ the theta series attached to the Schwartz form  $\psipq$. For a weak harmonic Maass form $f \in \HmfLp{k}{L^-}$ with $k=2 - (p+q)$, the singular theta lift of Borcherds type studied in \cite{FH21} is given by the regularized integral
\begin{equation}\label{eq:reg_int}
\Phi(z,f, \psipq)  = \int_{\Gamma/\Hp}^{reg}  f(\tau) \,\Theta(\tau,z; \psipq)\, d\mu(\tau) 
= \int_{\Gamma/\Hp}^{reg}  \left\langle f(\tau), \overline{\Theta(\tau,z)} \right\rangle_L d\mu.
\end{equation}
The regularization follows the standard procedure introduced by Harvey and Moore \cite{HM96} and Borcherds \cite{Bo98}: Denote by $\mathcal{F}_t$ $(t \in  \R_{>0})$ a truncated fundamental domain given by 
\[
\mathcal{F}_t \vcentcolon = \left\{ \tau = u  + iv;\; \abs{\tau}>1, \abs{u}<\tfrac12, 0<v\leq t\right\}.
\]
Then, 
\[
\int_{\Gamma/\Hp}^{reg}  \left\langle f(\tau), \overline{\Theta(\tau,z)} \right\rangle_L d\mu
\vcentcolon =   \CT_{s=0}\left[\lim_{t\rightarrow \infty} \int_{\mathcal{F}_t}\left\langle f(\tau), \overline{\Theta(\tau,z)}\right\rangle_L v^{-s} d\mu  \right],
\]
where the notation $\CT_{s=0}$ means taking the constant term at $s=0$ of the meromorphic continuation of the limit\footnote{If $0$ happens to be a pole, a slight variation of this recipe is required, see \cite{Br02}.}.

To evaluate the regularized theta integral and to calculate the Fourier-Jacobi expansion of $\Phi(z, f, \psi)$ we employ a method recently introduced by Kudla in \cite{K16}. The key observation is that, due to invariance under the action
of $\Gamma = \SL_2(\Z)$, the theta function can be decomposed along the
$\Gamma$-orbits of $\eta = [\beta, \beta']$: 
\begin{align*}
\Theta(\tau, z; \psipq) & = 
\sum_{h \in L^\dual/L} \sum_{\substack{ \lambda \in L + h \\ \lambda = (\alpha, x_0, \beta)}} \mathcal{F}\left( \omega\left(g_\tau, \psi_{\sqrt{2}}\right) \psipq(\lambda, z)\right) \ebase_h\\
& = \sum_{h \in L^\dual/L} \sum_{\substack{(\eta, x_0) \in L + h \\ \eta = [\beta,\beta'], \, x_0 \in W}}  \widehat\psipq\left(\sqrt{2}(\eta, x_0), \tau, z\right) \ebase_h\\
& = \sum_{\eta / \sim} \sum_{\gamma \in \Gamma_{\eta} \backslash \Gamma} \theta_{\gamma\eta}( \tau, z),
\intertext{with} 
\theta_{\gamma\eta}(\tau ,z) &  = 
\sum_{h \in L^\dual/L} \sum_{ \substack{ x_0 \in W \\ (\eta, x_0) \in L + h}} 
\widehat\psipq\left(\sqrt{2}(\eta g, x_0), \tau, z\right) \ebase_h
\end{align*}
Here, $\Gamma_\eta$ denotes the stabilizer of $\eta$ in $\Gamma$, and $\eta/\sim$ denotes the orbit of $\eta$ under the action of $\Gamma$. A set of orbit representatives is given by the following Lemma. Note that with this choice of representatives, all $\eta$ of a given rank have the same stabilizer $\Gamma_\eta$.

\begin{lemma}(\citep[see][p.\ 20]{K16})\label{lemma:Ks_reprs}
 The orbits of matrices in $\mathrm{M}_2(\mathbb{Q})$ under the operation of $\SL_2(\Z)$ with their respective sets of representatives and stabilizers in $\SL_2(\Z)$ are  the following:
 \begin{enumerate}
\item The zero orbit, stabilized by the whole of $\SL_2(\Z)$.
\item The orbits of rank 1 matrices. A set of representatives is given by
\[
\begin{pmatrix}
0  & a \\ 0 & b
\end{pmatrix}\quad \text{with}\;a>0 \;\text{or $a=0$ and $b>0$.}  
\]
They are stabilized by $\SL_2(\Z)_\infty = \left\{ \begin{psmallmatrix}
1 & n \\ 0 & 1
\end{psmallmatrix}; \, n\in\Z \right\}$. 
 \item The orbits of rank 2 matrices. A set of representatives is given by all matrices of the form
	\begin{equation} \label{eq:param_rk2}
	\begin{pmatrix} a & b \\ 0 & \alpha
	\end{pmatrix} \quad \text{with}\; a,\alpha \in \Q^\times, a>0 \quad\text{and}\quad
	b\in\Q \mod{a\Z}.
	\end{equation}
	The stabilizer of any such orbit is trivial. 
      \end{enumerate}
\end{lemma}
From the decomposition of the theta function and as a consequence of Lemma \ref{lemma:Ks_reprs}, the integrand of the regularized integral in \eqref{eq:reg_int} decomposes along $\Gamma$-orbits, which can be ordered according to the rank of their representatives,  
\[
\left\langle f, \overline{\Theta}\right\rangle_L = \sum_{i=0}^2 \sum_{\substack{\eta / \sim \\ 
\operatorname{rank}(\eta) = i}}\sum_{\gamma \in \Gamma_\eta \backslash \Gamma} \left\langle  
f(\gamma\tau), \overline{\theta_{\gamma\eta}( \tau, z)}  \right\rangle_L.
\]
Since each term is $\Gamma$-invariant, the regularized integral can be decomposed similarly.

Note that the contributions of $\gamma$ and $-\gamma$ are the same, as $-1_2 \in \Gamma$ operates trivially. Hence, after unfolding, a factor of $2$ occurs for the non-zero orbits, i.e.~for $i=1,2$. 
Thus, we write  
\begin{gather}  \label{eq:def_sum_phi}
  \Phi(z, f, \psipq) = \Phi_0 (z, f, \psipq) + 2\Phi_1 (z, f, \psipq) + 2\Phi_2(z, f, \psipq),\quad\text{with} \\ 
\notag \Phi_i(z, f, \psipq) \vcentcolon = 
\int_{\Gamma\backslash \mathbb{H}}^{reg}
\sum_{\substack{ \eta =  [\beta, \beta'] / \sim\\  \operatorname{rank}(\eta) = i}}
\sum_{\gamma\in\Gamma_\eta\backslash \Gamma} \left\langle 
f(\gamma\tau), \overline {\theta_{\eta}(\gamma\tau, z)} \right\rangle_L 
v^{-2} du\, dv \quad(i=0,1,2).
\end{gather}
Further, due to rapid decay of the integrand, the integral can be evaluated for each term separately, with fixed $\eta$, and summed up later. 

Hence, we set
\[
  \Phi_i(z, f, \psipq) =\vcentcolon \sum_{\substack{ \eta =  [\beta, \beta'] / \sim\\  \operatorname{rank}(\eta) = i}}
  \sum_{n \gg -\infty}
 \left(\hat{a}^+(n) \phi_i(n,\eta)^+ + \hat{a}^{-}(n)\phi_i(n,\eta)^{-}\right)\qquad \left(i = 0,1,2\right),
\]
with coefficients $\hat{a}^\pm (n)$ obtained from the mixed model Fourier coefficients 
$\hat{\mathbf{a}}^\pm(n)$ of $f^\pm$ (see Section \ref{subsec:Hmf_mixed}) via $ \hat{a}^\pm(n) = \left\langle\sum_{h \in L^\dual/L} \ebase_h, \hat{\mathbf{a}}^\pm(n) \right\rangle_L$,  
and with
\begin{align*}
\phi_i(n,\eta)^+ & = \int_{\Gamma_\eta\backslash \Hp}^{reg} e^{2\pi i n u} \theta_\eta(\tau, z)\, d\mu(\tau) \\
  \phi_i(n,\eta)^-
   & = \int_{\Gamma_\eta \backslash \Hp}^{reg} \Gamma(1-k, 4\pi\abs{n})e^{2\pi i n u} \theta_{\eta}(\tau, z)\,  d\mu(\tau) \qquad (n\neq 0), \\
   \phi_i(0, \eta)^- & = \int_{\Gamma_\eta \backslash \Hp}^{reg} \theta_{\eta}(\tau, z) v^{1-k}\, d\mu(\tau).
  \end{align*}
  
  The domain of integration depends on $i$. For $i=0$ we have a usual fundamental domain $\Gamma\backslash \Hp$, while for $i = 1$ we have 
  $\Gamma_\infty \backslash \Hp$. Finally, for $i=2$ the domain of integration is the whole upper half plane. 

An advantage of this approach is that the Fourier-Jacobi expansion of the theta lift is fairly easy to calculate. 
In fact, the constant term of the Fourier expansion is obtained from the rank 1 terms of the lift and the zero-orbit, 
while the remaining terms can be read off from the terms for non-singular $\eta$, see Section \ref{sec:FJ_exp} below.

To facilitate calculation across different signatures, we will first evaluate the integrals at the base point $z= z_0 \in \Dom$, this is carried out in Section \ref{sec:unfolding}, and apply intertwining operators for the operation of $G$ to the individual terms $\phi_i(n,\eta)^\pm$ after integration.

Recall the $NAM$ decomposition of $G$ from Section \ref{subsec:itw_G}. Assume that for every $z$ in $\Dom$ a continuous choice of  $n(w,r) \in N$, $a(t) \in A$ and $\mu \in M$ has been fixed such that $z  = (n(w,r)a(t)\mu )(z_0)$. 
We will use $w,r.a$ and $\mu$ as coordinates to describe the Fourier-Jacobi expansion, see Theorem \ref{thm:FJexp_pq}. Since the Fourier-Jacobi expansion is closely linked to the operation of the Heisenberg group $N$, this appears as a natural choice for our purpose. 

\begin{remark}\label{rmk:no_rkzero_term} 
	As mentioned in the introduction, we should point out that we will not evaluate the rank $0$ term $\Phi_0(z, f, \psipq)$, 
	in general, at least not if $p,q>1$. This term is given by a convolution integral of the harmonic Maass form $f$ with an 
        (in general) indefinite theta series for the lattice $L \cap W$.  
        Of course, if $q=1$, there is no contribution of the rank $0$ term, so $\Phi_0 \equiv 0$. 
	If $p=1$, the theta series is definite, and this kind of integral, at least for $f\in\MfLw{k}{L^-}$, 
	has already been treated by Borcherds \citep[see][Sec. 9]{Bo98}. 
	Hence, in this special case, the value of $\Phi_0(z_0, f,\psi_{p,1})$ can be deduced from Borcherds' results, following Kudla \citep[][]{K16}, see Lemma \ref{lemma:Bprod_rank01} below.
\end{remark}

\section{The lift at the base point}\label{sec:pq_lift}
In this section, we state our first main result, an explicit expression for the singular theta lift of a weak harmonic Maass form $f$ evaluated at the base point. Our Theorem \ref{thm:Phipq_atz0} covers the general case of signature $(p,q)$ with $p,q\geq 1$ and $p+q>2$. The calculations for this are carried out in Section \ref{sec:unfolding}. Two special cases, where either $q=1$ or $p=1$ are treated in Examples \ref{ex:sig_p1} and  \ref{ex:sig_q1}.

Recall the kernel function from Proposition \ref{prop:psipq_mixed}. We have
\begin{equation}
  \begin{aligned}
 \widehat\psipq & (\sqrt{2}(\eta, x_0),\tau)  = \frac{2 i (-1)^{q-1}}{2^{2(q-1)}}
     & \sum_{\substack{\underline\gamma = \{\gamma_1, \dotsc, \gamma_{q-1}\} \\
       \underline\delta = \{\delta_1, \dotsc, \delta_{q-1}\}}} 
\left( -i \sqrt{\pi} \right) ^{n_\gamma + n_\delta} 
   \left( \beta' + \bar\tau \beta\right)^{n_\delta}
   \left(\bar\beta' +\bar\tau \beta\right)^{n_\gamma}\\
& \qquad \cdot\sum_{\ell = 0}^{2q - n_\gamma - n_\delta -2}
  2^{\frac{\ell}{2}} v^{\frac{(\ell- n_\gamma - n_\delta)}{2}}
  & P_{\tilde\gamma, \tilde\delta; \ell}(x_{0,+})
   \exp\left(-\frac{2\pi}{v}\left(\abs{ \beta' + \bar\tau\beta}^2
     + 2v\Im\left(\beta'\bar\beta\right)\right) \right)
     \\ & & \quad\cdot\;
   e\bigl(\bar\tau\hlf{x_{0,-}}{x_{0,-}} + \tau \hlf{x_{0,+}}{x_{0,+}}\bigr)
   \otimes \Omega_{q-1}(\underline\gamma; \underline\delta),
   \end{aligned}
 \end{equation}
 where $n_\gamma$ and $n_\delta$ denote the number of $1$'s occurring  in  the respective  multi-index, and $\tilde\gamma$ and $\tilde\delta$ denote the remaining multi-indices after striking out all occurrences of $1$. The ranks of $\tilde\gamma$ and $\tilde\delta$  are $q-1 - n_\gamma$ and $q - 1 - n _\delta$, respectively.

First, we introduce some notation which will be used in this and the following sections:

\begin{notation}\label{not:def_ABC}
 If $\eta = [\beta, \beta']$ is non-singular, define $\AAp{\eta}$, $\AAm{\eta}$, $\AAb{\eta}$, $\BB_\eta$ and $\CC_\eta$ by setting
\begin{equation}\label{eq:def_A_eta}
  \begin{gathered} 
    \AAp{\eta}(n , x_0) \vcentcolon=  
    n  + 2\abs{\beta}^2 + \hlf{x_0}{x_0},  \quad 
    \AAm{\eta} (n, x_0) \vcentcolon =  n  - 2\abs{\beta}^2 + \hlf{x_0}{x_0} \\
    \AAb{\eta}(n,x_0) \vcentcolon = \frac12\left( n + \hlf{x_0}{x_0}\right)^2 + 2\abs{\beta}^4  + 2\abs{\beta}^2 \hlf{x_0}{x_0}_{z_0} + 2\abs{\beta}^2n, \\  =  \frac12\AAp{\eta}(n, x_0)^2 - 4\hlf{x_{0,-}}{x_{0,-}}
  \end{gathered}
\end{equation}
and
\begin{equation} \label{eq:def_BC}
\BB_\eta \vcentcolon =2\left( \abs{\beta'}^2 \abs{\beta}^2 - \Re\left(\beta'\bar\beta\right)^2 \right) 
= 2\Im\left(\beta'\bar\beta\right)^2 \quad
\text{and}\;\, \CC_\eta \vcentcolon=  \frac{\Re(\beta'\bar\beta)}{\abs{\beta}^2}.
\end{equation}
When $n$ or $x_0$ are fixed, we will drop either (or both) and just write, e.g.\   $\AAb{\eta}$ if both $x_0$ and $n$ are fixed.
We note that both  $\BB_\eta$ and $\AAb{\eta}$ are non-negative real numbers, and if $\beta\neq 0$ they are both positive.

Further, generalizing \citep[][(3.25)]{Br02}, we introduce the following special function 
\begin{equation}  \label{eq:def_vint}
 \begin{aligned}
\Vint{N, \mu}{A, B, c} & \vcentcolon  =   \int_0^\infty \Gamma(N - 1, c v) v^{-\mu} e^{-Av -Bv^{-1}} dv \qquad \bigl(N\geq 2, \Re(A+c), \Re(B)>0\bigr)\\
 & =  2(N-2)!  \sum_{r=0}^{N-2} \frac{c^r}{r!} \left(\frac{A + c}{B}\right)^{\frac{\mu - r -1 }{2}}\!
    K_{r+1-\mu}\left(2\sqrt{(A+c)B}\right),
  \end{aligned}
\end{equation}
see Lemma \ref{lemma:int_besselgamma} concerning the evaluation of the integral. Also note that for $N = p+q$ one has  $N-1 = 1 - k$ and $N-2 = -k = \kappa$.
\end{notation}
Now, with the notation from Section \ref{sec:eval_Kstyle}, consider the terms $\phi_i(n,\eta)^\pm$, $i=1,2$. Since the differential form part $\Omega_{q-1}(\underline{\gamma}, \underline{\delta})$  of the Schwartz form $\psipq$ depends on the multi-indices $\underline{\gamma}$, $\underline{\delta} \in \left\{1, \dotsc, p \right\}^{q-1}$, each term $\phi_i(n,\eta)^{\pm}$ can be decomposed as
\[
\phi_i(n,\eta)^\pm = \frac{2i(-1)^{q-1}}{2^{2(q-1)}}\sum_{\underline{\gamma},\underline{\delta}} \phi_i^{\underline{\gamma}, \underline{\delta}}(n,\eta)^\pm \otimes \Omega_{q-1}\left(\underline{\gamma}, \underline{\delta}\right).  
\]
Hence, for $f\in\HmfLp{k}{L^-}$ the regularized integral $\Phi(z_0,f, \psipq)$ can be written in the form
\begin{align*}
    \Phi(z_0, f, \psipq) &  = \Phi_{12}(z_0, f, \psipq) + \Phi_0(z_0, f, \psipq) \\
                       &  = 2\sum_{i=1}^2 \left( \Phi_i(z_0, f^+, \psipq) + \Phi_i(z_0, f^{-}, \psipq)  \right) + \Phi_0(z_0, f, \psipq) \\ 
  \intertext{with}
   \Phi_i(z_0, f^\pm, \psipq) & =  \frac{2i(-1)^{q-1}}{2^{2(q-1)}}
\sum_{\underline{\gamma}, \underline{\delta}} 
\Bigl(  \sum_{\substack{ \eta =  [\beta, \beta'] / \sim\\  \operatorname{rank}(\eta) = i}} 
\sum_{n\gg -\infty} \hat{a}^\pm(n)  
\phi^{\underline{\gamma}, \underline{\delta}}_i(n,\eta)^\pm   \Bigr)\otimes \Omega_{q-1}\bigl(\underline{\gamma}, \underline{\delta}\bigr).
\end{align*}
Note the factor of 2 which occurs after unfolding for the terms with $i=1,2$, since $\gamma$ and $-\gamma$ have the same contribution.  
The individual terms $\phi_i^{\underline{\gamma}, \underline{\delta}}(n, \eta)^\pm$ $(i=1,2)$  are calculated in Section \ref{sec:unfolding}. The Lemmas there are the basis for the proof the main result of this section, Theorem \ref{thm:Phipq_atz0}. 
Finally, the coefficients $\hat{a}^\pm(n)$ are explicitly given by 
\begin{equation}\label{eq:ahatn}
\hat{a}^\pm(n) =  
\left\langle \sum_h \ebase_h, \hat{\mathbf{a}}^\pm(n) \right\rangle_L = 
\sum_{\substack{\lambda = \lambda_\ell \ell + \lambda_W + \lambda_{\ell'}\ell \\ \lambda \in L^\dual/L}} a^\pm(\lambda,n) e\left( - \lambda_2\cdot \beta'\right). 
\end{equation}

We fix some more notation.   
Given $n_\gamma, n_\delta$ as defined above, with $0\leq n_\gamma, n_\delta \leq q-1$, and $0\leq M \leq n_\gamma + n_\delta$, set 		
\[
R_{n_\delta, n_\gamma}(\eta, M) \vcentcolon=  
\sum_{ \substack{ 0\leq \mu_1 \leq n_\delta  \\ 0\leq \mu_2 \leq n_\gamma  \\ \mu_1 + \mu_2 = M }}
\beta^{\mu_1} \bar\beta^{\mu_2} \beta'^{n_\delta - \mu_1} \bar\beta'^{n_\gamma - \mu_2} \binom{n_\delta}{\mu_1}\binom{n_\gamma}{\mu_2}.
\]
Now, we can state the Theorem. 
\begin{theorem}\label{thm:Phipq_atz0} 
	For a weak harmonic Maass form $f \in \HmfLp{k}{L^-}$ with Fourier coefficients $a^{\pm}(h,n)$ $(h\in L^\dual/L, n\in\Q, -\infty \ll n)$ 
	the regularized theta integral $\Phi_{12}(z_0, f, \psipq)$  is given by 	
	\[
	\begin{aligned}
	\Phi_{12} & (z_0, f, \psipq)   
	=  2\cdot\frac{2i(-1)^{q-1}}{2^{2(q-1)}}  
	\sum_{\substack{\underline{\gamma}, \underline{\delta}}}  \;
        \\ & \;  \cdot\,
   \biggl\{ 	\sum_{i = 1}^2 
	 \sum_{\substack{ \eta =  [\beta, \beta'] / \sim \\  \operatorname{rank}(\eta) = i }} 
	\sum_{n\gg -\infty}   
	\left(\hat{a}^+(n) \phi^{\underline{\gamma}, \underline{\delta}}_i (n,\eta)^+(z_0)
	+ \hat{a}^{-}(n)\phi^{\underline{\gamma}, \underline{\delta}}_i (n,\eta)^-(z_0)
	\right)  
	\biggr\}\otimes
		\Omega_{q-1}\bigl(\underline{\gamma}; \underline{\delta}\bigr),
	 \end{aligned}
	\]
	where for fixed $m$ and $\eta$ the contributions to the inner sum are the following:
	\begin{enumerate}
	\item 
	The non-singular terms $\phi^{\underline{\gamma}, \underline{\delta}}_2 (n,\eta)^+(z_0)$ and 
	$\phi^{\underline{\gamma}, \underline{\delta}}_2 (n,\eta)^-(z_0)$ are given by the sum 
	\begin{equation} 
	 \begin{multlined}   
	    \left( -i  \sqrt{\pi}\right)^{n_\gamma  + n_\delta} 
		 \sum_{\ell = 0}^{2q - 2 - n_\gamma - n_\delta} 2^{\frac{\ell + 1}{2}} 
		 P_{\tilde\gamma, \tilde\delta, \ell} (x_{0, +})   
		 \sum_{M=0}^{n_\gamma + n_\delta}  R_{n_\delta, n_\gamma}(\eta,M)
		 \\  \;\cdot\, \sum_{j=0}^{\left\lfloor \frac{M}{2} \right\rfloor} \frac{1}{\pi^j} \sum_{h = 0}^{M -2j}
		 \frac{i^h}{2^{h + 3j}}  
		 \AAm{\eta}^h \left( - \CC_\eta\right)^{M - 2j -h}  \abs{\beta}^{- 2h - 2 j}
		 \binom{M -2 j}{h} \frac{M!}{j! (M-2j)!} 
		 \\ \cdot \mathfrak{a}^{\pm}_\nu(\eta;n) \exp\Bigl( -2\pi i \CC_\eta\left(\hlf{x_0}{x_0}  + n\right) \Bigr).
	 \end{multlined}
	\end{equation}
	The index $\nu$ in the inner sum is given by $\nu = h + j +  \frac12(\ell - n_\gamma - n_\delta -1)$, the attached factors 
	$\mathfrak{a}^{\pm}_\nu(\eta;n)$  take the form 
	\begin{align}
		 \mathfrak{a}^{+}_\nu(\eta;n) & =  \left( \frac{\AAb{\eta}}{\BB_\eta}\right)^{-\frac{\nu}{2}}
		 K_\nu\left( \frac{2\pi}{\abs{\beta}^2} \left(\AAb{\eta}\BB_\eta\right)^{\frac12}\right), \label{eq:phi2pq_h} \\
		  \mathfrak{a}^{-}_\nu(\eta;n)  & =  
           \tfrac{1}{2} \Vint{p+q, 1-\nu}{\pi \left(\frac{ \AAb{\eta}}{\abs{\beta}^2} -2n \right), \frac{\pi \BB_{\eta}}{\abs{\beta}^2}, 4\pi \abs{n}} \quad (n\neq 0).  \label{eq:phi2pq_nh}
	 \end{align} 
         Finally for $n = 0$, the non-holomorphic part $\phi^{\underline\gamma, \underline\delta}(0, \eta)^{-}$
         takes the same form \eqref{eq:phi2pq_h} as   $\phi^{\underline\gamma, \underline\delta}(0, \eta)^{+}$
         but with a shifted index, where $\nu$ is replaced by $\nu - k +1$.
       \item 
	 The terms $\phi^{\underline{\gamma}, \underline{\delta}}_1 (n,\eta)^+(z_0)$ and  
	 $\phi^{\underline{\gamma}, \underline{\delta}}_1 (n,\eta)^-(z_0)$ consist of a  sum 
	 \[
	 	 \left(-i \sqrt{\pi}\right)^{n_\gamma + n_\delta} 
	 	 \beta'^{n\gamma}\bar\beta'^{n_\delta}\sum_{\ell = 0}^{2q - 2 -n_\gamma - n_\delta}
	 	  2^{\frac{\ell}{2}} P_{\tilde\gamma, \tilde\delta, \ell}(x_{0,+})  
	 	  \mathfrak{b}^\pm_\nu(\eta;n),
	  \]
	  wherein the factor $\mathfrak{b}^\pm_\nu (\eta;n)$ with index $\nu = \frac12\left(\ell - n_\gamma - n_\delta\right) -1$ is given by 
	  \[	  
	  \mathfrak{b}^+_\nu(\eta;n) =
	  2\abs{\beta'}^{\nu} \left(2 \abs{ \hlf{x_{0,-}}{x_{0,-}} }\right)^{-\frac{\nu}2} 
	 K_\nu\left( 2\sqrt{2} \pi\abs{\beta'} \abs{\hlf{x_{0,-}}{x_{0,-}}} \right), 
	 \]
	 for the holomorphic term and by 
	 \[
       \mathfrak{b}^{-}_\nu(\eta;n) = 
                     \Vint{p+q,1-\nu}{2\pi\hlf{x_{0,-}}{x_{0,-}}, \pi\abs{\beta'}^2, 4\pi\abs{n}} 
	 \]
	 for the non-holomorphic term if $n\neq 0$. 
	 
     For $n=0$, the contribution for the non-holomorphic part is the same as for the holomorphic part, but with index shifted by $-k+1$, i.e. with $\mathfrak{b}^-_{\nu'}(\eta;0) = \mathfrak{b}^+_{\nu - k+1}(\eta; 0)$. 
\end{enumerate} 
\end{theorem}
\begin{proof}
Follows directly from the calculations in Section \ref{sec:unfolding}. More specifically, the results for the non-singular terms can be found in the Lemmas \ref{lemma:pq_phi2h} and \ref{lemma:pq_phi2nh}, the results for rank 1 terms are found in the Lemmas \ref{lemma:pq_phi1h} and \ref{lemma:pq_phi1nh}. 
\end{proof}

\begin{remark}\label{rmk:if_nuhalf}
  In Theorem \ref{thm:Phipq_atz0}, whenever the index $\nu$ is a half-integer, i.e.\ $\nu\equiv \frac12 \pmod{1}$, the Bessel functions can be expressed through Bessel polynomials.

  For example, for the non-singular terms, if $\frac12\left(\ell- n_\gamma - n_\delta -1  \right) \equiv \frac12 \pmod{1}$, the Bessel function in  \eqref{eq:phi2pq_h} can be replaced by an expression of the form 
	 \begin{equation}\label{eq:rep_nuhalf_h}
	 \frac12 \abs{\beta} \left( \AAb{\eta}\BB_\eta\right)^{-\frac14}
	 h_{\nu'} \left( \left( 2\pi\sqrt{ \AAb{\eta}\BB_\eta} \right) ^{-1}\right)
	 \exp\left( -  \frac{2\pi}{\abs{\beta}^2} \left( \AAb{\eta}\BB_\eta\right)^{\frac12}\right),
       \end{equation}
with  $\nu' = \abs{\nu}-\frac12$. Likewise, in this case, the Bessel functions occurring in the expansion of $\mathcal{V}_{p+q, 1-\nu}$ in  \eqref{eq:phi2pq_nh} can be replaced by a similar expression, with $\nu' = \abs{\nu + r} -\frac12$ and $\AAb{\eta}$ replaced by $\AAb{\eta} + 4\abs{n}\abs{\beta}^2 - 2 n\abs{\beta}^2$.

For the rank 1 terms, if $\frac12\left(\ell- n_\gamma - n_\delta \right) \equiv \frac12 \pmod{1}$,  the Bessel functions in the holomorphic term $\phi_1^{\underline\gamma, \underline\delta}(n, \eta)^+$  can be replaced by expressions of the form
         \[
	 \frac12 \abs{\beta'}^{-\frac12} \left( 2\hlf{x_{0,-}}{x_{0,-}} \right)^{-\frac14}
	 h_{\nu'}\left( \left(2\sqrt{2}\pi\abs{\beta' } \abs{\hlf{x_{0,-}}{x_{0,-}}}^{\frac12} \right)^{-1} \right),
       \]
       with $\nu' = \abs{\nu} - \frac12$. Similarly for the Bessel function in the expansion of $\mathcal{V}_{p+q, 1-\nu}$ of the non-holomorphic terms, but with $\nu' = \abs{\nu + r}  - \frac12$ and $2\hlf{x_0}{x_0}$ shifted by $4\abs{n}$.
\end{remark}

\begin{example}\label{ex:sig_p1}
The lift for signature $(p,1)$ presents an interesting example. Recall that in this signature the Schwartz form is given by
 \[
\widehat\psi_{p,1}(\sqrt{2}(\eta, x_0), \tau) = 
2i \exp\left(-\frac{2\pi}{v}\left[  \abs{\beta' + \bar\tau\beta }^2 + 2v 
\Im\left(\beta' \bar\beta\right) \right]  \right) e^{2\pi i \tau 
	\hlf{x_0}{x_0}} \otimes 1, 
 \]
 so the polynomials and the differential form are trivial.  
 It is thus not surprising that the expressions from Theorem \ref{thm:Phipq_atz0}  simplify considerably.

 For the non-singular terms, the contribution to the lift of  of $f \in \HmfLp{1-p}{L^-}$  due to the holomorphic part $f^+$ is given by  (for fixed $\eta$ and $n$)
  \[
   \begin{aligned}
   \phi_2(n,\eta)^+(z_0)  & = \sqrt{2} \left( \frac{\AAb{\eta}}{\BB_\eta} \right)^{+\frac14}K_{-\frac12} \left( \frac{2\pi}{\abs{\beta}^2} \left({\AAb{\eta}\BB_\eta}\right)^{\frac12}\right)
   e^{-2\pi  i \CC_\eta\left( \hlf{x_0}{x_0} + n\right)} \\
   & = \frac{\sqrt{2}}{2}\abs{\beta} {\BB_\eta}^{-\frac12}
    \exp\left(
      - \frac{2\pi}{\abs{\beta}^2} \abs{\AAp{\eta}}
      \left( \tfrac{1}{2} \BB_\eta \right)^{\frac12} \right)
    e^{-2\pi  i \CC_\eta\left( \hlf{x_0}{x_0} + n\right)},
 \end{aligned}
\]
using \eqref{eq:bessel_12}. Note that here $\AAb{\eta} = \frac12 \AAp{\eta}^2$. The contribution of the  non-holomorphic part $f^-$ (for fixed $\eta$ and $m$) is given by
\[
    \phi_2(n,\eta)^{-} (z_0)  =
    \frac{\sqrt{2}}{2}
    \Vint{\kappa +2, \frac{3}{2}}{\frac{\pi\AAb{\eta}}{\abs\beta^2} - 2\pi n,  \frac{\pi \BB_\eta}{\abs{\beta}^2}, 4\abs{n}\pi}  e^{-2\pi  i \CC_\eta\left( \hlf{x_0}{x_0} + n\right)}.
\]
Since the index $\nu = \frac12$ is a half-integer, after expressing $\mathcal{V}_{\kappa + 2, 3/2}$ as a sum, we can use Bessel polynomials to write $\phi_2(n,\eta)^{-}$ in the form
\[
  \begin{aligned}
  \phi_2(n,\eta)^{-} & (z_0)  =
 {\sqrt{2}} (p-1)! \sum_{r=0}^{p-1} \frac{(4\pi\abs{n})^r}{r!}
\left( \AAb{\eta} + 2\abs{\beta}^2\left( 2\abs{n} - n\right) \right) ^{-\frac{r}{2}} \BB_\eta^{\frac{r-1}{2}} 
\\
 & \qquad \cdot\, h_{\max(0,r-1)}\left(\left[ \frac{2\pi}{\abs{\beta}^2} \left( \left( \AAb{\eta} + 2\abs{\beta}^2\left( 2\abs{n} - n\right) \right) \BB_\eta\right)^{\frac12} \right]^{-1}\right) \\
 &\qquad \cdot \exp\left( - \frac{2\pi}{\abs{\beta}^2} \left( \left( \AAb{\eta} + 2\abs{\beta}^2\left( 2\abs{n} - n\right) \right) \BB_\eta\right)^{\frac12}\right)
e^{-2\pi  i \CC_\eta\left( \hlf{x_0}{x_0} + n\right)}.
  \end{aligned}
\]
The rank 1 terms $\phi_1(n, \eta)^+$ are easily calculated directly. For a \lq one-line\rq\ version of the calculations in Section \ref{subsec:unfold1}, note that the domain of integration is given by $0\leq u \leq 1$ and $0\leq v <\infty$. Thus, the integral over $u$ just picks out the constant term of the Fourier expansion and terms are non-zero only if $\hlf{x_0}{x_0} = -n$ (note that $x_0$ is positive definite). Then, recalling that $\eta = [0, (a,b)^t]$, the integral over $v$ is given by
\begin{equation}\label{eq:s_dep_trick}
  \begin{gathered}
  \phi_1(n,\eta)^+(z_0) 
   =   \CT_{s=0}\left[  \int_{0}^\infty \!\exp\left( -\frac{2\pi}{v} \abs{\beta'}^2\right) v^{-s-2} dv \right] \\
 =   \left({2\pi\abs{\beta'}^2}\right)^{-s-1} \Gamma(s+1)\bigg\vert_{s=0} 
  = \frac{1}{2\pi} \frac{1}{a^2 + b^2} \Gamma(1).
\end{gathered}
\end{equation}
For the contribution of the non-holomorphic part, one has  (for $n\neq 0$) :
\[
  \begin{aligned}
  \phi_1(n,\eta)^{-}(z_0) & =
 & (p-1)! \sum_{r=0}^{(p-1)}
 \frac{\pi^r}{r!} \left( 4 \abs{n}\right)^{\frac{r}{2} - \frac{1}{4}} \abs{\beta'}^{r - \frac{1}{2}}
 h_{r}\left(  \frac{1}{4\pi \abs{\beta'}\abs{n}^{\frac12}}\right)
 e^{ -  4\pi\abs{\beta'}\abs{n}^{\frac12}}.
\end{aligned}
\]
For $n=0$, one has $\phi_1(0,\eta)^-(z_0) =  (2\pi(a^2 + b^2))^{-(p+1)} \Gamma(p+1)$,
from \eqref{eq:int_gamma} with $1-k = \kappa +1 = p$.

\end{example}
\begin{example} \label{ex:sig_q1}
  Another important example is the case of signature $(1,q)$. As we have seen in Proposition \ref{prop:FT_i}, the Schwartz form $\widehat\psi_{1,q}$ takes a fairly simple form,  as there is only one pair of multi-indices, given by $\underline\gamma = \underline\delta = (1, \dotsc, 1)$. Thus $n_\delta = n_\gamma = q-1$ and there  are no polynomials $P_{\tilde\gamma, \tilde\delta}(x_{0, +})$. Hence,
  \[
\Phi(z, f, \psi_{1,q}) =2\cdot \frac{2i(-1)^{q-1}}{2^{2(q-1)}} 
\Bigl( \sum_{i=1}^2 \sum_{\substack{ \eta =  [\beta, \beta'] / \sim\\  \operatorname{rank}(\eta) = i}}
\sum_{n\gg -\infty}\hat{a}^\pm(n)  \phi_i(n,\eta)^\pm   \Bigr)\otimes \Omega_{q-1}\bigl(\underline{1}; \underline{1}\bigr).
\]
We note that index $\nu$ from Theorem \ref{thm:Phipq_atz0} is half-integer and  Remark \ref{rmk:if_nuhalf} applies. Thus, the non-singular terms -- excluding the case $n=0$ for the non-holomorphic part  -- are given by  
\[
\begin{aligned}
  \phi_{2}(n,\eta)^{\pm} (z_0) &  =
  (-\pi)^{q-1}\frac{ \sqrt{2}}{2} \sum_{M=0}^{2(q-1)}  R_{q-1,q-1}(\eta; M) \sum_{j=0}^{\lfloor \frac{M}{2} \rfloor} \frac{1}{\pi^j} \abs{\beta}^{-2(M-j)} \\
  & \cdot \sum_{h = 0}^{M-2j} \frac{i^h}{2^{h +3j}}\AAm{\eta}^h ( - \CC_\eta)^{M-2j-h}  \abs{\beta}^{2h - 2j +1} \binom{M-2j}{h} \frac{M!}{j! (M-2j)!} \\ &  \cdot
  \exp\left(- \frac{2\pi}{\abs{\beta}^2} \left( \AAb{\eta}\BB_\eta\right)^{\frac12} \right)
e\left( - \CC_\eta\left( \hlf{x_0}{x_0} + n \right) \right) \\
 & \quad\cdot
\begin{cases}
 \begin{gathered}[b]
    \AAb{\eta}^{ -\frac{\nu}{2} - \frac14} \BB_\eta^{\frac{\nu}{2} - \frac14} 
h_{\nu'} \left( \frac{\abs{\beta}^2}{2\pi \left( \AAb{\eta}\BB_\eta\right)^{\frac12}}  \right)  
\end{gathered}
& \text{for $\phi^+$},
\\
\begin{gathered}[b]
(q-1)! \sum_{r=0}^{q-1}   \frac{(4\pi\abs{n})^r}{r!}
  \AAb{\eta}^{ -\frac{\nu + r}{2} - \frac14} \BB_\eta^{\frac{\nu+ r}{2} - \frac14} 
h_{\nu''} \left( \frac{\abs{\beta}^2}{2\pi \left( \AAb{\eta}\BB_\eta\right)^{\frac12}}  \right) 
\end{gathered}
& \text{for $\phi^{-}$}, 
\end{cases}
\end{aligned}
\]
wherein 
\[
R_{q-1,q-1}(\eta; M) = \sum_{\substack{0\leq \mu_1,\mu_2 \leq q-1 \\ \mu_1 + \mu_2 = M}}
\abs{\beta}^M \abs{\beta'}^{q-1} \Re\left(\beta'^{-\mu_1} \bar\beta'^{-\mu_2} \right) \binom{q-1}{\mu_1} \binom{q-1}{\mu_2}, 
\]
whilst the indices $\nu$, $\nu'$ and $\nu''$ are given by
\[
\nu = h + j - q - \frac12,  \qquad \nu' =  \abs{\nu} - \frac12, \qquad \nu'' = \abs{\nu + r} -\frac12. \\
\]
So, for example
\[
\nu'' =  
\begin{cases}
q -r - h - j -1  & \text{if}\quad q>h + j + r, \\
r + h + j - q & \text{otherwise.}
\end{cases}
\]
As before, for $n=0$, the terms for the non-holomorphic part (both rank 1 and rank 2) can be obtained from the respective holomorphic term after an index shift  by $-k+1$. 

For the rank 1 terms by the second part of Theorem \ref{thm:Phipq_atz0}, for $n\neq 0$,  one has  the holomorphic term
\[
\phi_1(n,\eta)^+(z_0) = (-\pi)^{q-1} \abs{\beta'}^{q-2} \left(2 \abs{n} \right)^{\frac{q}{2}} 
K_q\left( 2\sqrt{2}\pi \abs{\beta'}\abs{n}^\frac12 \right).  
\] 
Note that the index $\nu = -q$ and that $\hlf{x_{0,-}}{x_{0,-}} = -n$. The non-holomorphic term, again for $n\neq 0$, is given by
\[
\begin{aligned}
\phi_1(n,\eta)^-(z_0)  = &
2(-\pi)^{q-1} \abs{\beta'}^{q-2}(q-1)! \sum_{r=0}^{q-1} 
2^{\frac{q-r}{2}} \abs{\beta'}^{r} \frac{\left( 4\pi\abs{n}\right)^r}{r!}    
\\
& \quad \cdot \left(\Qf{x_{0,-}}  +2 \abs{n} \right)^{\frac{q-1}{2}}  
	K_{q-r}\left(2 \pi \abs{\beta'} \sqrt{2}\sqrt{\Qf{x_{0,-}}  +2 \abs{n} }\right).
\end{aligned}
\]
Finally, for $n=0$, one has
\[ 
\begin{gathered}
\phi_1(0,\eta)^+(z_0) = 
 (-\pi)^{q-1} \abs{\beta'}^{q-2} \left(2 \abs{\Qf{x_{0,-}}} \right)^{\frac{q}{2}} 
K_q\left( 2\sqrt{2}\pi \abs{\beta'}\abs{\Qf{x_{0,-}} }^\frac12 \right) \\
\text{and} \quad
\phi_1(0,\eta)^-(z_0) = 
(-\pi)^{q-1} \abs{\beta'}^{2(q-1)} 
K_0\left( 2\sqrt{2}\pi \abs{\beta'}\abs{\Qf{x_{0,-}} }^\frac12 \right).
\end{gathered}
\]
\end{example}

\section{Determining the Fourier-Jacobi expansion}\label{sec:FJ_exp}
In this section, we determine the Fourier-Jacobi expansion of the lift $\Phi(z_0, f, \psipq)$ for $f\in \HmfLp{k}{L^-}$ by applying the action of the parabolic subgroup to the lift at the base-point using the intertwining operators from Lemma \ref{lemma:op_Pell}.

\subsection{Operation of the parabolic subgroup in \texorpdfstring{$G_L$}{G}} \label{subsec:op_on_Phipq} 

To begin, we study the operation of $P_\ell$ on $\Phi(z_0,f, \psipq)$. For this, keep in mind that the theta function $\Theta(\tau, z; \psipq)$, like in \citep[][]{FH21}, is formed using a factor of $\sqrt{2}$ in $\eta$ and $x_0$, which has to be taken into account in all exponential factors occuring in Lemma \ref{lemma:op_Pell}.

\paragraph{On rank 2 terms:}
Let us first consider the action on the non-singular terms. Recall the definitions of $\BB_\eta$, $\CC_\eta$ and $\AAb{\eta}$ from Notation \ref{not:def_ABC} above.

Clearly, $\BB_\eta$ and $\CC_\eta$ are invariant under the operation of $n(w,r)\in N$ and $\mu \in M$, as they do not depend on $x_0$.  Further, under the operation of $a(t) \in A$,  the expression $\CC_\eta$ is invariant, while $a(t) \BB_\eta = \BB_{t\eta} = t^4\BB_\eta$. The quotient $\BB_\eta^{\frac12}\abs{\beta}^{-2}$ is again invariant.
Quite contrastingly, we have
\[
  \begin{gathered}
  n(w,0)\AAb{\eta}(m,x_0) = \AAb{\eta}(n, x_0 -  \beta w) \quad  (n(w,0) \in N), \\
  a(t) \AAb{\eta}(n , x_0) = \AAb{t \eta}(m, x_0)  \; (a(t) \in A),\qquad
  \mu\AAb{\eta}(n, x_0) = \AAb{\eta}(n, \mu^{-1} x_0) \; (\mu \in M),
\end{gathered}
\]
and similarly for 
$\AAp{\eta}(n, x_0)$ and $\AAm{\eta}(n, x_0)$. Also recall that by part 3.\ of Lemma \ref{lemma:op_Pell} the entire expression has to be multiplied with $t^2$.

The operation of the translations $n(0,r) \in N$  is most easily described: The terms $\phi_2^{\underline{\gamma}, \underline{\delta}}(n, \eta)^\pm$ are just multiplied with a factor (see Lemma \ref{lemma:op_Pell}) of
\[
\exp\left((2\pi i r \Im \left({2\beta'\bar\beta}\right)\right) = \exp\left( -4\pi i ra\alpha \right).
\]
The operation of $n(0,w)$ is more complicated. It affects the polynomials  $P_{\tilde\gamma, \tilde\delta, \ell} (x_{0,+})$. and all factors containing either $\AAb{\eta}$ and $\AAm{\eta}$.
From the last  factor in \eqref{eq:pq_phi2_hol} and \eqref{eq:pq_phi2_nhol},  one has 
\[
  \begin{multlined}
  e\left(- \CC_\eta\left(n + \hlf{x_0 - \beta w}{x_0 - \beta w}\right)\right) = \\
  e\left(- \CC_\eta\left(n + \hlf{x_0}{x_0}\right)\right)
  \cdot e\left( -\CC_\eta\left[ \abs{\beta}^2\hlf{w}{w} - 2\Re\left( \beta \hlf{x_0}{w}\right)\right]\right). 
\end{multlined}
\]
Further, again by Lemma \ref{lemma:op_Pell}, the term $\phi_2^{\underline{\gamma}, \underline{\delta}}(n, \eta)^\pm$ gets a factor of
\[
e\left( \Re\left( 2\beta'\bar\beta\right) \tfrac12\hlf{w}{w} - 2\Re\left( \beta' \hlf{x_0}{w}\right)\right).
\]
Multiplying the two factors, we have
\begin{equation}\label{eq:opw_factor}
  \begin{gathered}
    e\left( \left[- \CC_\eta\abs{\beta}^2 + \Re\left( \beta'\bar\beta\right) \right] \hlf{w}{w} - 2\CC_\eta\Re\left( \beta\hlf{x_0}{w}\right) - 2\Re\left( \beta' \hlf{x_0}{w} \right) \right) \\
    =   e\left(
      \Re\left[
     \left( 2\Re\left( \beta'\bar\beta \right)\abs{\beta}^{-1}  - 2\beta' \right) \hlf{x_0}{w} 
 \right] \right)  \\
= e\left(\Re\hlf{x_0}{w} \left(2 \tfrac{ab}{a }-  2b\right) +2 \alpha \Im\hlf{x_0}{w}\right)  
=  e\left( 2 \alpha\Im\hlf{x_0}{w} \right).
  \end{gathered}
\end{equation}
since $\beta = a>0$ and  $\beta' = b + i\alpha$.

\begin{remark} \label{rmk:p1_groupop}
 For the case of signature $(p,1)$, studied in Example \ref{ex:sig_p1} on p.~\pageref{ex:sig_p1}, we get a somewhat simpler expression for the operation of the elements $n(w,0) \in N$.  Recall that in this signature $\AAb{\eta} = \frac12 \AAp{\eta}^2$. Thus, in the exponential, from \eqref{eq:opw_factor} we have (with $\beta = a$, $\beta' = b + i\alpha$)
 \[
        \exp\left( - \frac{2\pi}{\abs{\beta}^2} \abs{ \AAp{\eta}(x_0 - \beta w)}\left( \tfrac12 \BB_\eta  \right)^\frac12\right) e\left( - \CC_\eta\left( \hlf{x_0}{x_0} + n \right) +2\alpha\Im\hlf{x_0}{w}\right).
  \]
  Considering the first term, since $\BB_\eta = 2\alpha^2 a^2$ and $a>0$ one has
    \begin{gather} \notag
  \exp\left( \mp \frac{2\pi}{\abs{\beta}^2}\left( \tfrac12 \BB_\eta  \right)^{\frac12}
   \left\lvert \AAp{\eta}(x_0) - 2\Re\left( \beta\hlf{x_0}{w} \right) + \abs{\beta}^2\hlf{w}{w} \right\rvert\right)  =  \\ 
\label{eq:p_Aopdiff} 
   \exp\left( - \frac{2\pi}{\abs{\beta}^2} \abs{ \AAp{\eta}(x_0)}\left( \tfrac12 \BB_\eta  \right)^\frac12\right) \cdot \exp\left(-2\pi \epsilon' \frac{\alpha}{a} \left[ a^2 \hlf{w}{w} -2a \Re\hlf{x_0}{w} \right] \right), 
 \end{gather}
 where we have set $\epsilon' \vcentcolon= \operatorname{sign}(\alpha\AAp{\eta}(x_0- a w))$ (recall that $\alpha\neq 0$). 
 Collecting the second factor in \eqref{eq:p_Aopdiff} and the factor $e(-\CC_\eta (n + \hlf{x_0}{x_0} ) +  2\alpha \Im\hlf{x_0}{w})$ we get
\[
  \begin{lgathered}
  e\Bigl( i \epsilon' \alpha \left(a\hlf{w}{w} - 2 \Re\hlf{x_0}{w} \right) + 2\alpha\Im\hlf{x_0}{w} 
  - \frac{b}{a}\left( n + \hlf{x_0}{x_0}\right)\Bigr) \\
  = \begin{cases}
    e\bigl( -\frac{b}{a}\left( n + \hlf{x_0}{x_0}\right) + i\alpha\left( a\hlf{w}{w} - 2\hlf{x_0}{w}\right)  \bigr) & \quad (\alpha\AAp{\eta}(x_0- aw)> 0), \\
    e\bigl( -\frac{b}{a}\left( n + \hlf{x_0}{x_0}\right) + i\alpha\left(- a\hlf{w}{w} + 2\hlf{w}{x_0}\right)  \bigr) & \quad (\alpha\AAp{\eta}(x_0-aw)< 0).
  \end{cases} 
\end{lgathered}
\]
\end{remark}

\paragraph{On rank 1 terms:}
Now, we have
 $\eta = [0, \beta']$ with $\beta' = a + ib$. 
The elements $a(t) \in A$ and $\mu \in M$ operate as usual, so $a(t)\beta' = t\beta'$  and $\mu:  x_0\mapsto \mu^{-1}x_0$.  
The operation of the Heisenberg group is much simpler in this case:
 $n(0,r)$ operates trivially, while  $n(w,0)$ only operates through the multiplicative factor from Lemma \ref{lemma:op_Pell},
 given by
\[
e\left( -\Re\left( \sqrt{2}(a+ib) \hlf{\sqrt{2}x_0}{w}\right) \right) = e\bigl( -2\Re\left( \beta'\hlf{x_0}{w}\right)\bigr). 
\]

\subsection{The Fourier-Jacobi expansion}  \label{subsec:fjexp_pq}

We want to determine the Fourier-Jacobi expansion of $\Phi(z, f, \psi)$, i.e.\ an expansion of the form
\[
\Phi(z,f,\psipq) = \sum_{\kappa \in \Q} c_\kappa (\sigma, \Im\tau) e^{2\pi i \kappa \Re\tau_\ell},  
\]
where $\tau_\ell$ is a coordinate attached to the $\ell$-component. Let us assume that for $z=z_0$, we have $\Re\tau_\ell  = 0$. This can easily be realized through a suitable choice of coordinate in $\Dom$. 
\begin{example}\label{ex:siegel_domain}
	Consider the case of signature $(p,1)$. Here, the Grassmannian $\Dom$ consists of one-dimensional subspaces of $V$. For each element $z \in \Dom$, choose a representative of the form 
	$\mathfrak{z} = \ell' + i\tau_\ell\ell + \sigma$ with $\sigma \in W$. One thus obtains the Siegel domain model of $\Dom$, given by the set
	\[
		\Hll \vcentcolon = \left\{ (\tau_\ell, \sigma) \in \C\times\W\,;\; 2\Im\tau_\ell > \hlf{\sigma}{\sigma} \right\}.
	\]
	The base point in $\Hll$ is given by $(\tau_\ell, \sigma) = (i,0)$, see e.g.~\cite{Hof14,Hof17} for details. We will use the Siegel domain to study this special case in more detail later, see Section \ref{sec:p_one}. 
\end{example}
Consider $z = g z_0$ with $g = n(w,0) n(0,r) a(t) \mu \in G$, and write 
$\Phi(z, f, \psi)$ in the form\footnote{Actually, below we will modify this notation slightly by introducing rescaled coefficients $c_0'$ and  $c_\kappa'$.} 
\[
\Phi(g z_0, f, \psipq) =  c_0(t, w, \mu) + \sum_{\kappa \in \Q^\times} c_\kappa(t, w, r,\mu) e^{2\pi i \kappa r }, 
\]
as a Fourier-Jacobi expansion with $\Re\tau_\ell = r$. 
 
Further, only the non-singular part of the lift, $\Phi_2(z, f, \psi)$ transforms under the action of the center of $N$, while the rank 0 and rank 1 contributions are invariant. Hence, the constant term of the Fourier-Jacobi expansion is given by 
\[
  c_0(t, w, \mu) = \Phi_0(z, f, \psipq) + 2\Phi_1(z, f, \psipq),
\]
and all other terms, for $\kappa > 0$, come from $\Phi_2(z, f, \psipq)$. In this case, $\kappa \neq 0$ is given by (possibly a constant multiple of) $a\alpha$. Hence,
\begin{equation}\label{eq:fjcoeff_m}
  \begin{lgathered}
    c_\kappa(t, w,r, \mu)   = \\ \; 2\cdot \sum_{\underline\gamma, \underline\delta}\sum_b \sum_n \sum_\lambda 
    e\left(-\bar\lambda_2\beta'\right) \left[ a^+(\lambda, n)
      \mathcal{F}(\hat{g})\circ \phi^{\underline{\gamma}, \underline{\delta}}_2(n,\eta)^+  + 
  a^-(\lambda,n)   \mathcal{F}(\hat{g})\circ \phi^{\underline{\gamma}, \underline{\delta}}_2(n,\eta)^- 
\right] 
\end{lgathered}
\end{equation}
These terms, as well as the rank 1 contribution to $c_0(t,w,\mu)$, can be obtained by applying the group operation to the results from Theorem \ref{thm:Phipq_atz0}.
\begin{theorem}\label{thm:FJexp_pq}
 For $z \in \Dom$, denote by $g_z \in G$ an element with $g_z z_0 = z$, and let $t,w,r,\mu$ be the parameters of its  $NAM$ decomposition, i.e.\ $g_z = n(w,0)n(0,r)a(t)\mu$. Then, the Fourier-Jacobi expansion of the singular theta lift $\Phi(z, f, \psi)$  for a weak harmonic Maass form $f\in\HmfLp{k}{L^-}$ is given by
 \[
  \begin{gathered}
  \frac{2^{2(q-1)}}{2i(-1)^{q-1}} 
    \Phi( z, f, \psipq) = 
     \frac{2^{2(q-1)}}{2i(-1)^{q-1}} 
    \Phi(g_z z_0, f, \psipq)
    = c_0'(t,w, \mu) + \sum_{\kappa \in \Q^\times} c_\kappa'(t,w,\mu)
    e^{2\pi i \kappa r}, 
  \end{gathered}
\]
where for $\kappa\neq 0$ the Fourier-Jacobi coefficients can be written in the form
\begin{equation}\label{eq:fj_pq_rk2} 
  c_\kappa'(t, w,  \mu) = 2 \sum_{\substack{a, \alpha \\ a\alpha = \kappa}}   \sum_{\underline\gamma, \underline\delta}
\Bigl(
  \sum_{b} \sum_n \sum_\lambda A_\kappa^{\underline\gamma, \underline\delta}\left(n, \lambda, \left(\begin{smallmatrix} a & b  \\ 0 & \alpha \end{smallmatrix}\right)\right)(t, w,  \mu)e(- \bar\lambda_2 \beta')\Bigr) \otimes \Omega_{q-1}(\underline\gamma, \underline\delta),
\end{equation}
while the constant coefficient $c_0'(t,w,\mu)$ consists of a contribution of rank 1 terms, which can be written in the form
\begin{equation}\label{eq:fj_pq_rk1} 
      2\sum_{\underline\gamma, \underline\delta}\Bigl( \sum_{a, b} \sum_n \sum_\lambda
      B^{\underline\gamma, \underline\delta}\left(n, \lambda, \left(\begin{smallmatrix} a \\ b \end{smallmatrix}\right) \right)(t, w,  \mu)
      e(-\bar\lambda_2\beta') \Bigr) \otimes \Omega_{q-1}(\underline\gamma, \underline\delta)
\end{equation}
and, if $p>1$, a contribution of the $0$-orbit, which we omit. (However, see Corollary \ref{cor:FJ_exp_pone} for the case of signature $(p,1)$.)
The initial factor of $2$ in \eqref{eq:fj_pq_rk1} and \eqref{eq:fj_pq_rk2} is due to the trivial operation of $-1_2\in \Gamma$, as usual. 

The coefficients in \eqref{eq:fj_pq_rk1} and \eqref{eq:fj_pq_rk2} are given as follows\footnote{Since the notation is already quite heavy, we have suppressed the $\mu$-dependence here, writing $x_0$ instead of $\mu'^{-1} x_0$.}: 
For the rank 1 contributions to the constant term one has (as usual $\beta' = \left(\begin{smallmatrix} a \\ b 
\end{smallmatrix}\right)$) 
\begin{equation} \label{eq:mrk1_coeff}
\begin{aligned} 
  B^{\underline\gamma, \underline\delta} & (n,\lambda, \beta')  (t,w,\mu)  =  \\
 	  & \; \left(-i \sqrt{\pi}\right)^{n_\gamma + n_\delta} 
 	 t^{n_\gamma + n_\delta + 2}\beta'^{n\gamma}\bar\beta'^{n_\delta}
         \sum_{\ell = 0}^{2q - 2 -n_\gamma - n_\delta}
          2^{\frac{\ell}{2} + 1} P_{\tilde\gamma, \tilde\delta, \ell}(x_{0,+})  e\left( - \Re\left( \beta' \hlf{x_0}{w} \right)\right).
         \\ & \quad \cdot 
         \Bigl[ 
          a^+(\lambda,n) B^+_n(\beta'; t) + a^-(\lambda, n) B^-_n(\beta'; t) 
        \Bigr], 
 \end{aligned}
\end{equation}
where $B^+_n(\beta'; t)$ and $B^-_n(\beta', t)$  denote contributions of the holomorphic and the non-holomorphic terms, respectively. Setting
$\nu = \frac12 (\ell - n_\gamma + n_\delta)$, they are given by 
\[
  \begin{gathered}
    B^+_n(\beta'; t) =  t^\nu\abs{\beta'}^{\nu} \left(2 \abs{\hlf{x_{0,-}}{x_{0,-}} }\right)^{-\frac\nu2} 
    K_\nu\left( 2\sqrt{2}\pi t\abs{\beta'} \abs{\hlf{x_{0,-}}{x_{0,-}}}^\frac12 \right), \\
    B^-_n(\beta'; t) = \tfrac12 a(\lambda, n)^- \Vint{p+q, 1 - \nu}{2\pi\hlf{x_{0,-}}{x_{0,-}}, \pi t^2\abs{\beta'}^2, 4\pi \abs{n}} \quad(n\neq 0), \\
    B^-_0(\beta'; t) =  t^{\nu -k +1}\abs{\beta'}^{\nu-k+1}
          \left(2 \abs{\hlf{x_{0,-}}{x_{0,-}} } \right)^{\frac{k - \nu -1}{2}}  
          \cdot K_{\nu - k +1}\left( 2\sqrt{2} \pi t\abs{\beta'}
            \abs{\hlf{x_{0,-}}{x_{0,-}}}^{\frac12} \right).
  \end{gathered}
\]
The coefficients for $\kappa \in \Q^\times$, coming from the contributions of the rank 2 terms  are given as follows (with $ \eta = \left(\begin{smallmatrix} a & b  \\ 0 & \alpha \end{smallmatrix}\right)$):
\begin{equation}\label{eq:mrk2_coeff}
\begin{aligned} 
A_\kappa^{\underline\gamma, \underline\delta} & \left(n, \lambda,  \eta \right) = \\
 &	  \left( -i  \sqrt{\pi}\right)^{n_\gamma  + n_\delta} t^{n_\gamma + n_\delta +2} 
		 \sum_{\ell = 0}^{2q - 2 - n_\gamma - n_\delta} 2^{\frac{\ell + 1}{2}} 
		 P_{\tilde\gamma, \tilde\delta, \ell} (x_{0, +} - \beta w)   
		 \sum_{M=0}^{n_\gamma + n_\delta}  R_{n_\delta, n_\gamma}(\eta,M) \sum_{j=0}^{\left\lfloor \frac{M}{2} \right\rfloor}\frac{1}{\pi^j} 
		 \\ &  \cdot  \sum_{h = 0}^{M -2j}
		 \frac{i^h}{2^{h + 3j}}  
		 \AAm{t\eta}^h(x_0 - \beta w) \left( - \CC_\eta\right)^{M - 2j -h}  t^{-2h-2j}\abs{\beta}^{- 2h - 2 j}
		 \binom{M -2 j}{h} \frac{M!}{j! (M-2j)!}  \\
&             \cdot e\Bigl( - \CC_\eta\left(\hlf{x_0}{x_0}  + n\right)
                + 2\alpha\Im\hlf{x_0}{w}  \Bigr)  \cdot \Bigl[ a^+(\lambda,n) A^+_n(\eta; t, w) + a^-(\lambda, n) A^-_n(\eta; t, w)\Bigr], 
\end{aligned}  
\end{equation}
wherein $A^+_n(\eta; t, w)$ and $A^-_n(\eta; t,w)$ denote contributions which come from the holomorphic and non-holomorphic terms, and are given by
\[
\begin{gathered}
A^+_n(\eta; t, w) = t^{2\nu}\left(\frac{\AAb{t\eta}(x_0 - \beta w)}{\BB_\eta}\right)^{-\frac{\nu}{2}}
K_\nu \left( \frac{2\pi}{ \abs{\beta}^2 } \left( \AAb{t\eta}(x_0 - \beta w) \BB_{\eta}\right)^{\frac12} \right),  \\
A^+_n(\eta; t, w) = \tfrac{1}{2} \Vint{p+q, 1 - \nu}{\pi\left( \frac{\AAb{t\eta}(x_0 - \beta w)}{t^2\abs{\beta}^2} -2n\right), \frac{\pi}{t^2\abs{\beta}^2}\BB_{t\eta}, 4\abs{n}\pi} \quad (n\neq 0),\\
A^+_0(\eta; t, w) = t^{2(\nu - k+1)}\left(\frac{\AAb{t\eta}(x_0 - \beta w)}{\BB_\eta}\right)^{-\frac{\nu- k +1}{2}}
	 K_{\nu - k+ 1} \left( \frac{2\pi}{ \abs{\beta}^2 } \left( \AAb{t\eta}(x_0 - \beta w) \BB_{\eta}\right)^{\frac12} \right). 
\end{gathered}
\]
   \end{theorem}
\begin{proof}
After setting
	\[
	\begin{aligned}
          A_k^{\underline\gamma, \underline\delta}(n, \lambda, \eta)(t,w,r,\mu) &\vcentcolon =
          a^+(\lambda, n)
	\mathcal{F}(\hat{g})\circ \phi^{\underline{\gamma}, \underline{\delta}}_2(n,\eta)^+  + 
	a^-(\lambda, n)   \mathcal{F}(\hat{g})\circ \phi^{\underline{\gamma}, \underline{\delta}}_2(n,\eta)^- , \\
	 B^{\underline\gamma, \underline\delta}(n, \lambda, \eta)(t,w,\mu)   & \vcentcolon = 
	 a^+(\lambda, n) \mathcal{F}(\hat{g})\circ \phi_1^{\underline{\gamma}, \underline{\delta}}(n,\eta)^+  + 
	 a^-(\lambda, n) \mathcal{F}(\hat{g}) \circ \phi_1^{\underline{\gamma}, \underline{\delta}}(n,\eta)^-,
	 \end{aligned}
	\]
the results follow directly from Theorem \ref{thm:Phipq_atz0} and the consideration concerning the group operation on the terms $\phi_i(n,\eta)^{\pm}$ from the beginning of this Section (see pp.~\pageref{subsec:op_on_Phipq}), including the factor of $t^2$ from the intertwining operation of $a(t)$ from Lemma \ref{lemma:op_Pell}.3. 
\end{proof}

\begin{corollary}\label{cor:FJ_exp_pone}
  In signature $(p,1)$ the Fourier-Jacobi  expansion of $\frac{1}{2i}\Phi(z, f, \psi_{p,1})$ (which now is equal to $\Phi(z,f,\varphi_0)$) takes the form
  \[
    c_0'(t,w) + \sum_{\kappa \in \Q^\times} c_\kappa'(t,w) e^{2\pi i \kappa \Re\tau_\ell},
  \]
   where the constant term $c_0'(t,w)$ is given by
 \[
    \begin{aligned} 
      c_0'(t,w)  =  & \; 4\pi I_0
      +  2\cdot t^2 \sum_{\beta' = (a,b)} 
      \sum_\lambda  \biggl[
      a^-(\lambda, 0)  \frac{1}{\left( 2\pi t^2 \abs{\beta'}\right)^{p+1}} \Gamma(p+1) \;+ \\
     & \;  \sum_{n\neq 0} \Bigl( a^+(\lambda, n)  \frac{1}{2 \pi t^2\abs{\beta'}^2}
     +      a^-(\lambda, n) (p-1)! \sum_{r = 0}^{p-1} \frac{\pi^r}{r!} (4\abs{n})^{\frac{r}{2} - \frac{1}{4}} t^{r - \frac12} \abs{\beta'}^{r-\frac12}
     \\ & \qquad \cdot h_r\left( \frac{1}{4\pi t \abs{\beta'}\abs{n}^{\frac12}}  \right) e^{-4\pi t\abs{\beta'}\abs{n}^\frac12}\Bigr) \biggr]  
   \cdot e\Bigl( -  \lambda_2 \beta' - 2\Re\left(\beta' \hlf{x_0}{w}\right) \Bigr),
    \end{aligned}
  \]
  with a rational\footnote{In fact, $I_0$ is an integer for $f\in\MfLw{k}{L^-}$.} constant $I_0$, which can be evaluated using the methods of \cite{Bo98}, see \cite{K16}.
  The coefficients $c_\kappa'(t,w)$ $(\kappa>0)$ take the form  
  \[
      c_\kappa'(t,w)  = 2\cdot\sum_{a,b}\sum_m\sum_\lambda A_\kappa(n,\lambda, [\beta, \beta'])(t,w) e^{-2\pi i \lambda_2\beta'}.
\]
wherein 
\[
\begin{multlined}
   A_\kappa(n, \lambda, [\beta, \beta'])(t,w) = \left( a^+(\lambda, n) \frac{\sqrt{2}}{2} 
   \frac{ t\abs{\beta}}{\BB_\eta^{\frac12}} + a^-(\lambda, n) A^-_n(\eta; t, w) \right) \\
    \cdot
   \exp\left( - 2\frac{\pi}{\abs{\beta}^2}
\abs{\AAp{t\eta}(x_0 - \beta w ) }\left(\tfrac12\BB_\eta\right)^{\frac12}
-  2\pi i\bigl[  \CC_{\eta}\left(\hlf{x_0}{x_0} + n\right) + 2\alpha\Im\hlf{x_0}{w} \bigr]\right),
\end{multlined}
\]
with a non-holomorphic term $A^-_n(\eta; t,w)$, given by
\[
\begin{multlined}
\sqrt{2}
(p-1)!\sum_{r=0}^{p-1}
\frac{(4\abs{n}\pi)^r}{r!} t^{2r + 1 + 2} \abs{\beta} \BB_{ \eta }^{\frac{r-1}2 } 
\left(
\tfrac12 \AAp{t\eta}^2(x_0 - \beta w) -  2 t^2 \abs{\beta}^2 \left(2\abs{n} - n\right) \right)^{-\frac{r}{2}} \\
\cdot
h_{\max\{0,r-1\}}\left(
\frac{\abs{\beta}^2}{2\pi} 
\left(  \left( \tfrac12\AAp{t\eta}(x_0 - \beta w) + 2t^2\abs{\beta}^2( 2\abs{n} - n) 
	\right)   \BB_\eta
	\right)^{-\frac12}   
\right).          
 \end{multlined} 
\]
\end{corollary}
\begin{proof}
Follows directly from Example \ref{ex:sig_p1} and Remark \ref{rmk:p1_groupop} after taking into account the factors of $t^2$ from the intertwining operation of $a(t)$. 
\end{proof}

\paragraph{The case of signature $(1,q)$.}
In this case, as $p=1$, there is no contribution from the $0$ orbit, since $\widehat\psi_{1,q}((0,x_0), \tau) = 0$. In other words, $\Phi_0(z, f, \psi_{1,q}) \equiv 0$.
Hence, Theorem \ref{thm:FJexp_pq} gives the complete Fourier-Jacobi expansion in this case. Thus, with Example \ref{ex:sig_q1} we get the following.
\begin{corollary} \label{cor:FJ_oneq}
  In signature $(1,q)$, the singular theta lift of a weak harmonic Maass form $f \in \HmfL{k}{L^-}$ with Fourier coefficients $a^+(\lambda, m)$ and $a^-(\lambda, m)$ has the following Fourier-Jacobi expansion
  \[
     \frac{2^{2(q-1)}}{2i(-1)^{q-1}}  \Phi(z, f, \psi_{1,q}) =
      \left( c_0'(t,w) +
        \sum_{\kappa \in \Q^\times} c_\kappa'(t,w) e^{2\pi i \kappa r} \right)
      \otimes \Omega_{q-1,q-1}(\underline{1}, \underline{1}), \\
\]
    with  
    \[
    \begin{aligned}
        c_0'(t,w) &= 2\cdot\sum_{ab}\sum_{n}\sum_\lambda
        B^{\underline{1}, \underline{1}}
          \left(m,\lambda, \left( \begin{smallmatrix}a \\ b 
              \end{smallmatrix}
              \right) \right)(t,w)
          e\left(-\lambda_2 (a +ib)\right) \quad\text{and}\\
          c_\kappa'(t,w) & =
          2\cdot\sum_{\substack{a,\alpha \\ a\alpha = \kappa }} \sum_b \sum_m\sum_\lambda
          A^{\underline{1}, \underline{1}}_\kappa\left(m, \lambda, \left(\begin{smallmatrix} a & b \\ & \alpha\end{smallmatrix}\right) \right)(t,w)   e\left(-\lambda_2 (a +ib)\right) \quad(\kappa\neq 0). 
    \end{aligned}
  \]
 For the coefficients $B^{\underline{1}, \underline{1}}$ we destinguish between the cases $n\neq0 $ and $n=0$. In the case 
 \begin{multline*}
   B^{\underline{1}, \underline{1}}(n,\lambda, \left(\begin{smallmatrix} a \\ b 
 \end{smallmatrix}\right))(t,w,\mu) = 
 (- \pi)^{q-1}  \Bigl[ 
 a^+(\lambda, n) t^{q}\abs{\beta'}^{q-2}\left(2 \abs{n}\right)^{\frac{q}{2}} K_{q}\left(2 \sqrt{2} \pi \abs{\beta'} \abs{n}^{\frac12} \right)  \\  
+ a^-(\lambda, n)  t^{2(q-1)} \abs{\beta}^{2(q-1)} \Vint{q+1, 1- q}{2\pi \Qf{x_{0,-}}, \pi t^2\abs{\beta'}^2, 4\pi\abs{n}}
 \Bigr] e\left( - \Re\left(\beta' \hlf{x_0}{w} \right)\right), 
\end{multline*} 
while in the second case $(n=0)$, 
\begin{multline*}
B^{\underline{1}, \underline{1}}(0,\lambda, \left(\begin{smallmatrix} a \\ b 
\end{smallmatrix}\right))(t,w,\mu) = 
(- \pi)^{q-1} 2  t^{q} \abs{\beta'}^{q-2} e\left( - \Re\left(\beta' \hlf{x_0}{w} \right)\right) 
\\ \Bigl[ a^+(\lambda, 0) \abs{\Qf{x_{0,-}}} ^{\frac{q}{2}}
K_{q}\left(2^{\frac32}   \pi \abs{\beta'} \abs{\Qf{x_{0,-}}}^{\frac12} \right)
+ a^-(\lambda, 0)  t^{q} \abs{\beta}^{q} 
K_{0}\left(2^{\frac32}  \pi \abs{\beta'} \abs{\Qf{x_{0,-}}}^{\frac12} \right)
\Bigr]. 
\end{multline*} 

The coefficients $A^{\underline{1}, \underline{1}}_\kappa$ are given as follows (with $\eta = \left(\begin{smallmatrix} a & b \\ & \alpha\end{smallmatrix}\right)$):
 \begin{multline*} 
 A^{\underline{1}, \underline{1}}_{\kappa}\left(n,\lambda, \eta \right)  = (-\pi)^{q-1}  t^{2(q-1)} 
\frac{\sqrt{2}}{2}\,
\sum_{M=0}^{2(q-1)}  R_{q-1, q-1}(\eta,M)
\\  \;\cdot\, \sum_{j=0}^{\left\lfloor \frac{M}{2} \right\rfloor} \frac{1}{\pi^j} \sum_{h = 0}^{M -2j}
\frac{i^h}{2^{h + 3j}}  
\AAm{t\eta}^h(x_0 - \beta w) \left( - \CC_\eta\right)^{M - 2j -h}  t^{2h-2j}\abs{\beta}^{- 2h - 2 j}
\binom{M -2 j}{h} \frac{M!}{j! (M-2j)!}  \notag
\\
\cdot \Bigl[ a^+(\lambda, n) A^+_n( \eta; t,w)  + a^-(\lambda, n) A^-_n(\eta; t,w )  \Bigr] \cdot e\Bigl( - \CC_\eta\left(\hlf{x_0}{x_0}  + n\right) + 2\alpha\Im\hlf{x_0}{w} \Bigr),
\end{multline*}
where $A^+_n$, the contribution from the holomorphic part, takes the form  
\begin{multline*} 
A^+_n( \eta; t,w) =  t^{2\nu-1}\left({\AAb{t\eta}(x_0 - \beta w)}\right)^{-\frac{\nu}{2} - \frac14}
 {\BB_\eta}^{\frac{\nu}{2}-\frac14}
 \exp\left(-\frac{2\pi}{\abs{\beta}^2} \left( \AAb{t\eta}(x_0 - \beta w) \BB_\eta\right)^{\frac12}\right) \\
\cdot h_{\nu'} \left( \frac{\abs{\beta}^2 }{ 2\pi  \left( \AAb{t\eta}(x_0 - \beta w) \BB_\eta\right)^{\frac12} }\right), 
\end{multline*}
with indices $\nu = h+j+1-q$ and $\nu' = \abs{\nu} - \frac12  = \operatorname{max}(q-1-h-j, h+j-q)$. The contribution of the non-holomorphic part, $A^-_n$, is given by 
\begin{multline*}
A^-_n( \eta; t,w )  = 
(q-1)!
\sum_{r=0}^{q-1}
\frac{\left(4\pi\abs{n}\right)^r}{r!}
\left(\AAb{t\eta}(x_0 -\beta w)  + 2t^2\abs{\beta}^2 \left( 2\abs{n} - n \right) \right)^{-\frac{\nu}{2} - \frac{r}{2} -\frac14}
\\ \cdot \left(t^4\BB_\eta\right)^{\frac{\nu}{2} + \frac{r}{2} - \frac14}  
 \exp\left(-\frac{2\pi}{\abs{\beta}^2} 
 \left(\AAb{t\eta}(x_0 -\beta w)  + 2t^2\abs{\beta}^2 \left( 2\abs{n} - n \right) \right)^{\frac12} \BB_\eta^{\frac12}\right) \\
 \cdot h_{\nu''}\left( \frac{\abs{\beta}^2}{2\pi} 
 \left(\AAb{t\eta}(x_0 -\beta w)  + 2t^2\abs{\beta}^2 \left( 2\abs{n} - n \right) \right)^{-\frac12}
 \BB_\eta^{-\frac12} \right), 
\end{multline*}
with $\nu'' = \abs{\nu + r} -\frac12$ for $n\neq 0$. For $n=0$, we have 
\begin{multline*}
A^-_{0} (\eta; t,w) = \\
t^{2\nu-k+1}\left({\AAb{t\eta}(x_0 - \beta w)}\right)^{-\frac{\nu}{2} + \frac{k}{2} - \frac14}
{\BB_\eta}^{\frac{\nu}{2}- \frac{k}{2}-\frac14} 
\cdot
h_{\abs{\nu -k} -\frac12} \left( \frac{\abs{\beta}^2 }{ 2\pi  \left( \AAb{t\eta}(x_0 - \beta w) \BB_\eta\right)^{\frac12} }\right). 
\end{multline*}
\end{corollary}
\begin{proof}
	Follows immediately from Theorem \ref{thm:FJexp_pq} and Example \ref{ex:sig_q1}.  
\end{proof}

\section{Product expansions in signature \texorpdfstring{$(p,1)$}{(p,1)}}
\label{sec:p_one}
In this section, we take up the case of signature $(p,1)$ from Example \ref{ex:sig_p1}, Remark \ref{rmk:p1_groupop} and Corollary \ref{cor:FJ_exp_pone}, restricting to the lift of a weakly holomorphic modular form $f \in \MfLw{k}{L^-}$, with $k=1-p$. In particular $f \equiv f^+$. 
Let $f =\sum_{m\gg -\infty} \hat{\mathbf{a}}(n) q^n$ be the Fourier expansion of $f$ in the mixed model of the Weil representation, with $\hat{a}(n) = \sum_{\lambda \in L^\sharp / L} a(\lambda, n) e\left(- \beta'\cdot\lambda_2 \right)$ see \eqref{eq:ahatn}.

\subsection{The Fourier-Jacobi expansion}

First, we reexamine the Fourier-Jacobi expansion for $(2i)^{-1}\Phi(z,f,\psi_{p,1})$ from Corollary \ref{cor:FJ_exp_pone} with $f\in\MfLw{k}{L^-}$. 

\paragraph{Singular terms.}
First, we gather the contributions to the constant term $c_0'(t,w)$, i.e. for singular $\eta$. By the results of \cite{K16}, the contribution of the 0-orbit is given by $4\pi I_0$.  

The terms with $\operatorname{rank}\eta = 1$ are given as follows. The operation of $n(w,0)$  consists simply
of a factor $e\left( -2\Re\left( \beta'\hlf{x_0}{w}\right) \right)$, while $n(r,0)$ operates trivially.
One has
\begin{equation}\label{eq:p1_rk1_fj0} 
\begin{lgathered}
 2\sum_{\eta = (0,\beta')}
\sum_n \hat{a}(n) \left(\mathcal{F}(\widehat{g})\circ \phi_1(n,\eta)^+\right) 
\\ 
=  2t^2 \sum_{\beta' = (a,b)^t}\sum_n \sum_{\substack{\lambda \\ \lambda_1 = 0}}
a(\lambda, n) \frac{1}{2\pi t^2\abs{\beta'}^2}e\bigl( -\beta'\cdot\lambda_2  -2 \Re\left(\beta'\hlf{x_0}{w}\right) \bigr)  \\
=  \sum_{a,b}\sum_n \sum_{\substack{\lambda \\ \lambda_1 = 0}}
a(\lambda, n)\frac{1}{\pi (a^2+ b^2)} e\bigl(-a \left( 2\Re\hlf{x_0}{w} + \lambda_{21}\right)  + b\left( 2\Im\hlf{x_0}{w} + \lambda_{22}  \right)\bigr). 
\end{lgathered}
\end{equation}
\paragraph{Non-singular terms.} 
Now, for the terms with $\operatorname{rank}\eta=2$.
Let $g = n(w,r) \circ a(t) \in G$. 
From Corollary \ref{cor:FJ_exp_pone} and Remark \ref{rmk:p1_groupop}, we have
\begin{multline*}
 \hat{a}(n) \cdot  \bigl( \mathcal{F}\bigl(\widehat{g}\bigr)  \circ\phi_2( n, \eta)^+ \bigr) = \\ \sum_\lambda a(\lambda,n)
\frac{t\sqrt{2}\abs{\beta} }{ \BB_\eta^\frac12 } 
\exp\Bigl( 
	- 2\pi\frac{\abs{\Im(\beta'\bar\beta)}}{\abs{\beta}^2}\epsilon'\left( \AAp{t\eta}(x_0) + a^2\hlf{w}{w}  - 2a\blf{x_0}{w}  \right)   \\
	 + 2\pi i \Bigl( - \frac{1} {\abs{\beta}^2}  \Re(\beta'\bar\beta)\left(n  + \hlf{x_0}{x_0}\right) + 2 \alpha\Im\hlf{x_0}{w} \Bigr)\Bigr)  
	 e^{2\pi i \left( - \beta'\lambda +  2ra\alpha \right)}
	\end{multline*}
with a sign $\epsilon' \vcentcolon = \operatorname{sign}\left(\alpha\AAp{t\eta}\left(x_0 - aw\right)\right)$.
Now, $\beta'\cdot\lambda = \lambda_{22}\alpha + \lambda_{21} b$, $\left(\frac12 \BB_\eta\right)^{1/2} = \abs{\Im(\beta'\bar\beta)} = \abs{\alpha a}$, $\beta^2 = a^2$ and $\AAp{t\eta}(x_0) = n  + 2a^2 t^2 + \hlf{x_0}{x_0}$, we get 
\[
\begin{multlined}
\sum_\lambda a(\lambda, n)\frac{t}{\abs{\alpha}} 
\exp\left( - 2\pi \epsilon' \alpha\left( \frac{1}{a} \left( n  + \hlf{x_0}{x_0} \right)   + 2a t^2 + a\hlf{w}{w}  - 2\blf{x_0}{w} \right)  \right) \\ \cdot e\left(  - \lambda_{22} \alpha  - \lambda_{21} b  -  \frac{b}{a} \left(n  + \hlf{x_0}{x_0}\right) + 2\Im\hlf{x_0}{w} + 2ra\alpha  \right).
\end{multlined}
\]
Further, there are a number of non-vanishing conditions from the Fourier expansion and the transformation behavior of $f \in\MfLw{k}{L^-}$, see \cite{K16}. Namely, $\lambda_{12} =0$,  $x_0 \equiv \lambda_0 \mod{L}$ and $0< a \in \lambda_{11} + N\Z$, with some positive integer $N$. 
 Moreover, by the transformation behavior of $f$, for all non-vanishing terms, we have $n + \hlf{x_0}{x_0} + \lambda_{11} a \in \Z$. Through this last condition, since $b \in \Q \mod a$, we get 
\begin{equation} \label{eq:pp_term_alpha}
\begin{multlined}
\sum_\lambda a(\lambda, n)\frac{t}{\abs{\alpha}} 
\exp\Bigl( - 2\pi \epsilon' \alpha\left( \frac{1}{a} \left( n  + \hlf{x_0}{x_0} \right)   + 2a t^2 + a\hlf{w}{w}  - 2\blf{x_0}{w} \right) \\  + 2\pi i \alpha\left( -\lambda_{22} + 2\Im\hlf{x_0}{w}  + 2ra \right) \Bigr).
\end{multlined}
\end{equation}
This is the contribution of $\phi_2(n,\eta)$ to the Fourier-Jacobi expansion of the lift $\Phi(z,f,\psi_{p,1})$ with $f \in \MfLw{k}{L^-}$. It corresponds to the contribution of $\phi_2(n,\eta)^{+}$ for $f\in\HmfLp{k}{L^-}$ in Corollary \ref{cor:FJ_exp_pone}. After summing over all terms with $\alpha>0$ and $\AAp{\eta}>0$ one gets the expression for $c_{a\alpha}'$ in the following Proposition, while the expression for $c_0'$ is obtained from \eqref{eq:p1_rk1_fj0}.

\begin{proposition}\label{prop:FJ_pone_fshreek}
	The Fourier-Jacobi expansion of  $(2i)^{-1}\Phi(z,f,\psi_{p,1})$ for $f \in \MfLw{k}{L^-}$ is  given by
	\[
		c_0'(w) + 
		\sum_{\substack{a, \alpha \\ \alpha >0}} c_{a\alpha}'(w) 
		e\left(2a\alpha \left(r + it^2\right) \right).
	\]
	The coefficients take the form 
	\[
	\begin{aligned}
	          c_0'(t,w) &   = 4\pi I_0 \, +   \\  &   \sum_{a,b} \sum_{n} \sum_{\substack{\lambda \\ \lambda_1 = 0}} 
          a(\lambda, n) \frac{1}{\pi(a^2 + b^2)} 
	e\left( - a\left(2\Re\hlf{x_0}{w}  + \lambda_{21} \right)
	+ b\left( 2\Im \hlf{x_0}{w} + \lambda_{22}\right)  \right), \\
		c_{a\alpha}'(t,w) & = 
	 t \Biggl[ \sum_n   \sum_{\substack{\lambda \\ \lambda_{12} = 0 \\ \lambda_{11} \equiv a \bmod{N}}}  
	 \sum_{\substack{x_0 \in \lambda_0 + L \\ a \mid n + \hlf{x_0}{x_0} + \lambda_{21}a}}  \frac{a(\lambda, n)}{ \abs{\alpha}} \\ 
	 &  \qquad\cdot
	e\Bigl(
	- \lambda_{22}\alpha  +   i\alpha \left(a \hlf{w}{w} - 2\hlf{x_0}{w} \right )   + i \frac{\alpha}{a} \left( n  + \hlf{x_0}{x_0}\right) 
	\Bigr)  \Biggr]. 
	\end{aligned}
	\]
\end{proposition}

\paragraph{A \texorpdfstring{$\log\abs{\cdot}^2$}{log-squared} expression.}

Now, we recover $\Phi_2(z, f,\psi_{p,1})$ by summing \eqref{eq:pp_term_alpha} over $\eta$. First, we sum over all positive $\alpha$'s. 
Let's assume that $\AAp{t\eta}(x_0 - aw)>0$, hence in this case $\epsilon'>0$.
We get (omitting the factor ${t}$ and the sum over $\sum_\lambda$ for the time being)
\begin{equation*}
- a(\lambda,n)\cdot\log\left(  1 - e\Bigl( -\lambda_{22}   + 2ra +  i\left( a^{-1} \left( n  + \hlf{x_0}{x_0} \right)   + 2a t^2 + a\hlf{w}{w}  - 2\hlf{x_0}{w} \right)   \Bigr) \right).
\end{equation*}
Similarly, we sum over $\alpha <0$, again assuming $\AAp{t\eta}(x_0 - aw)>0$, whence $\epsilon'<0$, and get exactly the complex conjugate of the exponential i.e.
\[
-a(\lambda,n)\cdot\log\left(  1 - e\Bigl( +\lambda_{22}   - 2ra -  i\left( a^{-1} \left( n  + \hlf{x_0}{x_0} \right)   + 2a t^2 + a\hlf{w}{w}  - 2\hlf{w}{x_0} \right)   \Bigr) \right).
\]
Note that for a fixed value of $t$, the positivity condition $\AAp{t\eta}(x_0 - aw)>0$ holds for all but (at most) finitely many terms. Denote by $a_0$ the smallest occurring $a$ and set $n_0 \vcentcolon= \max\{ \abs{n}; \, n<0, a(n,\lambda)\neq 0 \}$. Then, the positivity condition is certainly satisfied if $t^2 > \frac{n_0}{a_0^2}$.
Conversely, for all $t$ for which the condition holds, we get the total contribution of the non-singular terms (recall the factor $2$ to the contribution of $\Phi_2$ to $\Phi$)
\begin{equation}\label{eq:log_prod_mp} 
\begin{aligned} 
\tfrac{2}{2i}\Phi_2 & (z,f,\psi_{p,1})   = -t \sum_{\substack{n \\ \lambda \in L^\sharp/L }} \sum_{\substack{ a \in \lambda_{11} + N\Z }} \sum_{\substack{x_0 =  \lambda_0 + D \\ a \mid n + \hlf{x_0}{x_0} + a\lambda_{21}}}  a(\lambda,n) \\
&\; \cdot\log\left\lvert  1 - e\left( 2a (r + it^2) +\lambda_{22} +  i\left( a^{-1 } \left( n  + \hlf{x_0}{x_0} \right)    + a\hlf{w}{w}  - 2\hlf{x_0}{w} \right)   \right) \right\rvert^2.
\end{aligned} 
\end{equation}
 From this, we can recover an infinite product expansion, see Lemma \ref{lemma:Bop_nonsingular} below.

\begin{remark}\label{rmk:pos_condition}
	The positivity condition, explicitly given by
	\[
	\begin{aligned}
	0 < \AAp{t\eta}(x_0 - aw) &  = n + 2a^2 t^2 + \hlf{x_0 -a w}{x_0 - a w} \\
	& = n + 2a^2\left( t^2 + \tfrac12\hlf{w}{w} \right) + \hlf{x_0}{x_0} -2a\Re\hlf{x_0}{w}
	 \end{aligned}
	\]
	is also a sufficient condition for the non-vanishing of the real part of the argument of the exponent, and thus for the terms in \eqref{eq:log_prod_mp} to be non-singular.
\end{remark}

\begin{remark}
	We also remark that formally, summing over the at most finitely many terms with 
	$\AAp{t\eta}(x_0 - w)<0$ gives  terms with $\log\abs{\cdot}^2$ of the complex conjugate of the variables in \eqref{eq:log_prod_mp}. In Siegel domain coordinates, see \ref{ex:siegel_domain}, this would correspond to terms in $\bar{\tau}_\ell$, i.e.\ in a  generalized \lq lower half-plane\rq\  conjugate to the Siegel domain model.  
\end{remark}

\subsection{Infinite product expansions} \label{subsec:bo_products}
\paragraph{Siegel domain coordinates}\label{par:Siegeldom} 

Recall the construction of the Siegel domain model for $\Dom$ in the present signature $(p,1)$, see Example \ref{ex:siegel_domain} or see \cite{Hof14, Hof17} for details. 

We can recover the Siegel domain coordinates $\tau = \tau_\ell$ and $\sigma$ from the coordinates $r,w,t$ and $\mu$ arising from the $NAM$-decomposition of $G$ introduced in Section \ref{subsec:itw_sl}, as follows:
For $z \in \Dom$, denote by $\myz = \ell' + i\tau_\ell \ell + \sigma$ the usual representative attached to $(\tau_\ell, \sigma) \in \Hll$. The base point $z_0$ is then given by $\myz_0 = \ell' - \ell$, corresponding to $(i,0)\in\Hll$.

The operation of the subgroups $N$, $A$ and $M$  on $\Hll$ is given as follows: 
\begin{equation}\label{eq:Gop_onH}
 \begin{aligned}
N \ni (w,0):\; & (\tau_\ell, \sigma) \longmapsto (\tau_\ell  + i\hlf{w}{\sigma} +   i\tfrac{1}{2} \hlf{w}{w}, \sigma + w), \\
N \ni n(0,r):\; & (\tau_\ell, \sigma) \longmapsto (\tau_\ell + r, \sigma), \\
A \ni a(t):\; & (\tau_\ell, \sigma) \longmapsto (t^2 \tau_\ell , t\sigma), \\
M \ni \mu:\; & (\tau_\ell, \sigma) \longmapsto (\tau_\ell, \mu \sigma).
\end{aligned}
\end{equation}
Note that in this signature $M$ is compact. Also, 
 the $\ell'$-component of the vector $g(\mathfrak{z})$ is unchanged for all $g \in N, M$ and is multiplied by $t$ for $g = a(t) \in A$. Hence, there is a non-trivial automorphy factor $j(g,(\tau_\ell, \sigma))$ only if
 $g = a(t) \in A$, with $t\neq 1$. In which case,
 $j(a(t), \myz) = t^{-1}$. 

 We set $g(\tau_\ell,\sigma) \vcentcolon= g_{\mathfrak{z}} \vcentcolon= n(w,r) \circ a(t)$ with
\begin{equation}\label{eq:group_siegel_coords}
  w = \sigma, \quad r = \Re\tau_\ell \quad\text{and}
  \quad t^2 = \Im\tau_\ell - \frac12 \hlf{\sigma}{\sigma}.
\end{equation}
Then, $[g(\tau_\ell,\sigma) \myz_0] = [\myz]$, and we note that
 \[
   j\left(a(t), \myz\right) =  \left( \Im\tau_\ell   - \tfrac12\hlf{\sigma}{\sigma} \right)^{-\frac12}  = \abs{\tfrac12\hlf{\myz}{\myz}}^{-\frac12}.
 \]
\begin{remark}\label{rmk:groupaction_coords}
  One can consider  $\tau_\ell' \vcentcolon = \tau_\ell  - \tfrac{1}{2i}\hlf{\sigma}{\sigma}$
  instead of $\tau_\ell$,  and may view $(\tau_\ell', \sigma)$ (and the attached one-dimensional subspace $\C \mathfrak{z}(\tau_\ell', \sigma)$)  as the  \lq group-action\rq\ coordinate associated to $(\tau_\ell, \sigma)$ (and $\myz$, respectively).
\end{remark}

 Letting $g(\tau_\ell , \sigma)$ operate on $\Phi(z_0, f, \psi_{p,1})$ as usual, through intertwining, 
 and taking into account the automorphy factor, we recover $\Phi((\tau_\ell,\sigma), f, \psi_{p,1})$. 
 Hence, the following Corollary of Proposition \ref{prop:FJ_pone_fshreek}, which gives the Fourier-Jacobi expansion in Siegel domain coordinates. 
 
\begin{corollary}\label{cor:fj_siegel_dom}
In signature $(p,1)$ using the Siegel domain coordinates $\tau_\ell$ and $\sigma$, the lift of $f\in\MfLw{k}{L^-}$ has a Fourier-Jacobi expansion of the form 
 \[
\tfrac{1}{2i}\Phi((\tau_\ell, \sigma), f, \psi_{p,1} )
= 
 c_0'(\sigma) + 
\sum_{\kappa} c_{\kappa}'(\sigma) \cdot e\left(\kappa \tau_\ell\right), 
\]
wherein $\kappa = 2a\alpha$. The coefficient of the constant term is given by 
\[
c_0'(\sigma) = 
4\pi I_0 +
 \sum_{\beta' = (a,b)^t }\sum_n\sum_{\substack{\lambda \\ \lambda_1 = 0}}
\frac{ 2 a(\lambda, n)}{\pi \abs{\beta'}^2 \abs{\hlf{\myz}{\myz}}}
e\left( - \Re\bigl( \bar\beta' \left(\hlf{x_0}{w} + \lambda_2\right)\bigr)\right),
\]
while the coefficients $c_\kappa'(\sigma)$ for $\kappa \neq 0$ take the form
\[
\begin{aligned} 
c_\kappa'(\sigma)  =   \sum_{a} & \, \frac{2a}{\kappa} \biggl[ \sum_{n}
\sum_{\substack{\lambda \\ \lambda_{12} = 0 \\ \lambda_{11} \equiv a \bmod{N}  }}
\sum_{\substack{x_0 \in \lambda_0 + D\\ a \mid n + \hlf{x_0}{x_0} + \lambda_{21}a }}
\\ 
& \qquad 
a(\lambda, n) \cdot e\left( - \frac{\kappa}{2a} \left(  \lambda_{22}
-2 i\hlf{x_0}{w} \right) +
i\frac{\kappa}{2 a^2}\left( n + \hlf{x_0}{x_0}\right)  \right)\biggr].
\end{aligned}
\]
\end{corollary}

\paragraph[Borcherds forms]{Borcherds forms and infinite product expansions.}
For the following definition see Borcherds \citep[][Theorem 13.3 3.]{Bo98} and recall that $\psikm_{p,1} = 2i \varphi_0$.  

\begin{definition}\label{def:psi_bo}
We define the Borcherds form $\Psi(\myz, f)$ through the relation 
\begin{equation} \label{eq:def_psi_bo} 
\frac{1}{2i} \Phi( \myz, f, \psi_{p,1}) = \Phi(\myz, f, \varphi_0) = -  2\log\abs{\Psi(\myz,  f)}^2
- a(0,0)\left( \log\abs{\hlf{\myz}{\myz}} + \log(2\pi) - \gamma \right),
\end{equation}
where $\gamma = - \Gamma'(1) = 0.57721\dots$ is the Euler-Mascheroni constant.
\end{definition}
\begin{remark}\label{rmk:psi_automorphic}
  Since Definition \ref{def:psi_bo} is compatible with that of Borcherds, 
  the Borcherds form is automorphic on $\Dom$. This follows from the results in \cite{Hof11, Hof14} through the pullback under the embedding constructed there.
\end{remark}

Now from \eqref{eq:log_prod_mp} 
we find the contribution of the non-singular terms to $\Psi(\myz, f)$:

\begin{lemma}\label{lemma:Bop_nonsingular}
Let $f \in \MfLw{k}{L^-}$ be a weakly holomorphic modular form with Fourier expansion $f = \sum_{\lambda \in L^\dual/L} \sum_{n\gg -\infty} a(\lambda,n) q^n \ebase_\lambda$. Then, the contribution of the non-singular terms of $(2i)^{-1}\Phi(\myz, f, \psi)$ to $\Psi(\myz, f)$ has the following infinite product expansion 
\[
\begin{gathered}
\Psi_2(\myz, f) = \prod_{\lambda \in L^\sharp/L} \prod_{n \in \Z - \Qf{\lambda}} \prod_{\substack{x_0 =  \lambda_0 + D \\  a \mid n + \hlf{x_0}{x_0} + a\lambda_{21}}}
		\Psi_{n,\lambda}(\tau, \sigma)^{\frac12 a(\lambda, n)}  \quad\text{with}\\
\Psi_{n,\lambda}(\tau,\sigma) = 
\Bigl(
   1 - 
   e\Bigl(-\lambda_{22} +2a\tau_\ell +  i\Bigl(   
   - 2\hlf{x_0}{\sigma}   +  a^{-1} \left( n  + \hlf{x_0}{x_0} \right)\Bigl)    \Bigr).
\end{gathered} 
\]
The product is zero-free if 
\[
\hlf{\myz}{\myz} = 2\Im\tau_\ell - \hlf{\sigma}{\sigma} >  \frac{n_0}{a_0^2},
\]
with $n_0 = \max\{ \abs{n};\, n<0, a(\lambda, n)\neq 0\}$ and $a_0$ the smallest positive $a$ occurring in the system of representatives for full-rank matrices in $\mathrm{M}_2(\Q)$ from Lemma \ref{lemma:Ks_reprs}. The product is absolutely convergent for all $\myz$ of sufficiently large norm.
\end{lemma}
\begin{proof}
  Recall that the total contribution of the rank 2 terms to $\Phi(\myz, f, \psi_{p,1})$ is given by $2\Phi_2(\myz,f,\psi_{p,1})$. The product expansion now follows directly from \eqref{eq:log_prod_mp} and Definition \ref{def:psi_bo} by using the Siegel domain  coordinates from \eqref{eq:group_siegel_coords} and the automorphy factor $j(a(t), \myz) = t^{-1}$.  
  
  By Remark \ref{rmk:pos_condition}
the product is non-zero if 
\[
  \begin{gathered}
0< \AAp{t\eta}(x_0 - a\sigma)
= \left( n + 2a^2 t^2 + \hlf{x_0 - a\sigma}{ x_0 - a\sigma} \right) \\
= \left( n + 2a^2\Im\tau_\ell + \hlf{x_0}{x_0}  - 2a\Re\hlf{x_0}{\sigma}  \right).
\end{gathered}
\] 
Since $x_0 -\sigma$ is positive definite, this condition is clearly satisfied if $t^2 > \frac12 \abs{n}a^{-2}$ for all occurring $n$ and $a$. 

Hence, this condition describes the zero-free region as a neighborhood of the cusp $[\ell]$, in  which the region of absolute convergence is contained, of course. In turn, this region contains a suitable neighborhood of the cusp, which then can be described by an inequality of the form $\hlf{\mathfrak{z}}{\mathfrak{z}} > \epsilon^{-1}$  $(\epsilon >0)$, see e.g.~\cite{Hof14, Hof17} for details.
\end{proof}
\begin{remark}
  A more explicit description of the domain of absolute convergence can be obtained using the  asymptotic growth formula for the Fourier coefficients of $f$ \citep[see][Sec.~4.2]{K16}, i.e. $\abs{a(\lambda, n)} = O(e^{2\pi c_f\sqrt{m}})$ with a positive constant $c_f$ depending on $f$. Arguing as in loc.~cit.~one finds $\hlf{\myz}{\myz} > n_0 + c_f^2$.   
\end{remark}
Now, for the contribution of the rank 1 terms. 
\begin{lemma}\label{lemma:Bprod_rank01}
  The contribution of the rank 1 and rank 0 terms to the infinite product expansion attached to the lift of a weakly holomorphic modular form $f \in \MfLw{k}{L^-}$ with Fourier expansion $f = \sum_{n,\lambda} a(\lambda, n) q^n \ebase_\lambda$ is given by
  \[
    \begin{gathered}
    e^{2\pi I_0} \eta(i)^{a(0,0)} 
    \prod_{n}\prod_{\substack{\lambda \\ \lambda_1 = 0 }} 
    \left(
    \prod_{\substack{x_0 \in D + \lambda_0 \\ \Qf{x_0} = n}}
    \frac{\vartheta_1 (-i(2 \hlf{x_0}{ \sigma} + \lambda_2), i)}{\eta(i)}
    e^{-\pi C_0(\sigma)^2}\right)^{\frac12 a(\lambda, n)},
   \end{gathered}
\]
where $\eta(z)$ $(z\in\Hp)$ is the Dedekind eta function, $C_0(\sigma) =  - \Re\hlf{x_0}{\sigma}-\lambda_{21}$ and $I_0$ is given by 
\[
-\sum_n \sum_{\substack{\lambda \\ \lambda_1 =0}} \sum_{x_0 \in D + \lambda_0}
  a(\lambda, -n) \sigma_1(n - \Qf{x_0}).
\]
\end{lemma}
\begin{proof}
The rank 1 terms can be worked out as in \citep[][Section 4.3]{K16}. We sketch part of the argumentation.
  The total contribution of rank 1 terms to $\Phi(\myz, f,\psi)$ is given by \eqref{eq:p1_rk1_fj0}. Thus, on the left hand side of \ref{eq:def_psi_bo}  we have
  \begin{equation}\label{eq:phi1_sum}    
    \sum_{n,\lambda} a(\lambda, n) \sum_{a,b} \frac{1}{\pi} \frac{1}{\abs{a + ib}^2}  e\left(C_0 a + C_1 b \right),  
  \end{equation} 
  with $C_0 = -2\Re\hlf{x_0}{\sigma} - \lambda_{21}$ and $C_1 = + 2\Im\hlf{x_0}{\sigma} + \lambda_{22}$.
We treat the case that $C_0$ and $C_1$ are not both zero. Here, one applies the second Kronecker limit formula \citep[see][(39) on p. 28]{STata65}
   in the form 
   \[
   	\sum_{\substack{m,n \\ (m,n) \neq (0,0)} } e^{2\pi i (n C_0 + m C_1)} \frac{\Im z}{\abs{m z+ n}^2} 
   	= - \pi \log\abs{ \frac{\vartheta_1( C_1 - C_0 z, z)}{\eta(z)} e^{\pi i z C_0^2 } }^2    \quad(z\in\Hp).
      \]
 	Setting $z= i$, we get the contribution
  \[
    - \sum_{n, \lambda}a(\lambda, n)
    \log \abs{ \frac{\vartheta_1(C_1 - iC_0, i) }{\eta(i)} e^{-\pi C_0^2}}^2,  
  \]
  yielding, together with \eqref{eq:def_psi_bo}, the infinite product over theta values.
 
  The case where both $C_0$ and $C_1$ vanish, e.g.~with $n=0$ and $\lambda = 0$ (and thus with $x_0 = 0$), can be treated somewhat similarly through the the first Kronecker limit formula (after writing \eqref{eq:p1_rk1_fj0} with an s dependence), yielding the factor $\eta(i)^{a(0,0)}$. 
  
Finally, by \citep[][]{K16} (and \cite{Bo98}) the rank 0 terms contribute merely
\[
  e^{2\pi I_0} \quad\text{with}\;
  I_0 = -\sum_n \sum_{\substack{\lambda \\ \lambda_1 =0}} \sum_{x_0 \in D + \lambda_0}
  a(\lambda,-n) \sigma_1(n - \Qf{x_0}).
\]
\end{proof}

Now, let us summarize the results of this section  in the following theorem. 
\begin{theorem}
	Let $f \in \MfLw{k}{L^-}$ be a weakly holomorphic modular form with Fourier expansion $f = \sum_{\lambda \in L^\dual/L} \sum_{n\gg -\infty} a(\lambda,n) q^n \ebase_\lambda$. Then, the Borcherds form $\Psi(\myz, f)$ is an automorphic form on $\Dom$, which in a suitable neighborhood of the cusp associated to $[\ell]$ has an absolutely convergent infinite product expansion of the form 
	\[
          \Psi(\myz, f)
          = \Psi_0 \cdot\Psi_1(\sigma)\cdot \Psi_2(\tau_\ell, \sigma).
        \]
        Of the three factors, the second and third are infinite products and correspond to the contribution of the terms of rank $1$ with $n\neq 0$ and the terms of rank $2$, respectively. 
        The factors are given as follows.
        \begin{enumerate}
          \item
        The first factor is a constant depending on $f$ and given by 
	\[
		\Psi_0(f) = Ce^{2\pi I_0} \eta(i)^{a(0,0)}, \;\text{with}\quad 
	   I_0	= -\sum_n \sum_{\substack{\lambda \in L^\sharp/L\\ \lambda_1 =0}} \sum_{x_0 \in D + \lambda_0}
		a(\lambda, -n) \sigma_1(n - \Qf{x_0}),
              \]
             wherein $\eta(i)$ the value of the Dedekind eta-function in $i\in\Hp$ and $C$ a constant of absolute value one.
	\item
	The second factor $\Psi_1(\sigma)$ is given by 
	\begin{gather*}
		\Psi_1(\sigma) =  \prod_{n}\prod_{\substack{\lambda \in L^\sharp/L\\ \lambda_1 = 0 }} 
		\Bigl(
		\prod_{\substack{x_0 \in D + \lambda_0 \\ \Qf{x_0} = n}}
		\frac{\vartheta_1 (-i( 2\hlf{\sigma}{x_0} + \lambda_2), i)}{\eta(i)}
		e^{-\pi C_0(\sigma)^2}\Bigr)^{\frac12 a(\lambda, n)}, 
	\end{gather*} 
	wherein $\vartheta_1(w,z)$ $(w\in \C, z \in \Hp)$ denotes Jacobi's elliptic theta function, and $C_0(\sigma) = -2\blf{x_0}{\sigma} - \lambda_{21}$. 
\item
	The third factor $\Psi_2(\tau_\ell, \sigma)$ takes the form
	\begin{gather*}
	\begin{aligned}
		\Psi_2 & (\tau_\ell, \sigma)  = \prod_{\lambda \in L^\sharp/L} \prod_{n \in \Z - \Qf{\lambda}} \\ & \quad\cdot\prod_{\mathclap{\substack{x_0 =  \lambda_0 + D \\  a \mid n + \hlf{x_0}{x_0} + a\lambda_{21}}}}\quad
			\Bigl(
			1 - 
			e\Bigl(-\lambda_{22} +2a\tau_\ell +  i\Bigl(   
			- 2\hlf{x_0}{\sigma}   +  \frac{1}{a} \left( n  + \hlf{x_0}{x_0} \right)\Bigl)    \Bigr)^{\frac12 a(\lambda, n)}.
		\end{aligned}
	\end{gather*}
\end{enumerate}
      \end{theorem}
\begin{proof}
	By Remark \ref{rmk:psi_automorphic} it follows through the theory of embeddings $\Ug(V) \hookrightarrow \Og(V_\R)$  from \cite{Hof11, Hof14} that $\Psi(\myz, f)$ is an automorphic form. 
	We have already seen the factors of the infinite product expansion in the two Lemmas \ref{lemma:Bop_nonsingular} and \ref{lemma:Bprod_rank01}.
\end{proof}

\section{Calculation of the unfolding integrals}\label{sec:unfolding}
In this section, we evaluate the unfolding integrals for the theta lift, providing all the Lemmas used in the proof of Theorem \ref{thm:Phipq_atz0} above. We use the notation introduced in Section \ref{sec:pq_lift} (in particular, see Notation \ref{not:def_ABC}).

Also, recall that for the non-zero terms, a factor of $2$ occurs after unfolding, since $-1_2 \in \Gamma$ operates trivially, so that $\gamma$ and $-\gamma$ give the same contribution, see the definition of the $\Phi_i(z,f,\psipq)$ in \eqref{eq:def_sum_phi}.

\subsection{Non-singular terms} \label{subsec:unfold2}
First, we calculate the contribution of the terms where the matrix $\eta = [\beta, \beta']$ is non-singular. By Lemma  \ref{lemma:Ks_reprs} a set of representatives for these $\eta$ under the operation of $\SL_2(\Z)$  is given by the rational matrices 
\[
\begin{pmatrix}
a & b \\ 0 & \alpha 
\end{pmatrix} 
\in \mathrm{M}_2(\Q) \quad\text{with}\quad
a>0,\; \alpha \neq 0 \; \text{and}\; b \pmod{a\Z}.
\]
Since the stabilizer is trivial, for fixed $\eta$,  the unfolding integrals take the form  
\[
\begin{lgathered}
\sum_{n} \left( \hat{a}^{-}(n) \phi_2(n, \eta)^{-}(z_0) + \hat{a}^+ (n) \phi_2(n, \eta)^+(z_0) \right) = \\
= 
\sum_{n} \hat{a}^+(n) 
\int_{\mathbb{H}}^{reg} 
\widehat\psipq\left(\sqrt{2}(\eta, x_0), \tau, z_0\right)  e^{2\pi i n u} e^{-2\pi n v} v^{-s-2} du\, dv  \\
+  \sum_{n} \hat{a}^-(n) 
\int_{\mathbb{H}}^{reg} 
\sum_{\gamma \in \Gamma_\eta \backslash \Gamma} 
\widehat\psipq(\sqrt{2}(\eta, x_0), \tau, z_0)   e^{2\pi i n u} \Gamma(1-k, 4\pi \abs{n}v)  v^{-s-2} du \, dv.
\end{lgathered}
\]
Now, the inner integral, over $u$ ranging through $\R$, is simply a Fourier transform,
which we shall calculate first.
For the outer integral over $v$, ranging over $\R_{>0}$,
we will make use of the integral representations of Bessel functions (cf.\ Appendix \ref{sec:functions}). 

\subsubsection{Preliminary lemmas}
Let us first gather some lemmas. The first three will help us to calculate the Fourier transform.
\begin{lemma} \label{lemma:FT_exp_part}
	The Fourier transform 
	\[
	\int_{-\infty}^\infty e^{-\frac{2\pi}{v} \left( \abs{\beta'}^2 + \abs{\tau}^2 \abs{\beta}^2
		+ 2u\Re\left( \beta'\bar\beta\right)\right) +2\pi i u\hlf{x_0}{x_0} +2\pi v \left( \hlf{x_{0,-}}{x_{0,-}} - \hlf{x_{0,+}}{x_{0,+}} \right)} e^{2\pi i n u} du, 
	\]	
	with $n\neq 0$ is given by
	\[
	e^{+2\pi n v} \frac{v^{\frac12}}{\sqrt{2}\abs{\beta}}
	\exp\left( -\frac{v \pi}{\abs{\beta}^2}\AAb{\eta} -  \frac{ \pi}{v \abs{\beta}^2}  \BB_\eta \right) e\bigl( -\CC_\eta\left( \hlf{x_0}{x_0}  + n \right)\bigr),
	\]
	with $\AAb{\eta} = \frac12 \left(n + \abs{\beta}^2 + \hlf{x_0}{x_0} \right)^2 - 4\hlf{x_{0,+}}{x_{0,+}}\abs{\beta}^2$,
	$\frac12 \BB_\eta =  \abs{\beta'}^2 \abs{\beta}^2 - \Re\left(\beta'\bar\beta\right)^2 $
	and $\CC_\eta =  \Re\left(\beta'\bar\beta \right)\abs{\beta}^{-2}$  from \eqref{eq:def_A_eta} and \eqref{eq:def_BC}.
\end{lemma}
\begin{proof}
	With \eqref{eq:FT_Bo} setting
	\begin{equation}\label{eq:ABC_forFT}
	\begin{gathered}
	A = \frac{i\abs{\beta}^2}{v},
	\quad B = \frac{2i}{v} \Re(\beta'\bar\beta) + \hlf{x_0}{x_0}\quad\text{and} \\
	C = i\Bigl[\frac{\abs{\beta'}^2}{v} + v\abs{\beta}^2   +  v\left( \hlf{x_{0,+}}{x_{0,+}}
	- \hlf{x_{0,-}}{x_{0,-}}\right) \Bigr],
	\end{gathered}
	\end{equation}
	one obtains the Fourier transform
	\[
	\begin{aligned}
	\frac{\sqrt{v}}{\sqrt{2}\abs{\beta}} 
	\exp&\left( 2 \pi i \left(  - \frac{n^2}{4A} - \frac{n B}{2A}  - \frac{B^2}{4A} +C \right) \right) 	=  \\
	\frac{\sqrt{v}}{\sqrt{2}\abs{\beta}} 
	\exp&\biggl(-\pi   v \left[\frac{n^2}{2\abs{\beta}^2} + \frac{n\hlf{x_0}{x_0}}{\abs{\beta}^2}  
	+  \frac{\hlf{x_0}{x_0}^2}{2\abs{\beta}^2} + 2\left( \hlf{x_{0,+}}{x_{0,+}} - \hlf{x_{0,-}}{x_{0,-}}\right)   + 2\abs{\beta}^2 \right]  \\ 
	& \qquad
	- \frac{\pi}{v}\left[ 2\abs{\beta'}^2- \frac{2\Re\left(\beta'\bar\beta\right)^2}{\abs{\beta}^2}\right] -2\pi i\frac{\Re(\beta'\beta)}{\abs{\beta}^2}\left( \hlf{x_0}{x_0} + n \right)
	\biggr).
	\end{aligned}
	\]
	After multiplying with $e^{-2\pi n v}$, one has
	\[
	\begin{gathered}
	\frac{\sqrt{v}}{\sqrt{2}\abs{\beta} }
	\exp\biggl( -\frac{v \pi}{\abs{\beta}^2} 
	\left[ \frac12 \left(n + 2\abs{\beta}^2 + \hlf{x_0}{x_0} \right)^2 - 4\hlf{x_{0,-}}{x_{0,-}}
	\abs{\beta}^2  - 2n\abs{\beta}^2 \right] - \pi v 2 n   \\
	-   \frac{2 \pi}{v \abs{\beta}^2} \left[
	\abs{\beta'}^2 \abs{\beta}^2 - \Re\left(\beta'\bar\beta\right)^2\right]
	\biggr)  \cdot
	e^{-2\pi i\frac{\Re(\beta'\bar\beta)}{\abs{\beta}^2}\left( \hlf{x_0}{x_0}  +  n \right)}\\
	= \frac{\sqrt{v}}{\abs{\beta} }
	\exp\left( -\frac{v \pi}{\abs{\beta}^2}\AAb{\eta} -  \frac{ \pi}{v \abs{\beta}^2}  \BB_\eta \right)
	e^{-2\pi i\CC_\eta\left( \hlf{x_0}{x_0}  + n \right)},
	\end{gathered}
	\]	
	as claimed.
\end{proof}

\begin{lemma}\label{lemma:FT_phik1k2}
	The Fourier transform
	\[
	\int_{-\infty}^\infty (\beta' + \bar\tau\beta)^{k_1} (\bar\beta' + \bar\tau\bar\beta)^{k_2}
	e^{ -\frac{2\pi}{v}\left(\abs{\beta'}^2  + \abs{\beta}^2\abs{\tau}^2 
		+ 2u\Re\left( \beta'\bar\beta\right)\right) + 2\pi i \left(\tau \hlf{x_{0,+}}{x_{0,+}} + \bar\tau\hlf{x_{0,-}}{ x_{0,-}}\right)}
	e^{2\pi i nu} du,
	\]
	is given by
	\[
	e^{ + 2\pi n v}\frac{\sqrt{v}}{\sqrt{2}\abs{\beta} }\tilde{p}_{\eta}(v,n)
	\exp\left( -\frac{v \pi}{\abs{\beta}^2}\AAb{\eta} -  \frac{ \pi}{v \abs{\beta}^2}  \BB_\eta \right)
	e^{-2\pi i\CC_\eta\left( \hlf{x_0}{x_0}  + n\right)},
	\]
	with a polynomial $\tilde{p}_\eta( v, n)$ of the following form
	\[
	\begin{aligned}
		&
	\sum_{M=0}^{k_1 + k_2} R_{k_1, k_2}(\eta, M)  
	\sum_{j = 0}^{\left\lfloor \frac{M}{2}\right\rfloor}
	\sum_{\kappa =0}^{M-2j} 
	\frac{i^\kappa v^{\kappa + j}\AAm{\eta}(n)^\kappa}{\abs{\beta}^{2\kappa + 2j}}
	\frac{\left( -\CC_\eta\right)^{M-\crampedrlap{2j-\kappa}} }{2^{\kappa + 3j} \pi^j} 
	\frac{M!}{j! (M -2j)!}\binom{M-2j}{\kappa},
	\end{aligned}
	\]
	wherein
	\[
	R_{k_1, k_2}(\eta, M) \vcentcolon =
	\sum_{\substack{ 0\leq \mu_1 < k_1\\ 0 \leq \mu_2 \leq k_2\\ \mu_1 + \mu_2 = M }} 
	\beta^{\mu_1}\bar\beta^{\mu_2} \beta'^{k_1 -\mu_1} \bar\beta'^{k_2 - \mu_2}
	\binom{k_1}{\mu_1}\binom{k_2}{\mu_2}\;.
	\]
	Note that if $k_1 = k_2 =k$, the coefficient $R_{k,k}(\eta, M)$ is real and can be written in the form
	\[
	\sum_{\substack{ 0\leq \mu_1, \mu_2 < k \\ \mu_1 + \mu_2 = M }}
	\Re\left( \beta^{\mu_1} \bar\beta^{\mu_2} \beta'^{k-\mu_1} \bar\beta'^{k-\mu_2} \right) \binom{k}{\mu_1} \binom{k}{\mu_2}. 
	\]
\end{lemma}
\begin{proof}
	We need  only calculate the contribution of the polynomial, 
	the rest follows from Lemma \ref{lemma:FT_exp_part}. 
	Rewrite the polynomial in the form
	\begin{gather*}
	\left(\beta' + \bar\tau \beta \right)^{k_1}   \left(\bar\beta' + \bar\tau \bar\beta \right)^{k_2} =
	\sum_{\mu_1 = 0}^{k_1} \sum_{\mu_2= 0}^{k_2} \bar\tau^{\mu_1 + \mu_2} 
	\beta^{\mu_1} \bar\beta^{\mu_2} \beta'^{k-\mu_1} \bar\beta'^{k-\mu_2} \binom{k_1}{\mu_1} \binom{k_2}{\mu_2} \\
	= \sum_{M = 0}^{k_1 + k_2}  \bar\tau^{M} R_{k_1, k_2}([\beta, \beta'], M), 
	\end{gather*}
	with $R_{k_1, k_2}([\beta, \beta'], M)$ as above.
	Setting $A$, $B$ and $C$ as in \eqref{eq:ABC_forFT} from the proof of Lemma \ref{lemma:FT_exp_part},
	by Lemma \ref{lemma:FT_Bo} one has to apply $\exp\left(\frac{v}{8\pi \abs{\beta}^2} \frac{d^2}{d u^2} \right)$ and obtains
	\[
	\sum_{M=0}^{k_1 + k_2} \sum_{j = 0}^{\left\lfloor \frac{M}{2}\right\rfloor}
	\left[\frac{v}{8\pi \abs{\beta}^2} \right]^j \bar\tau^{M-2j} \frac{M!}{j! (M -2j)!}
	R_{k_1, k_2}([\beta, \beta'], M),
	\]
	where, further $\bar\tau$ is to be replaced by
	\[
	- \frac{n +B}{2A} -iv = 
	\frac{iv}{2\abs{\beta}^2} \left( n  - \abs{\beta}^2 + \hlf{x_0}{x_0} \right)
	- \frac{\Re(\beta'\bar\beta)}{\abs{\beta}^2} = \frac{iv}{2\abs{\beta}^2} \AAm{\eta} - \CC_\eta. 
	\]
	Thus, the polynomial part of the Fourier transform is given by
	\[
	\begin{lgathered}
	\tilde{p}_\eta( v, n) =  
	\sum_{M = 0}^{k_1 + k_2} R_{k_1, k_2} (\eta; M) \sum_{j = 0}^{\lfloor \frac{ M}{2}\rfloor} \frac{1}{\pi^j}\frac{1}{ \abs{\beta}^{2( M -j)}}  \\  \quad\cdot\sum_{\kappa=0}^{M-2j} \frac{i^\kappa v^{\kappa +j}}{2^{\kappa + 3j}}   \AAm{\eta}^{\kappa} \left( - \Re\left(\beta'\bar\beta\right)\right)^{M - 2j - \kappa} \binom{M -2 j}{\kappa} \frac{M!}{j! (M-2j)!}.
	\end{lgathered}
	\]
	With the definition of $\CC_\eta$ this gives the claimed form. 
\end{proof}

\begin{lemma}\label{lemma:FT_mzero}
	The integral 
	\[
	\begin{multlined}
	\int_{-\infty}^\infty (\beta' + \bar\tau\beta)^{k_1} 
	(\bar\beta' + \bar\tau \bar\beta)^{k_2} e\bigl( u\hlf{x_0}{x_0}\bigr) \\
	\cdot \exp\left(-\frac{2\pi}{v} \left( \abs{\beta'}^2 + \abs{\tau}^2 \abs{\beta}^2
	+ 2u\Re\left( \beta'\bar\beta\right)\right)
	+2\pi v \left( \hlf{x_{0,-}}{x_{0,-}}
	- \hlf{x_{0,+}}{x_{0,+}} \right)\right) du, 
	\end{multlined}
	\]
	is given by
	\[
	\frac{v^{\frac12}}{\sqrt{2}\abs{\beta}} \tilde{p}_\eta(v,0)
	\exp\left( - v\frac{\pi}{\abs{\beta}^2}
	\AAb{\eta}(0)  - \frac{1}{v}\frac{\pi}{\abs{\beta}^2} \BB_\eta\right)
	e^{-2\pi i \CC_\eta\hlf{x_0}{x_0}},  
	\]
	where $\tilde{p}_\eta$ is the polynomial from Lemma \ref{lemma:FT_phik1k2},
	and $\AAb{\eta}(0)$ is given by
	\[
	\frac12 \left( 2\abs{\beta}^2 + \hlf{x_0}{x_0}\right)^2 - 4\hlf{x_{0,+}}{x_{0,+}}\abs{\beta}^2.\]
\end{lemma}
\begin{proof}
	The integral is a Fourier transform. 
	Arguing similarly to the proof of  Lemma \ref{lemma:FT_exp_part}, with
	\[
	A = \frac{i\abs{\beta}^2}{v},\qquad  B=\frac{2i}{v}\Re\left( \beta' \bar\beta\right)
	\]
	we get the exponential factor
	\[
	\begin{multlined}
	\exp\left( - \pi v\left[ \frac{\hlf{x_0}{x_0}^2}{2\abs{\beta}^2}
	- \abs{\beta}^2    +     \left( \hlf{x_{0,-}}{x_{0,-}} - \hlf{x_{0,+}}{x_{0,+}} \right)
	\right]  
	-  \frac{\pi}{v} \left[ \frac{\Re\left( \beta' \bar\beta\right)^2}{\abs{\beta}^2} - \abs{\beta'}^2\right] \right) \\ \cdot\,
	e \left( -\frac{\hlf{x_0}{x_0}}{\abs{\beta}^2} \Re\left( \beta' \bar\beta\right)
	\right).
	\end{multlined}
	\]
	The polynomial is calculated as in Lemma \ref{lemma:FT_phik1k2}, 
	up to the last step where, now, $\bar\tau$ is replaced by
	\[
	- \frac{\hlf{x_0}{x_0} - B}{2A} - iv = \frac{iv}{2\abs{\beta'}}\left(\hlf{x_0}{x_0} - \abs{\beta}^2\right) - \frac{\Re(\beta'\bar\beta)}{\abs{\beta}^2}
	= \frac{iv}{2\abs{\beta}^2 }\AAm{\eta}(n=0) - \CC_\eta.
	\]
\end{proof}
Once the inner integrals are evaluated using the previous lemmas,
the following two lemmas allow us to evaluate the outer integrals.
\begin{lemma}\label{lemma:bessel_parts}
	Let $\ell$ be an integer. We have the following identity
	\begin{align*}
	\frac{1}{\abs{\beta}}
	\int_{0}^{\infty} 
	v^{-s-\frac{1}{2} - \ell} 
	\exp\left( -\frac{v \pi}{\abs{\beta}^2}\AAb{\eta} -  \frac{ \pi}{v \abs{\beta}^2}  \BB_\eta \right) dv \\
	= 	\frac{2}{\abs{\beta}}  \left(
	\frac{\AAb{\eta}}{\BB_\eta} \right)^{\frac12 \left( s + \ell - \frac12 \right)} 
	K_{-s+\frac12 -\ell}\left( \frac{2\pi}{\abs{\beta}^2} \AAb{\eta}^{\frac12} \BB_\eta^{\frac12} \right).
	\end{align*}
	Now,  for an integer $k>0$ denote by $h'_k$ the Bessel polynomial of index $k$ and set $h_0' = 1$.
	Then, for $s=0$, we have if $\ell \leq 0$:
	\[
	\frac{\AAb{\eta}^{\frac{\ell-1}{2}}}{\BB_\eta^{\ell/2}} 
	h_{-\ell}' \left( \frac{\abs{\beta}^2}{ 2\pi \AAb{\eta}^{\frac12} \BB_\eta^{\frac12}}   \right)
	\exp\left( - \frac{2\pi}{\abs{\beta}^2} \AAb{\eta}^{\frac12} \BB_\eta^{\frac12} \right), 
	\]
	whereas, if $\ell > 0$, we have 
	\[
	\frac{\AAb{\eta}^{\frac{\ell-1}{2}}}{\BB_\eta^{\ell/2}} 
	h_{\ell-1}' \left( \frac{\abs{\beta}^2}{ 2\pi \AAb{\eta}^{\frac12} \BB_\eta^{\frac12}}   \right)
	\exp\left( - \frac{2\pi}{\abs{\beta}^2} \AAb{\eta}^{\frac12} \BB_\eta^{\frac12} \right).
	\]
\end{lemma}
\begin{proof}
	Recall that $\AAb{\eta}$ and $\BB_\eta$ are both positive (as $\beta$ and  $\beta'$ are both non-zero). 
	Thus, the first equality is immediate from \eqref{eq:int_bessel}, 
	while the second, for $s=0$ follows by \eqref{eq:bessel_n2}. 
	For the third, use $K_{-\nu} = K_\nu$ from \eqref{eq:bessel_12} and argue similarly.
\end{proof}

\begin{lemma}\label{lemma:int_gammabess_part}
	Let $\ell$ be an integer and let $\kappa = p+q -2$. Then, the value at $s=0$ of the integral
	\[
	\frac{1}{\abs{\beta}}\int_0^{\infty} v^{-s+\ell-\frac12 } \Gamma\left( \kappa - 1, 4\pi\abs{n}v\right) e^{+2\pi n v} \exp\left( - \frac{\pi}{\abs{\beta}^2} \left\{ v\AAb{\eta} + \frac1{v} \BB_\eta\right\} \right) dv,
	\]
	is given by
	\begin{multline*}
	\frac{2}{\abs{\beta}}     \kappa!  
	\left( \frac{\AAb{\eta} -  2 n\abs{\beta}^2 + 4n \abs{\beta}^2}{\BB_\eta} \right)^{\frac12 \left(  \ell - \frac12 \right)}  
	\sum_{r=0}^{\kappa}\frac{\left( 4\pi \abs{n}\right)^{r}}{r!} \\
	\; \cdot\, \left(  \frac{\AAb{\eta}  - 2n\abs{\beta}^2 + 4\abs{n}\abs{\beta}^2}{\BB_\eta} \right)^{-\frac{r} {2}}
	K_{\frac12 -\ell + r}\left( \frac{2\pi}{\abs{\beta}^2} \left(\AAb{\eta}  -2n\abs{\beta}^2 + 4\abs{n}\abs{\beta}^2  \right)^{\frac12} \BB_\eta^{\frac12}  \right)
	\\ 
	= \frac{1}{\abs{\beta}}
	\Vint{\kappa + 2, \frac12 - \ell}{\pi\left(\frac{\AAb{\eta}}{\abs{\beta}^2} - 2n \right),\frac{\pi \BB_{\eta}}{\abs{\beta}^2}, 4\pi\abs{n}}. 
	\end{multline*}
\end{lemma}
\begin{proof}
	Follows from Lemma \ref{lemma:bessel_parts} with $\AAb{\eta}$ shifted by $-2n\abs{\beta}^2$ and Lemma \ref{lemma:int_besselgamma}, which in turn is a consequence of the finite series expansion of the incomplete Gamma function in \eqref{eq:gamma_sum}.
\end{proof}

\subsubsection[Holomorphic terms]{The contribution of the holomorphic terms}
Let us now calculate $\phi^{\underline{\gamma}, \underline{\delta}}_2(n,\eta)^+$,
the contribution due to the lift of the holomorphic part of the weak harmonic Maass form $f $,
with Fourier expansion $f^+ = \sum_{n \gg -\infty }\hat{a}^+(n) q^n$. 
We have
\[
\begin{multlined}
\phi^{\underline{\gamma}, \underline{\delta}}_2(m,\eta)^+= 
\left(-i  \sqrt{\pi} \right)^{n_\gamma + n_\delta}  
\sum_{\ell = 0}^{r_\gamma + r_\delta}  2^{\frac{\ell}{2}}
P_{\tilde\gamma, \tilde\delta, \ell} (x_{0, +}) \\
\cdot
\CT_{s=0}\int_{0}^{\infty} v^{\frac{ \ell - n_\gamma - n_\delta}{2} -s-2} 
e^{-2\pi \left(v \hlf{x_{0,+}}{x_{0,+}} + \hlf{x_{0,-}}{x_{0,-}}\right)}   e^{2\pi m v} 
\int_{-\infty}^{\infty} \left( \beta' + \bar\tau\beta \right)^{n_\delta} \left(\bar\beta' + \bar\tau\bar\beta \right)^{n_\gamma}  \\
\cdot
\exp\left( -\frac{2\pi}{v}\left( \abs{\beta'}^2 + \abs{\bar\tau}^2\abs{\beta}^2 \right)- 2\pi \Im\left( \beta'\bar\beta \right) + 2\pi i u \hlf{x_0}{x_0} \right)  e^{2\pi  i n u}du\, dv.
\end{multlined}
\]
After using Lemmas \ref{lemma:FT_exp_part} and \ref{lemma:FT_phik1k2} to evaluate the inner integral, the  
integrand of the  outer integral is given by
\[
\frac{2^{-\frac{1}{2} }}{\abs{\beta}}  v^{\frac12 \left( \ell - n_\gamma  - n_\delta\right) -s - \frac{3}{2}}  \tilde{p}_\eta(v,n)  
\exp\left(-\frac{v\pi }{\abs{\beta}^2} \AAb{\eta} - \frac{\pi}{v\abs{\beta}^2} \BB_\eta\right)
e^{-2\pi i \CC_\eta\left( \hlf{x_0}{x_0} + n\right)}, 
\]
with $\tilde{p}_\eta(v,n)$ the polynomial from Lemma \ref{lemma:FT_phik1k2}, with $k_1 = n_\delta$ and $k_2 = n_\gamma$.
Evaluating with Lemma \ref{lemma:bessel_parts}, and using \eqref{eq:bessel_n2}, we get the following result:
\begin{lemma}\label{lemma:pq_phi2h} 
	For fixed $n$ and $\eta$ and at the base point $z_0$, the rank two term $\phi^{\underline{\gamma}, \underline{\delta}}_2(n,\eta)^+$ is given by
	\begin{equation} \label{eq:pq_phi2_hol}
	\begin{aligned} \phi^{\underline{\gamma}, \underline{\delta}}_2& (n,\eta)^+  =  
	\left( -i \sqrt{\pi} \right)^{n_\gamma  + n_\delta} 
	\sum_{\ell = 0}^{2q - 2 - n_\gamma - n_\delta} 2^{\frac{\ell + 1}{2}} 
	P_{\tilde\gamma, \tilde\delta, \ell} (x_{0, +})   
	\sum_{M=0}^{n_\gamma + n_\delta}  R_{n_\delta, n_\gamma}(\eta,M)
	\\ & \;\cdot\, \sum_{j=0}^{\left\lfloor \frac{M}{2} \right\rfloor} \frac{1}{\pi^j} \sum_{h = 0}^{M -2j}
	\frac{i^h}{2^{h + 3j}}  
	\AAm{\eta}^h \left( - \CC_\eta\right)^{M - 2j -h}  \abs{\beta}^{- 2h - 2 j}
	\binom{M -2 j}{h} \frac{M!}{j! (M-2j)!} 
	\\
	& \qquad \qquad \cdot
	\left( \frac{\AAb{\eta}}{\BB_\eta}\right)^{-\frac{\nu}{2}}
	K_\nu\left( \frac{2\pi}{\abs{\beta}^2} \left(\AAb{\eta}\BB_\eta\right)^{\frac12}\right) 
	\exp\left( -2\pi i \CC_\eta\left(\hlf{x_0}{x_0}  + n\right) \right) ,
	\end{aligned}
	\end{equation}
	where $\nu = h + j  + \frac12\left( \ell - n_\gamma - n_\delta\right) - \frac12$.  
	Furthermore, if $\nu \equiv \frac12 \pmod{1}$, the K-Bessel functions in the last line can be replaced by
	\[
	\frac12 \abs{\beta} \left( \AAb{\eta}\BB_\eta\right)^{-\frac14}
	h_{\abs{\nu} -\frac12} \left( \left( 2\pi\left( \AAb{\eta}\BB_\eta\right)^{\frac12} \right) ^{-1}\right) \exp\left( -  \frac{2\pi}{\abs{\beta}^2} \left( \AAb{\eta}\BB_\eta\right)^{\frac12}\right),
	\]
	where $h_n$ denotes the nth Bessel polynomial. 
\end{lemma}

\subsubsection[Non-holomorphic terms]{The contribution of the non-holomorphic part}  \label{par:pq_nhol}

Now, we calculate $\phi^{\underline\gamma, \underline\delta}_2(n, \eta)^{-}$, the contribution
due to the non-holomorphic part $f^-$ of a weak harmonic Maass form, with a Fourier expansion of the form
\[
f^- (\tau) = 
\hat{a}^{-}(0) v^{1- k} + 
\sum_{\substack{n\in\Q \\ n\neq 0}} \hat{a}^{-}(n)\Gamma\left(1 - k, 4\pi \abs{n}v \right) e^{2\pi i n u}.
\]
Let us briefly examine the contribution due to the constant term (i.e.\ $n=0$). 
Using Lemma \ref{lemma:FT_mzero} to evaluate the inner integral over u, we get 
\[
\frac{2^{-\frac{1}{2}}}{\abs{\beta}} 
v^{-\frac12 - k + \frac12(\ell - n_\gamma - n_\delta) -s}\tilde{p}_\eta(v, 0)
\exp\left(-\frac{\pi}{\abs{\beta^2}} \left( v \AAb{\eta}(0) - \frac{1}{v}\BB_\eta  \right)\right)
e^{-2\pi i \CC_\eta\hlf{x_0}{x_0}}
\]
as the integrand of the integral over $v$, which can be evaluated exactly like for the holomorphic terms, with $n=0$ throughout and index shifted by $-(k-1)$.

\subparagraph{Terms with $n\neq 0$:} 
The argumentation is similar, as previously for the holomorphic part. The inner integral,
over $u$, is evaluated exactly as before. The integrand for the integral over $v$ takes the form
\[
\frac{2^{-\frac{ 1}{2} }}{\abs{\beta}}  v^{\frac12 \left( \ell - n_\gamma  - n_\delta\right) -s - \frac{3}{2}}  \tilde{p}_\eta(v,n)  \Gamma\left(k + 1,  4\pi \abs{n}v\right)
e^{2\pi n v}
e^{-\frac{v\pi }{\abs{\beta}^2} \AAb{\eta} - \frac{\pi}{v\abs{\beta}^2} \BB_\eta}
e^{-2\pi i \CC_\eta\left( \hlf{x_0}{x_0} + n\right)},
\]
with $\tilde{p}_\eta(v,n)$ from Lemma \ref{lemma:FT_phik1k2}.
The integral is now evaluated using Lemma \ref{lemma:int_gammabess_part}, yielding
\begin{lemma}\label{lemma:pq_phi2nh}
	Let $\kappa = p+q-2$.  For fixed $n$ and $\eta$ and at the base point $z_0$, 
	the contribution of the  rank two orbit to the lift of the non-holomorphic part $f^-$ of $f$
	is given by
	\begin{equation} \label{eq:pq_phi2_nhol}
	\begin{aligned} \phi^{\underline{\gamma}, \underline{\delta}}_2& (m,\eta)^{-}  = 
	\left( -i \sqrt{\pi} \right)^{n_\gamma  + n_\delta} 
	\sum_{\ell = 0}^{2q - 2 - n_\gamma - n_\delta} 2^{\frac{\ell + 1}{2}} 
	P_{\tilde\gamma, \tilde\delta, \ell} (x_{0, +})   
	\sum_{M=0}^{n_\gamma + n_\delta} R_{n_\delta, n_\gamma}(\eta,M)
	\\ & \;\cdot\, \sum_{j=0}^{\left\lfloor \frac{M}{2} \right\rfloor} \frac{1}{\pi^j}
	\sum_{h = 0}^{M -2j} \frac{i^h}{2^{h + 3j}}
	\AAm{\eta}^h \left( - \CC_\eta\right)^{M - 2j -h}  \abs{\beta}^{- 2h - 2 j}
	\binom{M -2 j}{h} \frac{M!}{j! (M-2j)!} 
	\\
	& \cdot \Vint{\kappa +2, \frac{3}{2} - \nu}{\pi\left( \frac{\AAb{\eta}}{\abs{\beta}^{2}}  - 2n\right), \frac{\pi\BB_\eta}{ \abs{\beta}^{2}}, 4\pi\abs{n}}  
	\exp\left( -2\pi i \CC_\eta\left(\hlf{x_0}{x_0}  + n\right) \right).
	\end{aligned}
	\end{equation}       
	where $\nu = h + j  + \frac12\left( \ell - n_\gamma - n_\delta\right) - \frac12$.
\end{lemma}

We recall that by the definition of the special function $\Vint{n,\nu}{A,B,c}$ we have 
\begin{equation*} 
\begin{multlined}
\Vint{\kappa +2, \frac{3}{2} - \nu}{\pi\left( \frac{\AAb{\eta}}{\abs{\beta}^{2}}  - 2n\right), \frac{\pi\BB_\eta}{ \abs{\beta}^{2}}, 4\pi\abs{n}}  =  
\kappa!
\sum_{r=0}^{\kappa}
\frac{\left(4\pi\abs{n}\right)^r}{r!}  \\ \cdot
\left( \frac{\AAb{\eta} + 4\abs{n}\abs{\beta}^2 - 2 n\abs{\beta}^2}
{\BB_\eta}\right)^{-\frac{\nu}{2} - \frac{r}{2}} 
K_{r+\nu}\left( \frac{2\pi}{\abs{\beta}^2} 
\left(\left(\AAb{\eta} + 4\abs{n}\abs{\beta}^2  - 2n \abs{\beta}^2\right) \BB_\eta\right)^{\frac12}\right). 
\end{multlined}
\end{equation*}
Further, in  every term where $\nu + r$ is a half-integer, one can set  $\nu' = \abs{r + \nu} - \frac12$
and replace $K_{r + \nu}$ with
\[
\begin{gathered}
\frac12 \abs{\beta}
\left( \left(\AAb{\eta} + 4\abs{n}\abs{\beta}^2  - 2n \abs{\beta}^2\right)\BB_\eta\right)^{-\frac14} 
h_{\nu'} \left( \left( 2\pi\left(
\left(\AAb{\eta} + 4\abs{n}\abs{\beta}^2  - 2n \abs{\beta}^2\right)\BB_\eta\right)^{\frac12} \right) ^{-1}
\right) \\ \cdot\,
\exp\left(- \frac{2\pi}{\abs{\beta}^2} 
\left( \left(\AAb{\eta} + 4\abs{n}\abs{\beta}^2  - 2n \abs{\beta}^2\right) \BB_\eta\right)^{\frac12}\right).
\end{gathered}
\]

\subsection{Rank one terms} \label{subsec:unfold1}
Now, let consider the case where $\eta = [\beta, \beta']$ is of rank 1. 
Recall from Lemma \ref{lemma:Ks_reprs} that a set of representatives for the orbit under $\SL_2(\Z)$
operation is given by
\[
[ 0, \beta'] = \begin{pmatrix} 0 & a \\ 0 & b\end{pmatrix}, \quad \text{with either $\alpha>0$ or $a=0$ and $b>0$.} 
\]
Further, the stabilizer in this case is $\SL_2(\Z)_\infty$. So the domain of integration is given by
\[
\Gamma'_\infty \backslash \Gamma' = \left\{ u + iv\,;\; 0 < u < 1, 0 < v < \infty \right\}.
\]
Again, we begin with the contribution of the holomorphic part. 

\subsubsection{Holomorphic part} 
For a fixed pair of multi-indices $\underline{\gamma}$, $\underline{\delta}$, and fixed $\eta = [0,\beta']$,
the rank one term $\phi_{1}^{\underline\gamma, \underline\delta}(n,\eta)^{+}$ is  given by the integral 
\[
\begin{multlined}
\phi_{1}^{\underline\gamma, \underline\delta}(n,\eta)^{+}
= (-i)^{n_\gamma + n_\delta} {\pi}^{\frac{n_\gamma + n_\delta}{2}} \beta'^{n_\delta}\bar\beta'^{n_\gamma}
\sum_{\ell = 0}^{2q -2 - n_\delta -n_\gamma}2^{\frac{\ell}{2}} P_{\tilde\gamma, \tilde\delta,\ell} (x_{0,+}) \\ 
\cdot
\CT_{s=0} \int_{\R_{>0}} \int_{0}^1 v^{\frac12( \ell - n_\gamma - n_\delta) - 2 -s}
e^{-\frac{2\pi}{v} \abs{\beta'}^2} e^{2\pi i u\hlf{x_0}{x_0} -
	2\pi v\left(  \Qf{x_{0, +}} -  \Qf{x_{0,-}} \right)} e^{2\pi i n \tau} du\, dv.
\end{multlined}
\]
The integral over $u$ just picks out the constant term. 
Hence, $n = - \hlf{x_0}{x_0}$ for all non-vanishing contributions. 
Now, $ \hlf{x_{0,-}}{x_{0,-}}  -  \hlf{x_{0,+}}{x_{0,+}} - n = 2\hlf{x_{0,-}}{x_{0,-}}$. Since the norm of $x_{0,-}$ is negative, the  integrals over $v$ take the form
\begin{equation}\label{eq:rk1_pq_intv}
\CT_{s=0} \int_{0}^\infty v^{\frac12 ( \ell - n_\gamma - n_\delta) - 2 -s}
\exp\left( -4\pi v \abs{\hlf{x_{0,-}}{x_{0,-}} } - \frac{2\pi}{v} \abs{\beta'}^2\right) dv.
\end{equation}
If $\Qf{x_{0,-}} \neq 0$, by the integral representation of the Bessel-functions \eqref{eq:int_bessel},
setting $\nu = \frac12( \ell - n_\gamma - n_\delta) - 1$  and evaluating at $s=0$, one obtains
\[
2 \abs{\beta'}^\nu \left( 2\abs{\hlf{x_{0,-}}{x_{0,-}} } \right)^{-\frac{\nu}{2}}
K_{\nu}\left( 2\pi \abs{\beta'} \abs{ \sqrt{2\hlf{x_{0,-}}{x_{0,-}}} } \right).
\]
Further, if $\nu$ is a half-integer, by \eqref{eq:bessel_n2},  this equals
\[
\abs{\beta'}^{\nu-\frac12} 
\left(2\abs{ \hlf{x_{0,-}}{x_{0,-}} } \right)^{-\frac{\nu}{2}-\frac14}
h_{\abs{\nu}-\frac12} \left(\left( 2\sqrt{2}\pi \abs{\beta'}
\abs{{\hlf{x_{0,-}}{x_{0,-}}} }^{\frac12}\right)^{-1}\right)
e^{-2\sqrt{2} \pi  \abs{\beta'}\abs{ \hlf{x_{0,-}}{x_{0,-}} }^{\frac12} },
\]
with the Bessel polynomial $h_{\nu'}$, $\nu' = \abs{\nu}-\frac12$. 
If $\Qf{x_{0, -}} = 0$ the integral \eqref{eq:rk1_pq_intv} can be evaluated using the integral representation of $\Gamma$-functions, see Example \ref{ex:sig_p1}.
\begin{lemma}\label{lemma:pq_phi1h}
	For fixed $n$ and fixed $\eta = \left[0, \beta' \right]$,
	at the base point $z_0$, and assuming that $\Qf{x_{0,-}} \neq 0$, the rank  one term is given by
	\begin{multline*}
	\phi_{1}^{\underline\gamma, \underline\delta}   (n,\eta)^{+}(z_0)  = 
	\left(-i \sqrt{\pi}\right)^{n_\gamma + n_\delta} \beta'^{n\gamma}\bar\beta'^{n_\delta}
	\sum_{\ell = 0}^{2q -2 -n_\gamma  - n_\delta} 2^{\frac{\ell}{2}}
	P_{\tilde\gamma, \tilde\delta, \ell}(x_{0,+})   
	\abs{\beta'}^{\nu} 2^{\frac{\nu}{2}} \abs{\hlf{x_{0,-}}{x_{0,-}} }^{\frac{\nu}{2}} \\
	\cdot    
	\begin{cases} 
	2 K_\nu\left( 2\sqrt{2} \pi \abs{\beta'} \abs{\hlf{x_{0,-}}{x_{0,-}} }^{\frac12} \right), \\
	\qquad \text{if}\;  \nu = \tfrac12\left( \ell -  n_\gamma - n_\delta\right) -1 \equiv 0 \pmod{1} \\
	\left({2\hlf{x_{0,-}}{x_{0,-}} } \abs{\beta'}^2\right)^{-\frac{1}{4}} 
	h_{\nu'}\biggl(   \frac{1}{2\pi\abs{\beta'}} 
	\abs{2\hlf{x_{0,-}}{x_{0,-}}}^{- \frac12} \biggr)
	\exp\left( - 2 \sqrt{2} \abs{\beta'}\pi  \abs{\hlf{x_{0,-}}{x_{0,-}}}^{\frac12} \right) \\
	\qquad \text{with $\nu' = \abs{\nu}-\tfrac12$, if $\nu \equiv \frac12 \pmod{1}$ .}
	\end{cases}
	\end{multline*}
	Note that $\nu$ ranges from $- \frac{n_\gamma + n_\delta}{2} -1$ to 
	$q - 2 - \frac{n_\gamma + n_\delta}{2}$. 
\end{lemma}

\subsubsection{Non-holomorphic part} 
For $n\neq 0$ the inner integral over $u$ picks out the constant term, similar to the evaluation of the holomorphic part.
Then, the outer integral is of the form
\[
\CT_{s=0}  \int_{0}^\infty v^{\frac12\left( \ell -n_\gamma - n_\delta \right)-2 -s}\Gamma(\kappa+1, 4\pi\abs{n})  e^{-\frac{\pi}{v}\abs{ \beta'}^2 } e^{2\pi v \left( \hlf{x_{0,-}}{x_{0,-}}  - \hlf{x_{0,+}}{x_{0,+}}\right) } dv,
\]
with $\kappa  =  (p+q) - 2$. Hence, using Lemma \ref{lemma:int_gammabess_part},  we get the following.
\begin{lemma}\label{lemma:pq_phi1nh}
	For fixed $n\neq 0$ and fixed $\eta = \left[0, \beta' \right]$, at the base point $z_0$,
	the rank  one term is given by
	\begin{gather*}
	\begin{aligned}
	\phi_{1}^{\underline\gamma, \underline\delta} & (n,\eta)^{-}(z_0)
	= \left(-i \sqrt{\pi}\right)^{n_\gamma + n_\delta} \beta'^{n\gamma}\bar\beta'^{n_\delta}
	\sum_{\ell = 0}^{2q -2 - n_\gamma - n_\delta} 2^{\frac{\ell}{2} + 1} 
	P_{\tilde\gamma, \tilde\delta, \ell}(x_{0,+})  \\
	& \cdot \kappa! \sum_{r=0}^\kappa \frac{\left( 4\pi\abs{n}\right)^r}{r!} 
	\left( \frac{2\hlf{x_{0,-}}{x_{0,-}} + 4\abs{n}}{\abs{\beta'}^2}\right)^{-\frac{\nu+r}{2}} 
	K_{\nu+r}\left( 2\pi \abs{\beta'} \left(2  {\hlf{x_{0,-}}{x_{0,-}}}+ 4\abs{n} \right)^{\frac12} \right) \\
	\end{aligned} \\
	= 
	\left(-i \pi^{\frac12}\right)^{n_\gamma + n_\delta}\!  \beta'^{n\gamma}\bar\beta'^{n_\delta}
	\sum_{\ell = 0}^{\substack{2q -2 \phantom{-n}\\   - n_\gamma - n_\delta}} 
    2^{\frac{\ell}{2}} 
	P_{\tilde\gamma, \tilde\delta, \ell}(x_{0,+})\cdot
	\Vint{\kappa +2, 1 -\nu}{2\pi \hlf{x_{0,-}}{x_{0,-}}, \pi\abs{\beta'}^2, 4\pi \abs{n}},
	\end{gather*}
	where, as usual, $\kappa = p+q -2$ and $\nu = \frac12 (\ell -n_\gamma -n_\delta) -1$.
	Note that if $\nu + r$ is a half-integer, the Bessel functions in the second line can be replaced by
	\[
	\begin{gathered}
	\frac1{2^{\frac{3}{2}}\left(\abs{\beta'}\right)^{\frac12}} 
	\left(2  {\hlf{x_{0,-}}{x_{0,-}}} + 4\abs{n} \right)^{-\frac14} 
	h_{\nu'}\left(\left[  2^{\frac{3}{2}}\pi \abs{\beta'} 
	\sqrt{  {\hlf{x_{0,-}}{x_{0,-}}}+ 2\abs{n} }\right] ^{-1}\right) \\
	\cdot   \exp\left(- 2^{\frac32} \pi \abs{\beta} \sqrt{\hlf{x_{0,-}}{x_{0,-}}  + 2\abs{n}}\right)   
	\end{gathered}
	\]
	wherein $h_{\nu'}$ is the Bessel polynomial of index $\nu' = \abs{\nu + r} - \frac12$.
\end{lemma}
For $n = 0$ the contribution of the non-holomorphic part 
$\phi_1^{\underline\gamma, \underline\delta}(0,\eta)^{-}$ is similar to the holomorphic part $\phi_1^{\underline\gamma, \underline\delta}(0,\eta)^{+}$ (see Lemma \ref{lemma:pq_phi1h}), but with index shifted by $-k+1$ due to the power of $v$ in the constant term of $f^-$.

\appendix

\section{Special functions and Fourier transforms}\label{sec:functions}
\subsection{Special functions} 
In this section we recall the integral representations of some special functions and their properties.
\paragraph{The incomplete Gamma function}
First, for the convenience of the reader, 
we recall the integral representations of the Gamma function and the incomplete Gamma function.
\begin{equation}\label{eq:int_gamma}
\Gamma(s) = \int_0^\infty t^{s-1} e^{-t} dt, \qquad
\Gamma(s,a) = 
\int_a^\infty t^{s-1} e^{-t} dt.
\end{equation}
For $n \in \mathbb{N}_0$, we note the following identity  \citep[cf.][p.~74]{Br02}:
\begin{equation}\label{eq:gamma_sum}
\Gamma(n+1,  a) = n! e^{-a} e_n(a)  = n! e^{-a}\sum_{r=0}^{n} \frac{a^r}{r!}.
\end{equation}

\paragraph{Relations for the $K$-Bessel functions}
The following integral representation for Bessel functions is well-known to number theorists,
 see \cite[][6.(17),  p. 313]{EMOT54}
\begin{equation}
  \label{eq:int_bessel}
  \int_0^\infty v^{\nu -1} \exp\left(-a v -b v^{-1} \right) dv =
  2 \left(\frac{a}{b} \right)^{-\frac{\nu}{2}}K_\nu\left(2\sqrt{ab}\right)
  \qquad(\Re a>0, \Re b>0).
\end{equation}
Beside this integral representation, we also make frequent use of the following relations \citep[cf.][10.27.3, 10.33.2]{DLMF}:
\begin{gather}
  K_{-\nu} (x)= K_{\nu}(x),  \quad  \label{eq:bessel_12}
  K_{-\frac12}(2\pi r) = K_{\frac12}(2\pi r)= \frac12 r^{-\frac12} e^{-2\pi r}, \\ 
  \label{eq:bessel_n2} \text{and}\qquad
    K_{n + \frac12}(2\pi r) = \frac12 r^{-\frac12} e^{-2\pi r} 
    \besspoly_n\left( \tfrac{1}{2\pi r}\right) \quad (n\in\Z, n\geq 0).  
\end{gather}
Here, $\besspoly_n$ is the n-th Bessel polynomial, explicitly given by
\begin{equation*}
\besspoly_n(x) = \sum_{k=0}^n \frac{(n+k)!}{(n-k)!\, k!}\left( \frac{x}{2} \right)^k.
\end{equation*}

A further special function, which generalises \citep[][(3.25) on p.74]{Br02} is useful for the Fourier-Jacobi expansion of the singular theta lift $\Phi(z,f,\psi_{p,q})$.
\begin{lemma}\label{lemma:int_besselgamma}
  For $n \in \mathbb{Z}$, $n \geq 2$,  $\Re (A+ c)>0$, $\Re B>0$, the special function defined as  
  \[
 \Vint{n, \mu}{A, B, c} \vcentcolon =   \int_0^\infty \Gamma(n - 1, c v) v^{-\mu} e^{-Av - B \frac{1}{v}} dv. 
 \]
is given by
 \[
   2(n-2)!  \sum_{r=0}^{n-2} \frac{c^r}{r!} \left(\frac{A + c}{B}\right)^{\frac{\mu - r -1 }{2}}\!
    K_{r+1-\mu}(2\sqrt{(A+c)B}) .
  \]
  Further, if $\mu \equiv \frac12 \pmod{1}$, we have
  \[
    (n-2)!\pi^{\frac12}  \sum_{r=0}^{n-2}\frac{c^ r}{r!} \left( A  +c\right)^{\frac{1}{2}(\mu -r) - \frac{3}{4}} B^{\frac12(r - \mu) + \frac{1}{4}}  
  e^{-2\sqrt{(A+c)B}} h_{r- \mu + \frac12}\left( \frac{\pi}{2\sqrt{(A+c)B}} \right).  
  \]
   \end{lemma}
 \begin{proof}
   Since $n-2$ is a non-negative integer, we can use the formula \eqref{eq:gamma_sum},
    and by \eqref{eq:int_bessel},  obtain the following: 
   \[
     \begin{gathered}
        \int_0^\infty \Gamma(n - 1, c v) v^{-\mu} e^{-Av -Bv^{-1}} dv  = \\
        (n-2)! \sum_{r=0}^{n-2} \frac{c^r}{r!} \int_0^\infty v^{r - \mu} e^{-cv-Av - B v^{-1}} dv  = \\
        (n-2)! \sum_{r=0}^{n-2} \frac{c^r}{r!}\cdot 2 \left( \frac{A+c}{B}\right)^{-\frac{r+1-\mu}{2}} K_{r-\mu + 1}(2\sqrt{(A+c)B}).
      \end{gathered}
 \]
      The rest follows directly from \eqref{eq:bessel_12} and \eqref{eq:bessel_n2}. 
  \end{proof}

\paragraph{Special polynomials}\label{par:sp_poly}
We now turn to two families of special polynomials, the Hermite and the Laguerre polynomials. Recall their definitions: 
For $k\geq 0$, the $k$-th Hermite polynomial $H_k$ is given by
\[
  H_k(t) = (-1)^k e^{t^2}\left(\frac{d}{dt} \right)^k e^{-t^2} =
  e^{t^2/2} \left( t - \frac{d}{dt} \right)^k e^{-t^2/2},
\]
while the $k$-th Laguerre polynomial $L_k$ is given by
\[
  L_k(t) = \frac{e^t}{k!}\left( \frac{d}{dt}\right)^k \left( e^{-t} t^k\right).
\]
A slight reformulation using the differentials from Section \ref{sec:psi_to_mm}, will be quite useful.
  \begin{equation}\label{eq:KM_Laguerre}
    \Da_1^k \bar{\Da}_1^k \varphi_0 = \left(\frac{1}{\pi} \right)^k 2^k k! L_k\left(2\pi\abs{z_1}^2 \right)\varphi_0.
  \end{equation}
  Further, since the operators $\Da_1$ and $\bar\Da_1$ commute, one obtains the following equality: 
  \begin{equation}\label{eq:Laguerre_binom_Hermite}
    \begin{gathered}
      \Da_1^k \bar{\Da}_1^k \varphi_0  = \left( \Da_1\bar{\Da}_1\right)^k \varphi_0 
      = \left[ \left( x_1 - \frac{1}{2\pi} \frac{d}{d x_1}\right)^2 + \left(  x_2 - \frac{1}{2\pi} \frac{d}{d x_2}\right)^2 \right] \varphi_0 \\
      = (2\pi)^{-k} \sum_{l=0}^k \binom{k}{l} H_{2(k-l)}\left(\sqrt{2\pi} x_1 \right)H_{2l}\left(\sqrt{2\pi} x_2\right) \varphi_0.
    \end{gathered}
  \end{equation}

\subsection{Fourier transforms}
Now, we gather some formulas for Fourier transforms, which come in useful for the  evaluation of theta integral and for  switching between the Schrödinger model and the mixed model of the Weil representation (cf.\ Section \ref{sec:psi_to_mm}). 

For the following Lemma, see \citep[][Corollary 3.3]{Bo98}. 
\begin{lemma}\label{lemma:FT_Bo}
	Let $x$ be a real indeterminate and $p(x) \in \C[x]$ a polynomial. Then, the Fourier transform of the Schwartz function 
	\[
	p(x)e^{2\pi i \left( Ax^2 + Bx + C\right)}
	\]
	is given by
	\begin{equation}\label{eq:FT_Bo}
	(-2iA)^{-\frac12}\exp\left( \frac{i}{8\pi 
		A}\frac{d^2}{dt^2}\right)(p(t))\left(-\tfrac{\xi}{2A} - \tfrac{B}{2A}\right)
	\exp\left(2\pi i\left[-\tfrac{\xi^2}{4A} - \tfrac{\xi\,B}{2A}  
	-\tfrac{B^2}{4A} + C\right]\right),
	\end{equation}
	wherein $\xi$ denotes the transformed variable. 
\end{lemma}
Now, consider a special case, the Fourier transform of $p(x)e^{-\pi x^2}$. It is given by 
\begin{equation} \label{eq:FT_poly_real}
\exp\left( \frac{1}{4\pi} \frac{d^2}{d t^2}\right) \bigl(p(t)\bigr)(i\xi)e^{-\pi\xi^2}  \vcentcolon 
= \tilde{p}(\xi)e^{-\pi x^{2}}.
\end{equation}
From this, one can immediately conclude the following statements  
\begin{enumerate}
\item The Fourier transform of  $p(x + c)e^{-\pi x^2}$ is given by $\tilde{p}(\xi - ic)e^{-\pi \xi^2}$.
  \item The transform of $p(-x)e^{-\pi x^2}$ is given by $\tilde{p}(\xi)e^{-\pi x^2}$. 
\end{enumerate}

\paragraph{Fourier transform of special polynomials}
The following Lemma from \citep[][Lemma 4.1]{FM13}, describes the Fourier transform of the Hermite polynomials  
\begin{lemma}\label{lemma:ft_hermite}
  The Fourier transform of the $k$-th Hermite polynomial $H_k(x)$ is given by
   \begin{equation}\label{eq:FM_Hk}
    \int_{0}^\infty \frac{1}{\sqrt{2\pi}^k}H_k(-\sqrt{\pi}x)e^{-\pi x^2}\, e^{2\pi i \xi x} dx = 
    \left( - \sqrt{\pi}i\xi\right)^{k} e^{-\pi\xi^2}.
\end{equation}
\end{lemma}
For the Laguerre polynomials, we have the following result:
\begin{lemma}\label{lemma:ft_laguerre}
  Let $z = x +iy$ be a complex variable. The Fourier transform in $z$ of the Laguerre polynomial 
  \begin{equation}
    L_{k}\left(\pi\abs{z}^2\right)e^{-\pi\abs{z}^2}\quad\text{is given by}
   \quad \frac{ \pi^{2k}}{2^{k} k!}\abs{w}^{2k} 
 e^{-\pi\abs{w}^2}, 
  \end{equation}
with the transformed variable $w$. 
\end{lemma}
\begin{proof}
  Recall from \eqref{eq:KM_Laguerre} that
    $\left(\Da_\alpha\Da_\alpha\right)^k \varphi_0 = L_k(2\pi\abs{z_\alpha})$. Thus, the claim follows directly from Lemma \ref{lemma:ft_hermite} via \eqref{eq:Laguerre_binom_Hermite}.
 Indeed, since
  \[
    \left(\frac{-1}{\pi}\right)^k 2^k k! L_{k}\left(\pi\abs{z}^2\right)e^{-\pi\abs{z}^2} = 
    \left(\frac{1}{2\pi}\right)^k \sum_{j = 0}^k \binom{k}{j}
    H_{2(k-j)}\left(\sqrt{\pi}x\right) 
    H_{2j} \left( \sqrt{\pi}y\right)e^{-\pi\left(x^2 + y^2 \right)}, 
    \]
    after applying \eqref{eq:FM_Hk} twice and writing $w = \xi + i\eta$, we get the Fourier transform
    \[
     (-1)^k \pi^k \sum_{j=0}^k
 \binom{k}{j}
    \xi^{2(k-j)} \eta^{2j} = (-1)^k \pi^k \abs{w}^{2k},
    \]
    as claimed.
  \end{proof}

\bibliographystyle{myplainnat}
\bibliography{habil_bib}
\end{document}